\documentclass[fleqn,reqno,11pt,a4paper,final]{amsart}
\usepackage[a4paper,left=20mm,right=20mm,top=20mm,bottom=20mm,marginpar=20mm]{geometry}
\usepackage{amsaddr}
\usepackage{amsmath}
\usepackage{amssymb}
\usepackage{amsthm}
\usepackage{mathtools}
\usepackage[ansinew]{inputenc}
\usepackage{cite}
\usepackage{bbm}
\usepackage{multirow}
\usepackage{xcolor}
\usepackage[english=american]{csquotes}
\usepackage[final]{graphicx}
\usepackage{hyperref}
\usepackage{calc}
\usepackage{tikz}
\usepackage{pgfplots}
\usepackage{graphicx}
\usepackage{stix}
\usepackage{soul}
\usepackage{enumitem}
\usepackage{multirow}
\usepackage[new]{old-arrows}
\usepackage{soul}
\makeatletter
\def\namedlabel#1#2{\begingroup
    #2%
    \def\@currentlabel{#2}%
    \phantomsection\label{#1}\endgroup
}
\makeatother

\graphicspath{{../Pictures/}}

\numberwithin{equation}{section}

\newtheoremstyle{thmlemcorr}{10pt}{10pt}{\itshape}{}{\bfseries}{.}{10pt}{{\thmname{#1}\thmnumber{ #2}\thmnote{ (#3)}}}
\newtheoremstyle{thmlemcorr*}{10pt}{10pt}{\itshape}{}{\bfseries}{.}\newline{{\thmname{#1}\thmnumber{ #2}\thmnote{ (#3)}}}
\newtheoremstyle{remexample}{10pt}{10pt}{}{}{\bfseries}{.}{10pt}{{\thmname{#1}\thmnumber{ #2}\thmnote{ (#3)}}}
\newtheoremstyle{ass}{10pt}{10pt}{}{}{\bfseries}{.}{10pt}{{\thmname{#1}\thmnumber{ A#2}\thmnote{ (#3)}}}

\theoremstyle{thmlemcorr}
\newtheorem{theorem}{Theorem}
\numberwithin{theorem}{section}
\newtheorem{lemma}[theorem]{Lemma}
\newtheorem{corollary}[theorem]{Corollary}
\newtheorem{proposition}[theorem]{Proposition}

\newtheorem{definition}[theorem]{Definition}

\newtheorem{thmintro}{Theorem}

\theoremstyle{thmlemcorr*}
\newtheorem{theorem*}{Theorem}
\newtheorem{lemma*}[theorem]{Lemma}
\newtheorem{corollary*}[theorem]{Corollary}
\newtheorem{proposition*}[theorem]{Proposition}
\newtheorem{problem*}[theorem]{Problem}
\newtheorem{conjecture*}[theorem]{Conjecture}
\newtheorem{definition*}[theorem]{Definition}
\newtheorem{assumption*}[theorem]{Assumption}

\theoremstyle{remexample}
\newtheorem{remark}[theorem]{Remark}
\newtheorem{example}[theorem]{Example}

\newtheorem*{remark*}{Remark}

\theoremstyle{ass}

\newcommand{\B}{\mathcal{B}}

\newcommand{\T}{\mathbb{T}}

\newcommand{\bad}{\mathscr{B}}
\newcommand{\good}{\mathscr{G}}

\newcommand{\tigma}{\tilde{\sigma}}

\newcommand{\Qu}{\mathscr{Q}}

\definecolor{Gump}{rgb}{0,0.6,0.4}
\definecolor{Hanks}{rgb}{0.7,0.3,0.1}
\newcommand{\veta}{a_{\eta}}

\DeclareMathOperator{\id}{id}

\DeclareMathOperator*{\wlim}{w-lim}

\DeclareMathOperator{\diverg}{div}

\DeclareMathOperator{\curl}{curl}

\newcommand{\norm}[1]{\|#1\|}

\newcommand{\abs}[1]{|#1|}

\newcommand{\dd}{\;\mathrm{d}}

\newcommand{\N}{\mathbb{N}}
\newcommand{\R}{\mathbb{R}}

\newcommand{\Z}{\mathbb{Z}}

\newcommand{\loc}{\mathrm{loc}}
\newcommand{\sym}{\mathrm{sym}}
\newcommand{\skw}{\mathrm{skew}}

\newcommand{\equi}{\mathrm{eq}}
\newcommand{\conc}{\mathrm{co}}
\newcommand{\M}{\mathcal{M}}
\newcommand{\pa}{\mathrm{pa}}
\newcommand{\el}{\mathrm{ell}}

\newcommand{\eps}{\varepsilon}

\newcommand\bra[1]{\left(#1\right)}
\newcommand{\weakto}{\rightharpoonup}
\newcommand{\lweakto}{\longrightharpoonup}
\newcommand{\expo}{{\tilde{p}}}
\newcommand{\tis}{{\tilde{s}}}

\def\XXint#1#2#3{{\setbox0=\hbox{$#1{#2#3}{\int}$}
\vcenter{\hbox{$#2#3$}}\kern-.5\wd0}}

\definecolor{luh-dark-blue}{rgb}{0, 0.313, 0.608}
\definecolor{luh-light-blue}{rgb}{0.6, 0.725, 0.847}
\definecolor{luh-green}{rgb}{0.784, 0.827, 0.09}

\newcommand{\rchange}[1]{\textcolor{black}{#1}}
\newcommand{\schange}[1]{\textcolor{black}{#1}}
\renewcommand{\st}[1]{}

\newcommand{\limit}{\gimel}
\newcommand{\varg}{\mathscr{g}}
\newcommand{\varh}{\mathscr{h}}
\renewcommand{\phi}{\varphi}

\newcommand{\barg}{\bar{g}}
\newcommand{\barh}{\bar{h}}
\newcommand{\barvarg}{\bar{\varg}}
\newcommand{\barvarh}{\bar{\varh}}
\begin{document}


\title[A variational approach to shear-dependent Navier--Stokes]{A variational approach to the Navier--Stokes equations with shear-dependent viscosity}

\author{Christina Lienstromberg, Stefan Schiffer$^\heartsuit$, Richard Schubert$^\ast$}
\address{Institute of Analysis, Dynamics and Modeling, University of Stuttgart, Pfaffenwaldring~57, 70569 Stuttgart, Germany}
\email{christina.lienstromberg@iadm.uni-stuttgart.de}
\email{schiffer@mis.mpg.de}\address{$^{\heartsuit}$ Max-Planck-Institute for Mathematics in the Sciences, Inselstra{\ss}e 22, 04103 Leipzig, Germany}
%
\email{schubert@iam.uni-bonn.de}\address{$^\ast$ Institute for Applied Mathematics, University of Bonn, Endenicher Allee~60, 53115 Bonn, Germany}

\begin{abstract}
 We present a variational approach 
for the construction of
Leray--Hopf solutions to the non-Newtonian Navier--Stokes system. Inspired by the work \cite{OSS} on the corresponding Newtonian problem, we minimise certain stabilised Weighted Inertia-Dissipation-Energy (WIDE) functionals and pass to the limit of a vanishing parameter in order to recover a Leray--Hopf solution of the non-Newtonian Navier--Stokes equations. The investigation of the non-Newtonian Navier--Stokes system via this variational approach is particularly well suited to gain insights into weak, respectively strong convergence properties \schange{of approximating sequences} for different flow-behaviour exponents. With this analysis we extend the results of \cite{BS22} to power-law exponents $\tfrac{2d}{d+2} < p < \tfrac{3d+2}{d+2}$, where weak solutions do not satisfy the energy equality and the involved convergence is genuinely weak. Key of the argument is to pass to the limit in the nonlinear viscosity term in the time-dependent setting. For this we provide an elliptic-parabolic solenoidal Lipschitz truncation that might be of independent interest.
\end{abstract}



\maketitle
\bigskip


\bigskip

\noindent\textsc{Acknowledgement. } We are very grateful to Michael Ortiz for his continued interest in the project and many insightful discussions. 

Funded by Deutsche Forschungsgemeinschaft (DFG, German Research Foundation) under Germany's Excellence Strategy -- EXC 2075 - 390740016. We acknowledge the support by the Stuttgart Center for Simulation Science (SimTech). Moreover, S. Schiffer and R. Schubert are grateful towards the University of Stuttgart for the kind hospitality.

\section{Introduction}

\noindent\textbf{Aim of the paper. } 
In this paper we study the non-Newtonian Navier--Stokes system
\begin{equation} \label{eq:PDE_intro1}
\begin{cases}
    \partial_t u + (u\cdot \nabla) u = -\nabla \pi + \diverg\bigl( 2\mu(|\epsilon(u)|) \epsilon(u)\bigr),
    & 
    t > 0,\, x \in \T_d
    \\
    \diverg u = 0,
    &
    t > 0,\, x \in \T_d
       \\
    u(0,x) = u^0(x), 
    & x \in \T_d,
\end{cases}
\end{equation}
describing the flow of an incompressible viscous non-Newtonian fluid with shear-dependent viscosity on the $d$-dimensional torus $\T_d$, $d \geq 2$.
In system \eqref{eq:PDE_intro1} we use the following notation:
\begin{itemize}
    \item $u \colon (0,\infty) \times \T_d \to \R^d$ denotes the \emph{velocity field};
    \item $\epsilon = \epsilon(u)= \tfrac{1}{2} (\nabla u + \nabla u^T)$ denotes the \emph{rate-of-strain};
    \item $\pi \colon (0,\infty) \times \T_d \to \R$ denotes the \emph{pressure};
    \item the function $\mu \colon [0,\infty) \to \R_+$ is the strain-dependent \emph{viscosity} of the fluid.
\end{itemize}
For a concise presentation we will restrict the mathematical analysis to the torus, which might be seen as the cube with periodic boundary conditions. We remark that this is assumed for technical reasons and most of the results should still hold true for Lipschitz domains $\Omega$ and reasonable (e.g. Dirichlet) boundary data.

The aim of this article is to study (Leray--Hopf) solutions to the system \eqref{eq:PDE_intro1} via a variational approach. More precisely, instead of directly solving the PDE, we minimise a family of space-time functionals and prove convergence of minimisers to \emph{Leray--Hopf solutions} of the non-Newtonian Navier--Stokes system \eqref{eq:PDE_intro1}. Our approach is based on the minimisation of so-called WIDE-functionals (WIDE = weighted inertia dissipation energy), cf. \cite{MO08,OSS} \rchange{and the overview \cite{Stefanelli}}, where \cite{OSS} addresses the Newtonian Navier--Stokes system with $\mu \equiv \text{const.} > 0$.
WIDE-functionals in the setting of shear-dependent viscosity have previously been examined in \cite{BS22} for strongly shear-thickening fluids with a focus on outflow boundary conditions.

In the Newtonian case, the Navier--Stokes equations read
\begin{equation} \label{eq:PDE_intro}
\begin{cases}
    \partial_t u + (u\cdot \nabla) u = -\nabla \pi + \mu \Delta u,
    & t > 0,\, x \in \T_d
    \\
    \diverg u = 0,
    & t > 0,\, x \in \T_d.
\end{cases}
\end{equation}
Weak solutions to \eqref{eq:PDE_intro} are called \emph{Leray--Hopf solutions} if they additionally satisfy the global \emph{energy inequality} 
\begin{equation} \label{intro:energyineq}
     E[u](t) + \int_0^t \int_{\T_d} \mu \vert \nabla u(s,x) \vert^2 \dd x \dd s \leq E[u](0)
    ,\quad \text{where} \quad E[u](t) = \tfrac{1}{2} \int_{\T_d} 
 \vert u(t,x) \vert^2 \dd x.
\end{equation}
 Recall that regular solutions to \eqref{eq:PDE_intro} satisfy \eqref{intro:energyineq} with equality.\\


\bigskip
\noindent\textbf{Non-Newtonian fluids. } Although many liquids and gases, such as water and air, may reasonably be considered Newtonian, many real fluids are in fact non-Newtonian. Newtonian fluids are characterised by a constant viscosity $\mu \equiv \text{const.} > 0$ and thus feature Newton's law 
\begin{equation*}
       \sigma = -\pi \id + 2\mu \epsilon
\end{equation*}
of a linear dependence of the shear stresses on the local strain rate, the viscosity being the constant of proportionality.
In contrast to this, the strain-stress relation is nonlinear in the case of a non-Newtonian fluid with shear-dependent viscosity, i.e. the constitutive law reads
\begin{equation*}
    \sigma = -\pi \id + 2\mu(|\epsilon(u)|) \epsilon(u),
\end{equation*}
with a general function $\mu\colon [0,\infty) \to \R_+$. The mathematical analysis in the present paper fits particularly well the setting of so-called \emph{power-law fluids} or \emph{Ostwald--de Waele fluids} the constitutive law of which is given by
\begin{equation*}
    \mu(|\epsilon(u)|) 
    =
    \mu_0 |\epsilon(u)|^{p-2},
    \quad
    p > 1.
\end{equation*}
Here, $p > 1$ and $\mu_0 > 0$ denote the \emph{flow behaviour exponent} and the \emph{flow-consistency index}, respectively. For $1 < p < 2$ the corresponding fluid is called \emph{shear thinning}, as the the viscosity decreases with an increasing strain rate. For $p=2$ we recover the case of a Newtonian fluid, while for $p > 2$ the viscosity is an increasing function of the strain rate and the fluid is called \emph{shear thickening}. In other words, under force the fluid becomes either more liquid ($1<p<2$) or more solid ($p > 2$).
One reason why power-law fluids receive particular attention is their nice mathematical structure. However, real fluids usually feature a Newtonian behaviour at rather low and/or rather high shear rates. In this regard, for instance the Ellis constitutive law \cite{WS} 
\begin{equation}\label{eq:ellis}
    \frac{1}{\mu(|\epsilon|)}
    =
    \frac{1}{\mu_0} \left(1 + \left|\frac{\tigma}{\tigma_{1/2}}\right|^{\alpha- 1}\right)
\end{equation}
seems to be a more appropriate model (for shear-thinning fluids). Here, $\mu_0>0$ denotes the viscosity at zero shear stress, $\tigma = 2\mu(|\epsilon|) \epsilon$  is the viscous part of the Cauchy stress tensor and $\tigma_{1/2} > 0$ denotes the shear stress at which the viscosity has dropped to $\mu_0/2$. The value $\alpha=1$ reflects purely Newtonian behaviour, while $\alpha>1$ corresponds to shear-thinning behaviour.

\bigskip
\noindent\textbf{A variational formulation. }
Inspired by \cite{OSS,BS22}, we define a family $I_\eta$ of functionals with respect to a positive parameter $\eta\to 0$.
As common in the literature we call these \emph{Weighted Inertia-Dissipation-Energy} (WIDE) functional\schange{s} as they are in general obtained as a weighted sum of inertia, dissipation and energy of the viscous fluid. In the present work, we assume 
$u^0$ to be a suitable initial value and choose a constant $C_4\ge \tfrac 12(9C_P^2+1)$, where $C_P$ is the Poincar\'e constant of $\T_d$ for integrability $p=4$. For $\eta>0$ we define the functional
\begin{equation*}
    I_\eta(u) 
    = 
    \int_0^\infty \int_{\T_d} e^{-t/\eta} 
    \left(\tfrac{1}{2} |\partial_t u + (u \cdot \nabla) u|^2 + \tfrac{1}{\eta} W(\epsilon(u)) + \tfrac{C_4}{4}\abs{\nabla u}^4\right) \dd x \dd t
\end{equation*}
for any (measurable) function $u \colon (0,\infty) \times \T_d \to \R^d$ that satisfies the initial condition $u(0,\cdot)=u^0$ in a suitable trace sense, periodic boundary conditions, and the equation $\diverg u =0$ at any time $t>0$ in the sense of distributions. 

The function $W \colon \R^{d \times d}_{\sym} \to [0,\infty)$ is a suitable \emph{energy potential} that will be specified later. In terms of this potential, the conservation-of-momentum equation in \eqref{eq:PDE_intro1} can be rewritten in the form
\begin{equation*}
    \partial_t u + (u\cdot \nabla) u = -\nabla \pi + \diverg DW(\epsilon(u)),
    \quad
    t > 0,\ x \in \T_d.
\end{equation*}
For now, we only mention that 
$W(\epsilon) = \tfrac1p \mu_0\vert \epsilon \vert^p$ 
is admissible and corresponds to the case of power-law fluids. In particular the exponent $p$ is intimately connected to the growth of $W$. Moreover, it is possible to extend the analysis to constitutive laws not allowing for such a potential $W$, c.f. \cite{BS22}, for clarity we however stick to the assumption that such a $W$ exists.
 
We show that minimisers of the functional $I_{\eta}$ converge weakly to \emph{Leray--Hopf solutions} of the Navier--Stokes system \eqref{eq:PDE_intro1}, as $\eta\to 0$.

Heuristically speaking, the convergence of minimisers reduces to convergence of the \emph{Euler--Lagrange equations} associated with $I_{\eta}$ to the Navier--Stokes system \eqref{eq:PDE_intro1}, as $\eta \to 0$. Formally, the Euler--Lagrange equation corresponding to $I_\eta$ reads 
\begin{equation} \label{eq:EL_intro}
\begin{split}
    0
    =&
   \left(
    \partial_t u+ (u\cdot \nabla) u - \diverg DW(\epsilon(u)) + \nabla \pi
    \right) 
    \\
    & 
    - \eta \Bigl(\left(\partial_t^2 u + \partial_t\bigl((u\cdot \nabla)u\bigr)\right)
    -\diverg \left(\partial_t u \otimes u\right)  + (\nabla u)^T \partial_t u
    \\
    &
    - \diverg \left([(u\cdot \nabla)u] \otimes u\right) + (\nabla u)^T\bigl((u\cdot \nabla)u\bigr)
    -C_4\diverg\bra{\abs{\nabla u}^2\nabla u} \Bigr).
\end{split}
\end{equation}
Note that the first line in \eqref{eq:EL_intro} contains the conservation-of-momentum equation, with the term $\nabla \pi$ being the Lagrange multiplier corresponding to the incompressibility condition $\diverg u = 0$. For the remaining two lines we have to guarantee that these terms  vanish, as the parameter $\eta$ tends to zero. 

We explicitly allude to the fact that the Euler--Lagrange equation contains a second-order time derivative, i.e. it might be seen as an elliptic regularisation of the parabolic problem.

Moreover, we turn special attention to the \emph{stabilising term} $\tfrac{C_4}{4} \abs{\nabla u}^4$ incorporated in the functional $I_\eta$, which was first used in \cite{BS22}. A different stabilisation term has been used in \cite{OSS}. From the viewpoint of the Euler--Lagrange equation (cf. the last line in \eqref{eq:EL_intro}), this additional term should not lead to a crucial change in the limit $\eta \to 0$. In the original treatment of WIDE functionals such an additional term was not required \cite{MO08}, as the corresponding equations have a \emph{gradient-flow structure}. On the contrary, in the case of Navier--Stokes flow, the term is necessary in order to obtain suitable bounds allowing for Leray--Hopf solutions.

 \schange{This stabiliser corresponds to approximating the constitutive law by a $q$-fluid law with $q=4$ (\cite{Lady4}). The choice $q=4$ is convenient but more or less arbitrary and any larger $q$ would work as well. We emphasise that apart from this the stabiliser is quite unphysical and only used for the mathematical analysis.} On the other hand, the stabilising term enforces the \emph{energy inequality}, i.e. among all solutions to the Navier--Stokes system, it chooses a solution that at least satisfies the energy inequality. In particular, if the solution $u$ is very smooth (i.e. obeys an energy equality), it seems that we \emph{do not} need the stabilising term (e.g. for $p \geq 4$); whereas the stabiliser is necessary for small $p$ (even for the Newtonian case $p=2$ in three dimensions). 

We point out that also other variational approaches, different from the WIDE approach, have been pursued in the context of the Navier--Stokes equations. While in the WIDE approach we minimise an energy over space \emph{and} time, another ansatz is to discretise in time and solve minimisation problems step-by-step. This \emph{causal approach} has been investigated for instance in \cite{Gigli,BFS}. This approach of time-discretisation has the advantage of being \emph{causal} in time, i.e. that re-starting the system \emph{does not} change the solution. However, it requires to solve multiple minimisation problems at each time step. In contrast, fixing a parameter $\eta>0$ in the WIDE approach, leads to \emph{only one} minimisation problem.

\bigskip
\noindent\textbf{Main results of the paper. }
We briefly summarise the main  results of the article without introducing the exact notation. For a precise definition of Leray--Hopf solutions and of the underlying space $U_\eta$ we refer to Section \ref{sec:LH} and Section \ref{sec:WIDE}, respectively. Leray-Hopf solutions are weak solutions that, in addition, obey the energy inequality
\begin{equation*}
     E[u](t) + \int_0^t \int_{\T_d} DW(\epsilon(u)) \colon \epsilon(u) \dd x \dd s \leq E[u](0)
    ,\quad \text{where} \quad E[u](t) = \tfrac{1}{2} \int_{\T_d} 
 \vert u(t,x) \vert^2 \dd x.
 \end{equation*}
 \schange{We assume that $1<p<\infty$ is a growth exponent and $W$ is convex and satisfies natural $p$-growth as well as $p$-coercivity conditions}. Ignoring for the moment questions of regularity, the main result can be roughly stated as follows.

\begin{thmintro}[Existence of Leray--Hopf solutions, cf. Theorem \ref{thm:main}] \label{thm:intro1}
Let $p > \tfrac{2d}{d+2}$. For each $\eta > 0$ the functional $I_\eta$ possesses a minimiser $u_\eta \in U_\eta$. Moreover, there exists a subsequence $u_\eta$ (not relabeled) that converges weakly to a Leray--Hopf solution of the non-Newtonian Navier--Stokes system 
\begin{equation} \label{eq:PDE_intro_W}
\begin{cases}
    \partial_t u + (u\cdot \nabla) u = -\nabla \pi + \diverg DW(\epsilon(u)),
    &
    t > 0,\, x \in \T_d
    \\
    \diverg u = 0,
    &
    t > 0,\, x \in \T_d
    \\
    u(0,x) = u^0(x), 
    & x \in \T_d.
\end{cases}
\end{equation}
\end{thmintro}






The first part of this paper is structured similarly to
\cite{OSS}, but we have to invest considerable effort in the second part to prove convergence of the nonlinear viscous term $\diverg DW(\epsilon(u_\eta))$ despite
weak convergence. This term is linear in the Newtonian setting of \cite{OSS} and thus agrees well with
weak convergence, while the underlying strong convergence helps in \cite{BS22}.

\rchange{
At this point, we mention that in \cite{BS22}, Theorem \ref{thm:main} is proven in the case $p \geq \tfrac{3d+2}{d+2} (=\tfrac{11}{5} \text{ in 3D})$ and with boundary conditions and more general assumptions on the constitutive law. In this regime of power-law exponents, the method of (elliptic-parabolic) Lipschitz truncation, which is one of the main novelties of the present article (cf. Section \ref{sec:truncation}), is \emph{not needed}. One of the merits of \cite{BS22} is the derivation of the outflow boundary conditions for the incompressible Navier-Stokes equation. Our focus is on the applicability of the WIDE method for lower $p$, and in order to avoid more technicality in the proofs of Section \ref{sec:truncation}, we restrict ourselves to the torus and refrain from using boundary conditions and more general constitutive laws. However, we remark that Lipschitz truncations with boundary values \cite{FJM,DSSV} and with more involved constitutive laws (e.g. \cite{Bulicek2,Bulicek3}) have been achieved in quite similar contexts, which suggests that it might also be possible here for $p<\tfrac{3d+2}{d+2}$.}

We observe that the statement of Theorem \ref{thm:intro1} can be further strengthened in the following sense for large $p$, see Section \ref{sec:definitely_Bergdoktor}.


\begin{thmintro}[\schange{Strongly} shear-thickening fluids -- cf. Theorem \ref{prop:strong}] \label{thm:intro3}
If $p > \tfrac{3d+2}{d+2}$ and $W$ is strictly convex, then minimisers $u_\eta \in U_\eta$ of $I_\eta$ converge \textbf{strongly} to a solution of the non-Newtonian Navier--Stokes system \eqref{eq:PDE_intro_W} that obeys the \textbf{energy equality}
\begin{equation*}
    E[u](t) + \int_s^t \int_{\T_d} DW(\epsilon(u)):\epsilon(u) \dd x \dd \tau = E[u](s) \quad \text{for all } s,t \in [0,\infty).
\end{equation*}

\end{thmintro}

The observation of Theorem \ref{thm:intro3} is that the convergence of the minimisers in Theorem \ref{thm:intro1} is actually strong in the regime $p > 2$ in dimension $d=2$, respectively for $p > \tfrac{11}{5} > 2$ in dimension $d=3$. In this case, the convergence of the nonlinear viscosity term is direct and the energy equality is conserved in the limit $\eta \to 0$. The corresponding solution to the non-Newtonian Navier--Stokes system is called an energy solution.
The threshold  $p\geq\frac{11}{5}$ is addressed in \cite{BS22} by similar arguments. We remark furthermore that Theorem~\ref{thm:intro1} answers the question raised in \cite[Section 4, Remark (f), p. 5570]{BS22} in the positive.

In other words, the convergence results reflect the rheological behaviour of the fluid, and allow to gain insights into possibly turbulent behaviour of the fluid flow. Roughly speaking, we obtain rather weak convergence results in the shear-thinning and mildly shear-thickening regimes, whence these flows might exhibit turbulent behaviour (\schange{displayed by an energy \emph{in}equality}). In contrast to that, for `strongly' shear-thickening fluids, we obtain strong convergence of the nonlinear terms and hence \schange{the energy inequality becomes an equality}.

In Section \ref{sec:definitely_Bergdoktor}, the strong convergence result of Theorem \ref{thm:intro3} is obtained as a consequence of Theorem \ref{thm:intro1} and the validity of the energy inequality, while in \cite{BS22} the same result is achieved without the use of Theorem \ref{thm:intro1}. The crucial step in the proof is to show the convergence of the non-linearity $DW(\epsilon(u_\eta))$, provided only weak convergence of $u_\eta$. In the \schange{strongly shear-thickening} regime, this convergence follows from the strong convergence of $u_\eta$, and, in the Newtonian regime this term is actually linear and therefore unproblematic (cf. \cite{OSS}). Our approach reveals the differences between sub- and supercritical exponents and extends the analysis to the former.


\bigskip
\noindent\textbf{The variational approach in the context of related PDE results. }
We emphasise that by no means we attempt to review the huge existing literature on the Navier--Stokes equations. For an overview we refer the reader for instance to the textbooks \cite{Temam,RRS}.
Instead, we only briefly comment on how the results outlined in Theorems \ref{thm:intro1}--\ref{thm:intro3}, achieved via variational methods, fit into the literature on more classical results obtained by PDE methods. 

The existence of solutions obeying an energy equality 
for supercritical exponents $p > \frac{3d+2}{d+2}$ can be traced back to Ladyshenskaya \cite{Lady1,Lady2,Lady3}, cf. also \cite{Lionsbook}. For more recent results on weak solutions and Leray--Hopf solutions of the non-Newtonian Navier--Stokes system we refer to \cite{Malek2,MN01,Malek1,DRW}.


Even in the Newtonian case, it is still unclear whether there exists a solution obeying an energy equality even for very regular initial data. Heuristically speaking, the main issue for subcritical exponents is that the map $u \mapsto \diverg (u \otimes u)=(u\cdot\nabla)u$ is not weakly compact in the right space, i.e. even if $u_k \weakto u$ in $L_p((0,T);W^1_p)$, then $\diverg (u_k \otimes u_k)$ might \emph{not} converge weakly to $\diverg(u \otimes u)$ in the required space $L_q((0,T);(W^1_p)')$, where $\frac 1q+\frac 1p=1$. 
One of the virtues of the variational approach used in this paper is that this phenomenon of lacking weak compactness becomes clearly visible. From a PDE perspective this has been studied via the method of \emph{convex integration}, cf. \cite{DL09,DL10,DL22} for Euler's equation and, more recently, in \cite{BV19,BMS} for both the Newtonian and the non-Newtonian Navier--Stokes system. For a recent discussion of the energy inequality see \cite{BC20}.

\bigskip

\noindent \textbf{Solenoidal Lipschitz truncation. }
We approximate the non-Newtonian Navier--Stokes system \eqref{eq:PDE_intro_W} by a sequence of minimisation problems that can be understood as a sequence of \emph{elliptic} problems in space-time. In contrast, in previous literature on the existence of solutions to the non-Newtonian Navier--Stokes equations for small $p$, cf. \cite{FMS,DMS,BDF,DKS}, the system is approximated by a \emph{parabolic} regularisation that involves introducing a modified constitutive law, i.e. $\tilde{\sigma}(\epsilon) = \sigma(\epsilon) + C_4\eta \vert \epsilon \vert^2 \epsilon$. In this case, one solves the non-Newtonian system with the modified law and considers the limit $\eta \to 0$. 

The main challenge in the mentioned works (cf. also \cite{BDF} for a stationary system) is to show the weak convergence of the nonlinear term $DW(\epsilon_\eta)$. As the main technical tool, based on \cite{AF84}, the method of (solenoidal) Lipschitz truncation has been further developed and applied to problems in fluid mechanics, cf. \cite{BDS} and references mentioned earlier. This technique solves the main issue when showing that $DW(\epsilon_\eta) \weakto DW(\epsilon)$ (as $\eta \to 0$) for a weakly converging sequence $\epsilon_\eta$: while concentrations are uncritical for weak convergence of nonlinear terms, oscillatory effects usually destroy weak convergence when faced with nonlinear terms. The \emph{Lipschitz truncation} deals with the issue as follows: it cuts away concentration effects and allows to show that there are no oscillations (e.g. by a modified version of Minty's trick).

Returning to the setting at hand, as we deal with an elliptification of the problem in time, the different forms of Lipschitz truncation developed in previously mentioned works are not suited to our problem -- they work best with \emph{parabolic problems} (cf. also \cite{KL00,DSSV}). Instead, we modify the parabolic approach to fit into the setting of elliptic regularisation. We refer to Section \ref{sec:5} for the application of the truncation result and to \ref{sec:truncation} for its proof.

\bigskip
\noindent\textbf{Outline of the paper. }
In Section \ref{sec:functional_setup} we introduce the functional analytic setting. More precisely, we introduce the function spaces, define a suitable notion of a Leray--Hopf solution to the non-Newtonian Navier--Stokes problem and introduce the WIDE functional. Moreover, Section \ref{sec:functional_setup} contains some frequently used interpolation results and a detailed discussion on the regularity properties of Leray--Hopf solutions.

Apart from a short intermezzo in Section \ref{sec:definitely_Bergdoktor} that deals with the strong convergence result Theorem \ref{thm:intro3}, the rest of the paper is concerned with the proof of Theorem \ref{thm:intro1}. The first part of the proof is conducted in Section \ref{sec:shearthinning}. In particular, we show the existence of minimisers to $I_\eta$ that satisfy a particular energy inequality. Moreover, we derive bounds in certain function spaces and convergence to a solution $u$ of an equation containing the unidentified limit $\chi$ of the nonlinear term $DW(\epsilon(u_\eta))$. As mentioned before, the main issue is to identify the limit $\chi$ with $DW(\epsilon(u))$. We conclude Section \ref{sec:shearthinning} by a pedagogical presentation of a short argument that finishes the proof under an additional, but unnatural, constraint on the minimising sequence $u_\eta$.

Without this additional constraint the proof is much more involved and hinges on a Lipschitz truncation statement, cf. Lemma \ref{lemma:trunc}. In Section \ref{sec:5} we carry out the proof of Theorem \ref{thm:intro1} when provided with Lemma \ref{lemma:trunc}. As the proof of this Lipschitz truncation result is rather technical and independent of the proof of Theorem \ref{thm:intro1}, we defer the construction of this truncation to the last Section \ref{sec:truncation}.

In between, in Section \ref{sec:definitely_Bergdoktor}, we discuss how to obtain strong convergence and the energy \emph{equality} for supercritical $p$, i.e. we prove Theorem \ref{thm:intro3}.

\section{Functional setup and statement of the main result} \label{sec:functional_setup}

\subsection{Functional setup and the potential $W$}
We denote by $\T_d$ the $d$-dimensional flat torus, $d \geq 2$. Let $1 \leq p \leq \infty$.
We define the space 
\begin{equation*}
    V^1_p \coloneq \{ u \in W^1_{p}(\T_d;\R^d) \colon \diverg u=0 \},
\end{equation*}
where $W^1_{p}$ is the Sobolev space with integrability $p$. Accordingly, we define $V^0_p$ to be the closure of $V^1_p$ in $L_p(\T_d;\R^d)$. Note that $V^1_p$ is weakly and strongly closed in $W^1_{p}(\T_d;\R^d)$ and inherits its norm. By the Hahn--Banach theorem \schange{we may extend any linear functional acting on $V^1_p$ to the entire space $W^1_p(\T_d;\R^d)$ and therefore we may identify  elements in $(V^1_p)'$ with suitable representatives in $(W^1_{p}(\T_d;\R^d))'$}. Moreover, denoting for $1<p<\infty$ the \emph{dual exponent of $p$} by $q=\tfrac{p}{p-1}$, we may identify 
\begin{equation*}
    (W^1_{p}(\T_d;\R^d))' = W^{-1}_q(\T_d;\R^d).
\end{equation*}

\bigskip
\noindent\textbf{Notation. }
As long as we do not fear ambiguity, we use the short hand $L_p$ for the usual (spatial) Lebesgue spaces $L_p(\T_d;\R^d)$ and $L_p(\T_d;\R^{d\times d})$. Moreover, by $\langle\cdot,\cdot\rangle$ we denote the dual pairing between an element of some function space and its dual. The choice of the function space should always be clear from the context. We denote the nonlinear term by either of the forms $(u\cdot\nabla)u=\diverg(u\otimes u)$ which are equal due to the incompressibility condition $\diverg u=0$.


\bigskip
\noindent\textbf{{The potential $W$. }}  For the sake of simplicity, we stick to the case where $\sigma(\epsilon)$ may be written as $\sigma(\epsilon)= - \pi \id + DW(\epsilon)$. In particular, this characterisation is valid for any fluid with $\sigma(\epsilon) = -\pi \id + 2 \mu(\vert \epsilon \vert) \epsilon$ (sometimes called generalised Newtonian). If the law $\sigma= \sigma(\epsilon)$ is not given by an energy potential, then there is a related approach that preserves the nice variational structure, cf. \cite{BS22}.

In this paragraph we specify the conditions on the energy potential $W$. By $\R^{d \times d}_{\sym,0}$ we denote the space of symmetric $(d \times d)$-matrices with zero trace. Observe first that, if $u \in V^1_p$, then the rate-of-strain tensor satisfies
\begin{equation*}
    \epsilon(u) = \tfrac{1}{2} (\nabla u + (\nabla u)^T) \in L_p\bigl(\T_d;\R^{d \times d}_{\sym,0}\bigr).
\end{equation*}
We now consider a \emph{potential} $W \colon \R^{d \times d}_{\sym,0} \to [0,\infty)$ with the following properties
\begin{enumerate} [label=($W$\arabic*)]
     \item\label{it:W1} $W \in C^1\bigl(\R^{d\times d }_{\sym,0};\R\bigr)$;
    \item\label{it:W2} \schange{$W$ is \textbf{convex}.}
    \item\label{it:W3} \textbf{Coercivity and $p$-growth. } \rchange{There exists a constant $C>0$ such that} \begin{equation*}
        C^{-1} |\epsilon|^p - C \leq W(\epsilon) \leq C(1 + |\epsilon|^p)
    \end{equation*} 
    \item\label{it:W4} \textbf{Coercivity  and bounds for $DW$. } \rchange{There exists a constant $C>0$ such that}
    \begin{equation*}
     DW(\epsilon) : \epsilon \geq \max\{C \vert \epsilon \vert^p -C,0\}
     \quad \text{and} \quad 
     \vert DW(\epsilon) \vert \leq C(1+\vert \epsilon \vert^{p-1}).
     \end{equation*}
\end{enumerate}
Observe that if $W$ has the form $W(\epsilon) = \tilde{W}(|\epsilon|)$,
then $DW(\epsilon) = \left(\vert \epsilon \vert^{-1} \tilde{W}'(\vert \epsilon \vert)\right)\epsilon$.
In this case, we may consider
\begin{equation}\label{eq:mu_W}
    \mu(s) \coloneq \tfrac{1}{2s} \tilde{W}'(s)
\end{equation}
as the viscosity of the fluid at shear rate $s \in (0,\infty)$, and by \eqref{eq:mu_W} it is possible to construct $W$ by integration from a given viscosity function $\mu$ .
Moreover, we can rewrite the stress $\sigma$ in terms of $DW$ as
\begin{equation*}
    \sigma(\epsilon) = -\pi \id + DW(\epsilon) 
    = 
    -\pi \id + 2\mu(|\epsilon|) \epsilon.
\end{equation*}


\begin{example}
We briefly discuss the form of the potential $W$ for widely-used constitutive viscosity laws.
\begin{itemize}
    \item \textbf{Newtonian fluids. } Newtonian fluids feature a constant viscosity $\mu(|\epsilon|) \equiv \mu_0 > 0$. In this case, the relation between the viscous stress $DW(\epsilon)$ and the local strain $\epsilon$ is perfectly linear with $DW(\epsilon) = 2 \mu_0 \epsilon$. We assume without loss of generality that $\mu_0 = 1/2$. Then, the potential $W$ is given by
    \begin{equation*}
        W(\epsilon) 
        =
        \tilde{W}(|\epsilon|)
        = 
        \tfrac{1}{2} |\epsilon|^2
        \quad
        \text{with}
        \quad
        DW(\epsilon) = 2 \mu_0 \epsilon = \epsilon.
    \end{equation*}
    \item \textbf{Power-law fluids. } In the case of power-law fluids, the constitutive law for the fluid's viscosity is $\mu(|\epsilon|) = \mu_0 |\epsilon|^{p-2}$ with a flow-consistency index $\mu_0 > 0$ and a flow-behaviour exponent $p$. As above, we assume without loss of generality that $\mu_0=1/2$. Then, the potential $W$ is given by
    \begin{equation*}
        W(\epsilon)
        =
        \tilde{W}(|\epsilon|)
        = 
        \tfrac{1}{p} |\epsilon|^p
        \quad
        \text{with}
        \quad
        DW(\epsilon) = 2 \mu(|\epsilon|) \epsilon = |\epsilon|^{p-2} \epsilon.
    \end{equation*}
    \item \textbf{Ellis fluids. } The implicit characterisation \eqref{eq:ellis} does not allow for an explicit formula for $W$. Nevertheless, the right-hand side of \eqref{eq:ellis} is monotone in both $\mu(s)$ and $s$, and it follows that for $\alpha>1$, $\mu$ is monotonically decreasing. Furthermore  
    \begin{align*}
        \mu(s)\sim \mu_0\min\left\{1,\bra{\frac{\mu_0 s}{\sigma_{1/2}}}^{\frac{1-\alpha}{\alpha}}\right\}.
    \end{align*}
    Thus, for $\alpha>1$, defining $\tilde W(s)\coloneq \int_0^s 2r\mu(r)\dd r$, and $p=\frac{\alpha+1}{\alpha}$ , it follows that $W$ has $p$-growth and $p$-coercivity outside a compact set and thus \ref{it:W3} is satisfied. Since $\mu$ is monotonically decreasing, it follows from \eqref{eq:ellis} that $s\mapsto s\mu(s)$ is monotonically increasing, and hence both $\tilde W$ and $s\tilde W^\prime(s)$ are convex, whence \ref{it:W1}--\ref{it:W4} are satisfied for this $W$.
\end{itemize}
\end{example}


\subsection{Leray--Hopf solutions} \label{sec:LH}
We define a suitable concept of weak solutions to the non-Newtonian Navier--Stokes system. If we express the viscous stresses in terms of the potential $W$ instead of the viscosity function $\mu$, the Navier--Stokes system in strong form reads
\begin{equation} \label{strongform}
    \begin{cases}
        \partial_t u + (u \cdot \nabla) u = - \nabla \pi + \diverg  DW(\epsilon(u)),  
        &
        t > 0,\, x \in \T_d 
        \\
        \diverg u=0, 
        &  
        t > 0,\, x \in \T_d 
        \\
        u(0,x) = u^0(x), 
        & x \in \T_d.
    \end{cases}
\end{equation}
We assume throughout the paper that $p > \frac{2d}{d+2}$, i.e. $p > 1$ in dimension $d=2$ and $p > \tfrac65$ in dimension $d=3$. For solutions $u \in L_{p}((0,\infty);V_p^1)$ this implies in particular that  $u\in L_{p}((0,\infty);L_2)$ by the Sobolev embedding (the embedding also works for $p \geq \frac{2d}{d+2}$, but we avoid dealing with the dual of $V^1_1$ when investigating Leray--Hopf solutions) and we may define the \emph{energy functional} $E[u](\cdot)$ for almost every time $t$ as 
\begin{equation*} 
    E[u](t) \coloneq 
    \tfrac{1}{2}
    \int_{\T_d} \vert u(t) \vert^2 \dd x.
\end{equation*}
Furthermore, we introduce the exponent $\beta$ as 
\begin{equation} \label{def:beta}
        \beta \coloneq \max \left \{p,\tfrac{dp}{dp+2p-2d} \right\}.
\end{equation}
An explanation for the definition of the exponent $\beta$ is given in Remark \ref{rem:regularity-I} below.
We now define the concept of a \emph{Leray--Hopf solution}, which is a solution satisfying certain regularity assumptions \emph{and} a global energy inequality (cf. \cite{Leray,BMS}).

\begin{definition}[Leray--Hopf solution] \label{def:Leray--Hopf}
Let $u^0 \in V^0_2$ and $\tfrac{2d}{d+2}<p<\infty$ be given. We call a function 
\begin{equation*} 
    u \in L_p((0,\infty);V_p^1)
    \quad \text{with} \quad
   \partial_t u \colon (0,\infty) \to (V_{\beta}^1)^\prime
\end{equation*}
a \textbf{Leray--Hopf solution} to the non-Newtonian Navier--Stokes system \eqref{strongform} if it has the following properties.
\begin{enumerate} [label=(\roman*)]
    \item $u$ satisfies \eqref{strongform} at almost every time in the sense of distributions, i.e. 
    \begin{equation} \label{LH:equation}
        \langle \partial_t u, \phi \rangle + \langle (u \cdot \nabla) u, \phi \rangle + \langle DW(\epsilon(u)), \epsilon(\phi) \rangle =0 \quad \text{for a.e. } t>0 
    \end{equation}
    for all test functions $\phi \in C^{\infty}(\T_d;\R^d)$ with $\diverg \phi=0$ and $\int_{\T_d} \phi \dd x=0$;
    \item $u$ satisfies the initial condition $u(0,\cdot) = u^0$ in the weak $L_2$ sense;
    \item $u$ satisfies the \textbf{energy inequality}
        \begin{equation} \label{LH:energy}
            E[u](t) 
            +
            \int_0^t \int_{\T_d} DW(\epsilon(u(\tau))) : \epsilon(u(\tau)) \dd x \dd \tau\leq E[u^0]
        \end{equation}
        for almost every time $t >0$. 
    \end{enumerate}   
\end{definition}

\begin{remark} \label{rem:regularity-I}
We briefly mention some regularity properties of Leray--Hopf solutions that immediately follow from the definition. A more detailed discussion of the regularity and the initial condition is offered in Section \ref{sec:regularity}.
\begin{enumerate} [label=(\roman*)]
    \item\label{it:reg1} By the Sobolev embedding we have $u \in W^1_p \hookrightarrow L_\frac{pd}{d-p}$. Since we require $p>\tfrac{2d}{d+2}$, this implies $W^1_p \hookrightarrow L_{2+2\varepsilon}$ for some small $\varepsilon>0$. Therefore, $(u \otimes u)$ is a well-defined element in $L_{1+\varepsilon}$ and $\diverg (u \otimes u)$ is in $W^{-1}_{1+ \varepsilon}$.
    \item The second term in \eqref{def:beta} is derived as follows: for $u\in V^1_p$ we obtain that $(u\otimes u)$ is well-defined in $L_{\beta'}$ with $\beta' = \tfrac{dp}{2(d-p)}$ and consequently $\diverg(u\otimes u) \in (V^1_\beta)'$. This coincides with the spatial regularity of $\partial_t u$ in the definition \ref{def:Leray--Hopf} of a Leray--Hopf solution. The dual exponent $\beta$ is given by $\beta = \tfrac{dp}{dp+2(p-d)}$. \\
    Moreover, $\beta\to \infty$, as $p\to \tfrac{2d}{d+2}$, while $\beta=p$ for $p\ge \frac{3d}{d+2}$.
    \item  By a density argument we obtain that \eqref{LH:equation} is in fact satisfied for all $\phi \in V^1_\beta$.
    \item In particular, if $p \geq \tfrac{3d}{d+2}$, we have $\beta=p$ and \eqref{LH:equation} holds in the space $(V^1_p)'$.
    \item The energy inequality dictates that $u \in L_{\infty}((0,\infty);L_2)$, as the energy $E[u](\cdot)$ is decreasing in time along solutions.
\end{enumerate}
\end{remark}
Concerning the energy inequality, it may be shown that \eqref{LH:energy} holds with \emph{equality}, whenever $u$ is smooth enough.


\begin{definition}
    We say that a Leray--Hopf solution is an \textbf{energy  solution} to the non-Newtonian Navier--Stokes equation, whenever \eqref{LH:energy} is satisfied with equality for any $t>0$.
\end{definition}


The following threshold of regularity is well-known (e.g. \cite{MNRR,Bulicek}).


\begin{proposition} \label{prop:energy}
Let $p > \tfrac{3d+2}{d+2}$ and let $u$ be a Leray--Hopf solution to the non-Newtonian Navier--Stokes equations. Then the following holds true:
    \begin{enumerate} [label=(\roman*)]
        \item \label{CL:1} $\partial_t u \in L_{q,\loc}((0,\infty);(V^1_p)')$;
        \item \label{CL:2} $u$ is an energy solution.
    \end{enumerate}
\end{proposition}
The first assertion in Proposition \ref{prop:energy} follows from repeatedly using the equation \eqref{LH:equation} and interpolation results. We further hint at this in Subsection \ref{sec:regularity}. The second assertion then follows from testing equation \eqref{LH:equation} with $u$ itself and realising that it is allowed to integrate in time.

We mention that in the physical dimensions $d = 2,3$ the critical exponent in Proposition \ref{prop:energy} is given by
\begin{equation*}
    \tfrac{3d+2}{d+2}
    =
    \begin{cases}
        2 & \text{if } d=2
        \\
        \tfrac{11}{5} > 2 & \text{if } d=3.
    \end{cases}
\end{equation*}

\subsection{Definition of the WIDE functional} \label{sec:WIDE}
We now introduce the Weighted Inertia-Dissipation-Energy (WIDE) functional. \schange{For a positive parameter $\eta>0$ we define functionals $I_\eta$ and are interested in the limit as $\eta \to 0$}.  To this end, let $1<p<\infty$ and assume that the potential $W$ satisfies \ref{it:W1}--\ref{it:W4}. Given $\eta > 0$ , we define for $u\in U_\eta$ the WIDE functional $I_\eta$ by

\begin{equation*} 
    I_{\eta}(u) \coloneq \int_0^{\infty} e^{-t/\eta} \int_{\T_d} \tfrac{1}{2} \vert \partial_t u 
    + 
    \diverg (u \otimes u) \vert^2+
    \tfrac{1}{\eta} W(\epsilon(u)) 
    +
    \tfrac{C_4}{4}\abs{\nabla u}^4
    \dd x \dd t.
\end{equation*}
In order for the functional $I_\eta$ to be well-defined, we introduce a suitable function space $U_\eta$. Let $u^0_{\eta} \in V^1_p$ be an approximation of $u^0$.


\begin{definition}\label{def:U}
We define the space $U_{\eta}$ as the space of functions $u \colon [0,\infty) \times \T_d \to \R^d$, such that the following is satisfied.
\begin{enumerate} [label=($U_\eta$-\arabic*)]
    \item $e^{-t/(p \eta)} u \in L_p((0,\infty);V^1_p)$;
    \item\label{it:Ueta2} $e^{-t/(2\eta)} \partial_t u \in L_2((0,\infty);L_2)$;
    \item\label{it:Ueta3} $e^{-t/(4\eta)} u \in L_4((0,\infty);V^1_4)$;
    \item\label{it:Ueta4} $u(0) = u^0_\eta$;
    \item\label{it:Ueta5} $\int u_\eta(t) \dd x = \int u^0_{\eta} \dd x$ for a.e. $t \in (0,\infty)$.
\end{enumerate}
\end{definition}


\noindent Note that the exponential factor is adapted to the functional $I_\eta$ and that the $L_2$-norm of $\diverg (u\otimes u)$ is controlled by the $V^1_4$ norm of $u$. The terms without the exponential factor enjoy the corresponding regularity \emph{locally} in time. Furthermore \ref{it:Ueta4} is a reasonable initial condition as \ref{it:Ueta2} implies that $u \in C^{\alpha}((0,T);L_2)$ for some $\alpha>0$. 
In Section \ref{sec:shearthinning}
we investigate minimisers to the functional $I_\eta$ among all functions $u_\eta \in U_\eta$. 

\bigskip
\noindent{\textbf{Approximation of the initial value. }Let an initial condition $u^0 \in V^0_2$
be given. For simplicity, throughout this paper we assume that
\begin{equation} \label{eq:av0}
    \int_{\T_d} u^0 \dd x = 0.
\end{equation}
The average is conserved by the evolution, and the assumption is without loss of generality since other averages can also be considered, e.g. by a change of coordinates.
For $\eta>0$ we approximate the initial condition $u^0$ by $u^0_\eta$ such that
\begin{enumerate} [label=(IV\arabic*)]
    \item \label{it:u01} 
    $u^0_\eta \to u^0$ in $V^0_2$;
    \item\label{it:initial2} $\Vert \nabla u^0_\eta \Vert_{L_p}^p \leq C_0 \frac{1}{\eta}$;
    \item\label{it:initial3} $\Vert \diverg (u^0_{\eta} \otimes u^0_{\eta}) \Vert_{L_2}^2 \leq C_0 \frac{1}{\eta}$;
    \item\label{it:initial4} $\Vert \nabla u^0_\eta \Vert_{L_4}^4 \leq C_0 \tfrac{1}{\eta}$;
    \item \label{it:inital5} $\int_{\T_d} u^0_\eta \dd x =0$.
\end{enumerate}


\subsection{Main result}
We are now in a position to precisely formulate the main result of this paper. For this purpose we define an exponent $\gamma$ as 
\begin{equation} \label{eq:def_gamma}
    \gamma \coloneq \max\{ p,4\}.
\end{equation}
The choice of $\gamma$ will be justified in Subsection \ref{sec:EL:1}.


\begin{theorem}[Main result] \label{thm:main}
Let $p>\frac{2d}{d+2}$. For each $\eta > 0$ the functional $I_\eta$ possesses a minimiser $u_{\eta} \in U_{\eta}$. Moreover, there exist a subsequence $u_{\eta}$ (not relabeled) and a function $u \in L_p((0,\infty);V^1_p) \cap L_{\infty}((0,\infty);L_2)$ such that  
\begin{equation} \label{main:convergence}
    \begin{split}
        u_{\eta} \lweakto u &\quad \text{in } L_p((0,\infty);V^1_p).\\
        \partial_t u_{\eta} \lweakto \partial_t u &\quad \text{in } L_{p\wedge q}((0,T);(V^1_\gamma)') 
        \quad \text{for any } T>0,
    \end{split}
\end{equation}
where $p\wedge q=\min\{p,q\}$. The limit element $u$ is a Leray--Hopf solution to the non-Newtonian Navier--Stokes system with initial value $u^0$. 
\end{theorem}

If $p>\tfrac{3d+2}{d+2}$ in Theorem \ref{thm:main}, then $u$ is an energy solution and the convergence
\begin{equation*}
    u_{\eta} \longrightarrow u \quad \text{in } L_p((0,\infty);V^1_p) 
\end{equation*}
is strong, cf. Theorem \ref{prop:strong}.


Note that the statement for shear-thinning fluids is slightly different from the shear-thickening case regarding the convergence of the time derivative. While for $p \geq 2$, the regularity in time is the one that is natural in order to be able to test with $u$, for $p<2$ we obtain less regularity. 

The proof of Theorem \ref{thm:main} roughly consists of the following steps:
\begin{itemize}
    \item show that minimisers of $I_{\eta}$ exist and derive a rather weak bound on their norm in $U_{\eta}$;
    \item verify the energy inequality for minimisers $u_\eta$ and infer uniform estimates for the terms appearing in \eqref{main:convergence} to allow for a weakly convergent subsequence; 
    \item show that the limit is a Leray--Hopf solution to the non-Newtonian Navier--Stokes system.
\end{itemize}
Before we start, we recall some crucial and frequently used interpolation results and their consequences in the following Subsection \ref{sec:Amann}.

\subsection{Interpolation results} \label{sec:Amann}

In this subsection, we apply some abstract interpolation results for fractional Sobolev spaces from Amann's book \cite{Bibel} in order to obtain a rather sharp Aubin--Lions type result that fits the functional setting of (fractional) Sobolev spaces of our paper and is crucial for the subsequent analysis.
For this purpose, we introduce the following notation. Let $s_0, s_1 \in \R$ with $s_0\neq s_1$ and let $q_0, q_1 \in [1,\infty)$. Then we define
\begin{equation*}
    s_\theta = (1-\theta) s_0 + \theta s_1
    \quad
    \text{and}
    \quad
    \tfrac{1}{q_\theta}
    =
    \tfrac{1-\theta}{q_0} + \tfrac{\theta}{q_1}.
\end{equation*}
A central role regarding the regularity results in the following subsection \ref{sec:regularity} is occupied by the following Aubin--Lions type interpolation result.

\begin{theorem}[{\cite[Thm.  VII. 7.4.1]{Bibel}}] \label{thm:Aubin-Lions}
Let $X_1 \hookrightarrow X \hookrightarrow X_0$ be locally convex Banach spaces and let $0 < \theta < 1$ such that
\begin{equation*}
    X_1 \xhookrightarrow{c} X_0
    \quad
    \text{and}
    \quad
    \|x\|_X
    \leq
    C
    \|x\|_{X_0}^{1-\theta} \|x\|_{X_1}^{\theta},
    \quad
    x \in X_1,
\end{equation*}
for some constant $C>0$.
Moreover, let $s_0, s_1 \in [0,\infty)$ with $s_0\neq s_1$ and let $q, q_0, q_1 \in [1,\infty)$ such that
\begin{equation*}
    0 \leq s < s_\theta 
    \quad
    \text{and}
    \quad
    s-\tfrac{1}{q} < s_\theta - \tfrac{1}{q_\theta}.
\end{equation*}
Then we have
\begin{equation*}
    W^{s_1}_{q_1}\bigl((0,T);X_1\bigr) 
    \cap
    W^{s_0}_{q_0}\bigl((0,T);X_0\bigr)
    \xhookrightarrow{c} 
    W^{s}_{q}\bigl((0,T);X\bigr) 
\end{equation*}
\end{theorem}


Theorem \ref{thm:Aubin-Lions} is mainly applied to the case, where both $X_0$ and $X_1$ are fractional Sobolev spaces, whence we need the following Lemma 
(cf. \cite[Thm. VII. 7.2.5]{Bibel} and \cite[Thm. I.2.11.1]{OldTestament2}).

\begin{lemma}[Interpolation of fractional Sobolev spaces] \label{lem:MichaelWendler}
Let $\Omega \subset \R^d$ be a bounded domain and let $X$ be a Banach space. Moreover, let $h_0 < h_1 \in \R$ and $r_0, r_1 \in [1,\infty)$ such that $\tfrac{1}{r_1} - \tfrac{h_1}{d} < \tfrac{1}{r_0} - \tfrac{h_0}{d}$, i.e.
\begin{equation*}
    W^{h_1}_{r_1}(\Omega;X) 
    \xhookrightarrow{c}
    W^{h_0}_{r_0}(\Omega;X).
\end{equation*}
Then the interpolation inequality
\begin{equation*}
    \|x\|_{W^{h_\theta}_{r_\theta}(\Omega;X)}
    \leq
    C
    \|x\|_{W^{h_0}_{r_0}(\Omega;X)}^{1-\theta} \|x\|_{W^{h_1}_{r_1}(\Omega;X)}^{\theta},
    \quad
    x \in W^{h_1}_{r_1}(\Omega;X),
\end{equation*}
holds true for the space $W^{h_\theta}_{r_\theta}(\Omega;X)$.
\end{lemma}


Combining Theorem \ref{thm:Aubin-Lions} and Lemma \ref{lem:MichaelWendler} yields the following Aubin--Lions type interpolation result which is suitable for the setting of the non-Newtonian Navier--Stokes problem treated in our paper. 


\begin{corollary}[Aubin--Lions type interpolation] \label{coro:Draco_Malfoy}
Suppose that $h_0 < h_1 \in \R$ and $q_0, q_1, r_0, r_1 \in [1,\infty)$ satisfy $\tfrac{1}{r_1} - \tfrac{h_1}{d} < \tfrac{1}{r_0} - \tfrac{h_0}{d}$. Then
\begin{equation*}
    L_{q_1}\bigl((0,T);W^{h_1}_{r_1}(\T_d)\bigr)
    \cap
    W^1_{q_0}\bigl((0,T);W^{h_0}_{r_0}(\T_d)\bigr)
    \xhookrightarrow{c}
    L_{q(\theta)}\bigl((0,T);L_{r(\theta)}(\T_d)\bigr),
\end{equation*}
whenever $\max\{0,\tfrac{h_0}{h_0-h_1}\} < \theta < 1$ and
\begin{equation} \label{eq:conditions_Aubin-Lions}
    \tfrac{1}{q(\theta)}
    >
    -1 + \theta
    +
    \tfrac{1-\theta}{q_0}
    +
    \tfrac{\theta}{q_1}
    \quad
    \text{and}
    \quad
     \tfrac{1}{r(\theta)}
    >
    \tfrac{1-\theta}{r_0}
    +
    \tfrac{\theta}{r_1}
    -
    \tfrac{(1-\theta)h_0}{d}
    -
    \tfrac{\theta h_1}{d}.
\end{equation}
\end{corollary}

\begin{proof}
The proof follows from Theorem \ref{thm:Aubin-Lions} and Lemma \ref{lem:MichaelWendler}. Indeed, we identify
\begin{equation*}
    X_0 = W^{h_0}_{r_0}(\Omega)
    \quad
    \text{and}
    \quad
    X_1 =  W^{h_1}_{r_1}(\Omega),
\end{equation*}
as well as
\begin{equation*}
    s_0 = 1,
    \quad
    s_1 = 0,
    \quad
    s_\theta = (1-\theta) s_0 + \theta s_1
    =
    1-\theta.
\end{equation*}
Choosing $q=q_\theta$ and then $s < s_\theta = 1-\theta$, the conditions of Theorem \ref{thm:Aubin-Lions} are satisfied. Together with Lemma \ref{lem:MichaelWendler}, this implies
\begin{equation*}
    L_{q_1}\bigl((0,T);W^{h_1}_{r_1}(\T_d)\bigr)
    \cap
    W^1_{q_0}\bigl((0,T);W^{h_0}_{r_0}(\T_d)\bigr)
    \xhookrightarrow{c}
    W^s_q\bigl((0,T);W^{h_\theta}_{r_\theta}(\T_d)\bigr)
\end{equation*}
for 
\begin{equation*}
    h_\theta = (1-\theta) h_0 + \theta h_1
    \quad 
    \text{and}
    \quad
    \tfrac{1}{r_\theta}
    =
    \tfrac{1-\theta}{r_0} + \tfrac{\theta}{r_1}.
\end{equation*}
Observe that $h_\theta > 0$. Thus, by the Sobolev embedding theorem, we find that
\begin{equation*}
    W^{h_\theta}_{r_\theta}(\Omega) \xhookrightarrow{c}
    L_{r(\theta)}(\Omega)
    \quad
    \text{and}
    \quad
    W^s_q((0,T);B) 
    \xhookrightarrow{c}
    L_{q(\theta)}((0,T);B),
\end{equation*}
for a separable Banach space $B$,
whenever the conditions in \eqref{eq:conditions_Aubin-Lions} hold true.
\end{proof}


\subsection{Remarks on the regularity of weak solutions} \label{sec:regularity}
We use the interpolation results of the previous subsection to be able to make more precise statements on the regularity of Leray--Hopf solutions. For this purpose recall that by definition a Leray--Hopf solution has the regularity
\begin{equation}\label{eq:reg_rem}
    u \in L_p((0,\infty);V^1_p) \cap L_\infty((0,\infty);L_2)
    \quad \text{with} \quad
    \partial_t u\colon (0,\infty) \longrightarrow (V^1_\beta)'.
\end{equation}

The critical term in \eqref{LH:equation} for the regularity is the nonlinear duality pairing $\langle \diverg(u\otimes u), \phi \rangle$.


\begin{remark}[Regularity of equation \eqref{LH:equation} in space] \label{rem:reg_beta} 
We study the regularity of the equation \eqref{LH:equation} \emph{only in space}. Note that (for $p<d$)  $u \in V^1_p$ implies $ (u \otimes u) \in L_{dp/2(d-p)}$ and hence $\diverg (u \otimes u) \in W^{-1}_{dp/2(d-p)} \subset (W^1_{\beta,0})'$ for $\beta$ as in \eqref{def:beta}.
Thus, the dual pairing $\langle \diverg(u \otimes u), \phi \rangle$ is well-defined for any $\phi \in V^1_\beta$.
\end{remark}

\begin{remark}[Initial value and integrability in time]
We start from the regularity of $u$ in \eqref{eq:reg_rem}. Using Sobolev embedding and H\"older inequality we get that 
\begin{equation*}
    u \in L_\infty((0,\infty);L_2) \cap L_p((0,\infty);V^1_p) 
    \longhookrightarrow 
    L_2((0,T);L_r)
\end{equation*}
for some $r>2$, whenever $p > \tfrac{2d}{d+2}$, as then $W^1_p \hookrightarrow L_{2 +\varepsilon}$ for some $\varepsilon >0$. Consequently, $\diverg(u \otimes u) \in L_1((0,T);(V^1_s)')$ for $s$ being the dual exponent of $r/2$. Using the pointwise equation \eqref{LH:equation}, we get that $\partial_t u \in L_{1}((0,T);(V^1_s)')$, i.e. $\partial_t u$ is locally integrable in some weak space. \\
Therefore, $u \in L_{\infty}((0,T);L_2) \cap C([0,T];(V^1_s)')$. As a consequence, $t \mapsto u(t)$ is continuous with respect to the \emph{weak} topology of $L_2$. Therefore, prescribing the initial value $u^0$ is sensible.
\end{remark}


\begin{remark}[Regularity of equation \eqref{LH:equation} in space-time] The considerations of the previous remark show that the equation
\begin{equation} \label{LH:equation:2}
       \partial_t u + (u \cdot \nabla) u = \diverg DW(\epsilon(u)) 
\end{equation}
holds in space-time as an element of $L_{1}((0,T);(V^1_s)')$ for some large $s<\infty$. If $p \geq \max\{2, \tfrac{3d}{d+2}\}$, in light of Remark~\ref{rem:regularity-I}\ref{it:reg1} and Remark~\ref{rem:reg_beta}, equation \eqref{LH:equation:2} indeed holds in $L_{1}((0,T);(V^1_p)')$.  \\
Moreover, if $p>\tfrac{3d+2}{d+2}$, we may even show that \eqref{LH:equation:2} holds in $L_q((0,T);(V^1_p)')$, which is the dual of the space $L_p((0,T);V^1_p)$. This follows by bootstrapping via the interpolation statement Corollary \ref{coro:Draco_Malfoy} and the equation. For the sake of clarification, suppose that we already have the critical regularity in time, $\partial_t u \in L_q((0,T);(V^1_p)')$. Applying Corollary \ref{coro:Draco_Malfoy} with $\theta=\tfrac{d+1}{d+2}$ gives the interpolation statement
 \begin{align}\label{eq:interpol}
    L_p((0,T);W^1_p) \cap W^1_q((0,T);W^{-1}_q) \xhookrightarrow{c} L_{2q}((0,T);L_{2q}),
\end{align}
whenever $p \geq \tfrac{3d+2}{d+2}$; i.e. \eqref{LH:equation:2} holds in $L_q((0,T);(V^1_p)')$.\\
In the general case $\partial_t u \in L_{1,\loc}((0,T);(V^1_s)')$ the application of Corollary \ref{coro:Draco_Malfoy} yields $\diverg (u \otimes u) \in L_{q_1}((0,T); (V^{1}_{s_1})')$ for some $q_1>1$. Using equation \eqref{LH:equation}, this yields $\partial_t u \in L_{q_1}((0,T); (V^{1}_{s_1})')$ and iterated application of this argument allows bootstrapping up to $L_q((0,T);(V^1_p)')$.
\end{remark}

\section{Properties of solutions to the approximate problem} \label{sec:shearthinning}
This section is devoted to the first part of the proof of Theorem~\ref{main:convergence}.

Recall that the set $U_{\eta}$ contains all functions $u$ with zero spatial average obeying 
\begin{equation*}
        e^{-t/(p \eta)} u \in L_p((0,\infty);V^1_p);
        \quad
        e^{-t/(2\eta)} \partial_t u \in L_2((0,\infty);L_2);
        \quad
       e^{-t/(4\eta)} u \in L_4((0,\infty);V^1_4);
       \quad
       u(0) = u^0_\eta,  
\end{equation*} 
such that, for $u \in U_\eta$, the functional
\begin{equation*}
    I_\eta(u) 
    = 
    \int_0^\infty e^{-t/\eta}
    \int_{\T_d} \tfrac{1}{2} 
    |\partial_t u + \diverg(u\otimes u)|^2 
    + 
    \tfrac{1}{\eta} |W(\epsilon(u))|
    +
    \tfrac{C_4}{4}\abs{\nabla u}^4
    \dd x \dd t
\end{equation*}
is well-defined and finite.
Note that $U_{\eta}$ is an affine space and hence $U_\eta$ is weakly closed. The initial condition is also closed under weak convergence as $u \in C([0,T);L_2({\T_d}))$. This allows for the use of the \emph{direct method in the calculus of variations} to show existence of a minimiser of $I_\eta$.


\subsection{Existence of minimisers and a weak bound} \label{sec:min_thick}

\begin{proposition}[Existence of minimisers] \label{prop:existence}
   Assume that $p > \tfrac{2d}{d+2}$.
   There exists a minimiser $u_\eta$ of $I_{\eta}$ in the space $U_{\eta}$. Moreover, the minimiser satisfies $I_{\eta}(u_{\eta}) \leq C \eta^{-1}$.
\end{proposition}


\begin{proof}
Note that $U_\eta \neq \emptyset$ since the constant (in time) function $\bar{u}(t) = u^0_\eta$ is contained in $U_{\eta}$ and that 
\begin{equation*}
    \begin{split}
        I_{\eta}(\bar u) 
        &= \int_0^{\infty} e^{-t/\eta} 
        \int_{{\T_d}} \tfrac{1}{2} \vert \diverg(u^0_{\eta} \otimes u^0_{\eta}) \vert^2 + \tfrac{1}{\eta} W(\epsilon(u^0_{\eta}))+\tfrac{C_4}{4}\abs{\nabla u^0_\eta}^4 \dd x \dd t \\
        &\leq C\left( \eta \Vert \diverg(u^0_{\eta} \otimes u^0_{\eta}) \Vert_{L_2}^2 + \Vert \epsilon(u^0_{\eta}) \Vert_{L_p}^p +\eta\Vert \nabla u^0_\eta\Vert_{L_4}^4\right) 
        \leq 
        C \frac{1}{\eta}
    \end{split}
\end{equation*}
due to \ref{it:W3} and Assumptions \ref{it:initial2}--\ref{it:initial4} on $u_\eta^0$ from Subsection~\ref{sec:WIDE}. 
We now apply the direct method in the calculus of variations to show existence of a minimiser. First, note that
\begin{equation*}
    0 \leq \inf_{u \in U_{\eta}}
    I_{\eta}(u)  
    \leq 
    I_\eta(\bar{u})
    \leq      
    C \tfrac{1}{\eta}.   
\end{equation*}
This in particular proves the last statement of the proposition.
Moreover, the functional $I_\eta$ is coercive in the sense that if $u_k$ is such that 
    \[
    \sup_{k \in \N} I_{\eta}(u_k) \leq C,
    \]
then 
\begin{equation}\label{eq:coerciveness_bounds}
    \hspace{-1cm}\Vert e^{-t/(p\eta)} u_k \Vert_{L_p((0,\infty);V^1_p)} + \Vert e^{-t/(2\eta)} \partial_t u_k \Vert_{L_2((0,\infty);L_2)} + \Vert e^{-t/(2\eta)} (u_k \cdot \nabla) u_k \Vert_{L_2((0,\infty);L_2)}+\Vert e^{-t/(4\eta)}u_k\Vert_{L_4((0,\infty);V^1_4)} \leq C.
\end{equation}
Here, the first bound follows from assumption \ref{it:W3}, in combination with Korn's and Poincar\'e's inequality. The bound for the nonlinear term follows from the estimate
\begin{align}\label{eq:nonl_control}
    \Vert e^{-t/(2\eta)}(u_k \cdot \nabla) u_k\Vert^2_{L_2((0,\infty);L_2)}\le C\Vert e^{-t/(4\eta)} u_k\Vert^4_{L_4((0,\infty);V^1_4)}
    \leq C.
\end{align}
Thus, taking a minimising sequence $u_k$, there exists a function $u\in U_\eta$, such that for a subsequence (not relabeled)
\begin{equation*}
    e^{-t/(p\eta)} u_k 
    \lweakto 
    e^{-t/(p\eta)} u 
    \quad 
    \text{in } 
    L_p((0,\infty);V_p^1).
\end{equation*}
The other terms in \eqref{eq:coerciveness_bounds} admit, up to a subsequence, a weak limit. By linearity and the subsequence principle \footnote{\label{note1}The weak topology of $L_2((0,T);L_2)$ and $L_4((0,T);L_4)$ is metrisable on bounded sets. On metric spaces $(X,d)$ we may use the subsequence principle: let $x_k \in X$ be a sequence. If for any subsequence $k_l$ there is a subsequence $k_{l_m}$, such that $x_{k_{l_m}} \to x$, as $m \to \infty$, then $x_k \to x$ as $k \to \infty$.} we have
\begin{enumerate} [label=(\roman*)]
        \item $e^{-t/(p\eta)} u_k \lweakto e^{-t/(p\eta)} u \quad$ in $L_p((0,\infty);V_p^1)$;
        \item $e^{-t/(2\eta)} \partial_t u_k \lweakto e^{-t/(2\eta)} \partial_t u \quad$ in $L_2((0,\infty);L_2)$;
        \item \label{iii} $e^{-t/(4\eta)} u_k \lweakto e^{-t/(4\eta)} u \quad$ in $L_4((0,\infty);V^1_4)$;
        \item $e^{-t/(2\eta)} (u_k \cdot \nabla) u_k \lweakto w\quad$ in $L_2((0,\infty);L_2)$.
\end{enumerate}
In order to identify $w= e^{-t/(2\eta)} (u \cdot \nabla) u$, we prove the convergence 
\begin{equation} \label{proof:closedness1}
     (u_k \cdot \nabla) u_k 
     \lweakto 
     (u \cdot \nabla) u \quad \text{in } L_1((0,T);W^{-1}_{\alpha/2})
\end{equation}
for some $\alpha>2$. Together with the uniqueness of limit functions, this already implies that
\begin{equation*}
    (u_k \cdot \nabla) u_k 
    \lweakto 
    w
    =
    (u \cdot \nabla) u 
    \quad \text{in } L_2((0,T);L_2).
\end{equation*}
Observe that due to Corollary \ref{coro:Draco_Malfoy} we have 
\begin{equation*}
    L_4((0,T);V^1_4) \cap W^{1}_2((0,T);L_2)
    \xhookrightarrow{c}
    L_2((0,T);L_{\alpha})
\end{equation*}
for some $\alpha >2$ and therefore
\begin{equation*}
    u_k \longrightarrow u \quad \text{in } L_2((0,T);L_{\alpha}).
\end{equation*} 
Consequently, we obtain
    \[
        u_{k} \otimes u_{k} \longrightarrow u \otimes u \quad \text{in } L_1((0,T);L_{\alpha/2})
    \]
and, therefore,
\[
    (u_k \cdot \nabla) u_k=\diverg (u_k\otimes u_k) \longrightarrow (u \cdot \nabla) u \quad \text{in } L_1((0,T);W^{-1}_{\alpha/2}).
\]
This proves \eqref{proof:closedness1}.    

We are left with showing that the functional $I_\eta$ is weakly lower-semicontinuous. This is ensured by the convexity of the functions $s \mapsto \vert s \vert^2$, $\xi \mapsto W(\xi)$, and $\xi \mapsto \vert \xi \vert^4$. That is, we have
    \[
        I_{\eta}(u) \leq \liminf_{k \to \infty} I_{\eta}(u_k)
    \]
and consequently, if $u_k$ is a minimising sequence, then $u$ is a minimiser.
\end{proof}


\subsection{Euler--Lagrange equations} \label{sec:EL:1}
 The goal of this section is to derive suitable Euler--Lagrange equations for the functional $I_{\eta}$. Since $U_{\eta}$ is an affine space, the definition of the tangent space is straightforward.
    
\begin{definition}\label{def:TU_thick}
We say that $\phi \in TU_{\eta}$, if
\begin{enumerate} [label=(\roman*)]
    \item \label{it:TU1} $e^{-t/(p\eta)} \phi \in L_p((0,\infty);V_p^1)$;
    \item  \label{it:TU2} $e^{-t/(2\eta)} \partial_t \phi \in L_2((0,\infty);L_2)$;
    \item  \label{it:TU3} $e^{-t/(4\eta)}\phi \in L_4((0,\infty),V^1_4)$;
    \item \label{it:TU4} $\phi(0)=0$;
    \item \label{it:TU5} $ \int_{\T_d} \phi(t,x) \dd x =0$ for a.e. $t \in [0,\infty)$.
\end{enumerate}
\end{definition}

Note that, as in Definition~\ref{def:U}, condition \ref{it:TU4} is a reasonable initial condition as \ref{it:TU2} implies that $\phi \in C^{\alpha}((0,T);L_2)$ for some $\alpha>0$.

Given $u \in U_{\eta}$ and $\phi \in TU_{\eta}$, we obtain that $u+h\phi \in U_\eta$, i.e. $I_{\eta}(u+h\eta)$ is finite for $h \in \R$. This allows us to formally compute the Euler--Lagrange equation corresponding to $I_\eta$.

\begin{lemma}\label{lemma:EL1}
Let $p > \tfrac{2d}{d+2}$ and let $u_{\eta} \in U_{\eta}$ be a minimiser of $I_{\eta}$ and let $\phi \in TU_{\eta}$. Then 
\begin{equation} \label{eq:EL1}
\begin{split}
    0
    =
    \int_0^\infty \int_{\T_d} e^{-t/\eta} 
    \Bigl(&
    \partial_t u_\eta \partial_t \phi  
    + 
    \diverg(u_\eta \otimes u_\eta) \partial_t \phi
    +
    \tfrac{1}{\eta} DW(\epsilon(u_\eta)) : \epsilon(\phi)
    +
    \partial_t u_\eta 
    \bigl(
    \diverg (u_\eta \otimes \phi + \phi \otimes u_\eta)
    \bigr)
    \\
    &
    + \diverg (u_\eta \otimes \phi + \phi \otimes u_\eta) 
    \diverg (u_\eta \otimes u_\eta)
    +C_4\abs{\nabla u_\eta}^2\nabla u_\eta : \nabla \phi\Bigr)\dd x \dd t.
\end{split}
\end{equation}
\end{lemma}

\begin{proof}
As argued above, $u+h\phi\in U_\eta$ for $u\in U_\eta$ and any $\phi \in TU_\eta$, $h \in \R$. Hence, $I_{\eta}(u+h \phi) < \infty$. The minimising property, 
\begin{equation*}
    I_{\eta}(u_{\eta}) 
    \leq 
    I_{\eta}(u_\eta + h \phi), \quad  h \in \R
\end{equation*}
implies that
\begin{equation*}
    \lim_{h \to 0} \frac{1}{h}\bigl(I_{\eta}(u_\eta+h\phi)-I_{\eta}(u_{\eta})\bigr) =0,
\end{equation*}
\schange{provided that this limit exists.}
Indeed, writing 
\begin{align*}
    I_{\eta}(u_{\eta}+h\phi) 
    = 
    \int_0^{\infty} e^{-t/\eta} 
    \int_{{\T_d}} 
    &
    \tfrac{1}{2}\left| \partial_t(u_\eta + h\phi) 
    + 
    \diverg \bigl((u_\eta + h\phi)\otimes(u_\eta+h\phi)\bigr)\right|^2 
    \\
    &
    + \tfrac{1}{\eta} W(\epsilon(u_\eta+h\phi))+\tfrac{C_4}{4} \abs{\nabla (u_
    \eta+h\phi)}^4 \dd x \dd t
\end{align*}
\schange{and using \ref{it:W1}}, we may compute the derivative with respect to $h$ at $h=0$, which leads to \eqref{eq:EL1}.
\end{proof}


If we take a function $\phi \in C_c^{\infty}((0,\infty) \times {\T_d};\R^d)$ that is solenoidal  in the spatial variable and consider $\tilde{\phi} = e^{t/\eta}\phi$, then we have $\tilde{\phi} \in TU_{\eta}$. Thus, we can use $\tilde{\phi}$ as a test function in \eqref{eq:EL1} and formulate the following alternative Euler--Lagrange equation.


\begin{corollary} \label{coro:EL3}
Let $p > \tfrac{2d}{d+2}$, let $u_{\eta} \in U_{\eta}$ be a minimiser of $I_{\eta}$ and let $\phi \in C_c^{\infty}((0,\infty)\times {\T_d};\R^d)$ satisfy $\diverg \phi=0$ for all $t>0$. Then 
\begin{equation} \label{eq:EL3}
\begin{split}
    0
    =
    \int_0^\infty \int_{\T_d} 
    \Bigl(&
    \bigl(\partial_t u_\eta  
    + 
    \diverg(u_\eta \otimes u_\eta)\bigr) \bigl(\partial_t \phi + \tfrac{1}{\eta} \phi\bigr)
    +
    \tfrac{1}{\eta} DW(\epsilon(u_\eta)) : \epsilon(\phi)\\
    &
    +
    \bigl(\partial_t u_\eta +\diverg(u_\eta\otimes u_\eta)\bigr)
    \bigl(
    \diverg (u_\eta \otimes \phi + \phi \otimes u_\eta)
    \bigr)
   +
   C_4\abs{\nabla u_\eta}^2\nabla u_\eta : \nabla \phi
    \Bigr)
    \dd x \dd t.
\end{split}
\end{equation}
\end{corollary}


Our next goal is to improve Corollary \ref{coro:EL3} such that the Euler--Lagrange equation \eqref{eq:EL3} holds true for test functions $\phi$ with weaker space regularity.
In this regard, the following terms are critical. 
First, observe that for $u \in U_\eta$ the dissipation term
\begin{equation} \label{def:A}
    A[u,\phi] \coloneq \int_{{\T_d}} DW(\epsilon(u)) :\epsilon(\phi) \dd x
\end{equation}
is well-defined for any $ \phi \in V^1_p$, such that we may view $A(u) = A[u,\cdot]$ as an element in $(V^1_p)'$. Similarly, we define for $u \in U_\eta$ the three nonlinear terms
\begin{align}
    R_1[u,\phi] 
    &\coloneq \int_{{\T_d}} \bigl(\partial_t u + \diverg ( u \otimes u)\bigr) \diverg(u \otimes \phi) \dd x, 
    \label{def:R11} 
    \\
    R_2[u,\phi] 
    &\coloneq \int_{{\T_d}} \bigl(\partial_t u +  \diverg ( u \otimes u)\bigr) \diverg(\phi \otimes u) \dd x,
    \label{def:R21}
    \\
     R_3[u,\phi] 
    &\coloneq C_4\int_{{\T_d}} \abs{\nabla u}^2\nabla u : \nabla \phi \dd x
    \label{def:R31}.
\end{align}
Recall that, due to Definition~\ref{def:U}, \ref{it:Ueta2} and \ref{it:Ueta3}, we have for almost all times $(\partial_t u + \diverg( u \otimes u)) \in L_2$ and $\vert \nabla u \vert^2 \nabla u \in L_{4/3}$. Using H\"older's inequality, $R_1[u,\phi], R_2[u,\phi]$ and $R_3[u,\phi]$ are thus well-defined, whenever $\phi \in V^1_4$.
Therefore, choosing
\begin{equation} \label{def:gamma}
    \gamma = \max\{ p, 4\}
\end{equation}
ensures that the Euler--Lagrange equation \eqref{eq:EL3} is well-defined for every solenoidal $\phi \in C_c^{\infty}((0,\infty);V^1_\gamma)$.

For convenience, we introduce the notation
\begin{equation}\label{eq:v_eta_def}
    \veta 
    \coloneq
    \partial_t u_{\eta} + (u_{\eta} \cdot \nabla) u_{\eta}
    \in (V^1_\gamma)^\prime.
\end{equation}
Finally, we obtain the following version of the Euler--Lagrange equation, formulated in the dual space $(V^1_\gamma)'$.


\begin{lemma} \label{lemma:EL3}
Let $p > \tfrac{2d}{d+2}$, let $u_\eta \in U_{\eta}$ be a minimiser of $I_{\eta}$, and let $\phi \in C_c^{\infty}((0,\infty);V^1_{\gamma})$. Then 
\begin{equation*}
    \int_0^{\infty} 
    \langle \eta \veta,\partial_t \phi \rangle 
    + 
    \langle \veta,\phi \rangle 
    + A[u_\eta,\phi]
    + 
    \eta \bigl(R_1[u_\eta,\phi] + R_2[u_\eta,\phi] +R_3[u_\eta,\phi]\bigr)\dd t =0.
\end{equation*} 
In particular, $\veta$ is a weak solution to
\begin{align}\label{eq:EL_dual}
    - \eta \partial_t \veta 
    + 
    \veta 
    + 
    A(u_\eta)
    + 
    \eta \bigl(R_1(u_\eta) + R_2(u_\eta)+ R_3(u_\eta)\bigr)
    = 0
\end{align}
as an equation in $(V^1_\gamma)'$.    
\end{lemma}

In the following subsection, this formulation of the Euler--Lagrange equation in the dual space $(V^1_\gamma)'$ will help us to derive uniform estimates on the time derivative $\partial_t u_\eta$ and on the nonlinear term $(u_\eta \cdot \nabla) u_\eta$.


\subsection{Energy inequality, a-priori bounds and proof of Theorem \ref{thm:main}} \label{sec:bounds:shearthinning}

In this subsection, we derive an energy-dissipation inequality for minimisers of $I_{\eta}$ in $U_{\eta}$ and show that they satisfy certain a-priori estimates \emph{without} the weight $e^{-t/\eta}$.

\begin{remark} \label{remark:regularityII}
    In the following Proposition \ref{prop:energyinequality1} we use the identity 
    \begin{equation} \label{trick1}
        \int_{{\T_d}} u \cdot \bigl[ (u \cdot \nabla)u \bigr] \dd x =0
    \end{equation}
    for almost every time for functions $u \in L_4((0,T);V^1_4)$. This is due to the frequently used identity
    \begin{equation*}
        u \cdot \bigl[ (u \cdot \nabla)u \bigr] = \tfrac{1}{2} \diverg ( \vert u \vert^2 u),
    \end{equation*}
    and the integrability of $u$.
\end{remark}

\begin{proposition}[Energy inequality and a-priori estimates]
\label{prop:energyinequality1}
Let $p > \tfrac{2d}{d+2}$. All minimisers $u_{\eta} \in U_{\eta}$ of the functional $I_{\eta}$ have the following properties.
\begin{enumerate}[label=(\roman*)]
    \item \label{prop:bounds1a1} \textbf{Energy inequality. } It holds for all $0\leq T <\infty$ that
    \begin{equation} \label{energy:ineq11}
        E[u_\eta](T)
        +
        \int_0^T \int_{\T_d} (1-e^{-t/\eta}) DW(\epsilon(u_\eta)) : \epsilon(u_\eta) \dd x \dd t
        \leq
        E[u_\eta^0];
    \end{equation}
    \item \label{prop:bounds1b1} \textbf{uniform estimates. } We have $u_{\eta} \in L_{\infty}((0,\infty);L_2)$ and 
    \begin{equation}\label{eq:uetabound}
        \Vert u_{\eta} \Vert_{L_p((0,\infty);V_p^1)} + \Vert u_{\eta} \Vert_{L_{\infty}((0,\infty);L_2)} \leq C;
    \end{equation}
    \item \label{prop:bounds1c1} moreover, we have the following bounds
    \begin{equation*}
        \Vert \partial_t u_\eta \Vert_{L_2((0,\infty);L_2)}^2 + \Vert \nabla u_\eta\Vert_{L_4((0,\infty);L_4)}^4 \leq C \tfrac{1}{\eta}. 
    \end{equation*}
\end{enumerate}
\end{proposition}
We remark that the bound \ref{prop:bounds1c1} directly implies a bound on the nonlinearity, i.e. 
    \begin{equation}\label{eq:weak_bound_nonl}
        \|(u_\eta\cdot \nabla) u_\eta\|_{L_2((0,\infty);L_2)}^2 \leq C \Vert \nabla u_\eta\Vert_{L_4((0,\infty);L_4)}^4 \leq C \tfrac{1}{\eta}.
    \end{equation}


\begin{proof} \textbf{\ref{prop:bounds1a1} Energy inequality. }
In order to derive the energy inequality \eqref{energy:ineq11} it would be desirable to test the equation with $u$ itself; this is however not possible as $u_\eta \notin TU_{\eta}$ (it does not satisfy the initial condition). Instead, we consider 
\begin{equation*}
    \psi(t) \coloneq 
    \begin{cases}
        e^{t/\eta} -1, 
        & 0\leq t <T 
        \\
        e^{T/\eta} -1, 
        & 
        t \geq T
    \end{cases}           
\end{equation*}
and observe that $\phi \coloneq \psi u_\eta \in TU_{\eta}$. 
Therefore, we may plug $\phi= \psi u_\eta$ into the Euler--Lagrange equation \eqref{eq:EL1}.
This yields
\begin{equation}\label{eq:Hans_Sigl1}
\begin{split}
    0
    =&
    \int_0^T \int_{\T_d} e^{-t/\eta}
    \Bigl(
    \psi |\partial_t u_\eta|^2 
    +
    \tfrac{1}{\eta} e^{t/\eta} u_\eta \partial_t u_\eta 
    +
    \tfrac{1}{\eta} \psi DW(\epsilon(u_\eta)) : \epsilon(u_\eta) 
    \\
    & 
    +
    3 \psi \partial_t u_\eta \diverg(u_\eta \otimes u_\eta)
    +
    2 \psi |\diverg(u_\eta \otimes u_\eta)|^2+C_4\psi \abs{\nabla u_\eta}^4
    \Bigr)
    \dd x \dd t&
    \\
    &+
    (e^{T/\eta} - 1)
    \int_T^\infty \int_{\T_d} e^{-t/\eta} 
    \Bigl(
    |\partial_t u_\eta|^2 
    +
    \tfrac{1}{\eta} DW(\epsilon(u_\eta)) : \epsilon(u_\eta)
    \\
    & 
    +
    3 \partial_t u_\eta \diverg(u_\eta \otimes u_\eta)
    +
    2 |\diverg(u_\eta \otimes u_\eta)|^2
    +
    C_4\abs{\nabla u_\eta}^4
    \Bigr)
     \dd x \dd t,
    \end{split}
\end{equation}
where we have used \eqref{trick1}. We estimate, using Young's inequality,
\begin{equation}\label{eq:young}
    \vert \partial_t u_{\eta} \vert^2 
    + 3 \diverg(u_\eta \otimes u_\eta)  \partial_t u_{\eta} 
    \geq 
    \tfrac 12 \vert \partial_t u_{\eta} \vert^2 - \tfrac 92 \vert \diverg(u_\eta \otimes u_\eta) \vert^2\geq
    \tfrac 12 \vert \partial_t u_{\eta} \vert^2 - \tfrac 92 \abs{u_\eta}^2\abs{\nabla u_\eta}^2.
\end{equation}
By the choice $C_4\ge \tfrac 12 (9C_P^2+1)$ we obtain
\begin{equation*}
    \int_{\T_d} \vert \partial_t u_{\eta} \vert^2 
    + 3 \diverg(u_\eta \otimes u_\eta)  \partial_t u_{\eta} +C_4\abs{\nabla u_\eta}^4 \dd x
    \geq 
    \frac 12 \bigl(\Vert\partial_t u_\eta\Vert_{L_2}^2+\Vert \nabla u_\eta\Vert_{L_4}^4\bigr).
\end{equation*}
Together with the assumption \ref{it:W4} on $DW$ this implies
\begin{equation*}
    \begin{split}
        (e^{T/\eta} - 1)&
        \int_T^\infty \int_{\T_d} e^{-t/\eta} 
        \Bigl(
        |\partial_t u_\eta|^2 
        +
        \tfrac{1}{\eta} DW(\epsilon(u_\eta)) : \epsilon(u_\eta)
        +
        3 \partial_t u_\eta \diverg(u_\eta \otimes u_\eta)
        \\
        &\hspace{1.5cm}
        +
        2 |\diverg(u_\eta \otimes u_\eta)|^2
        +
        C_4 \abs{\nabla u_\eta}^4
        \Bigr)
         \dd x \dd t \geq 0.
         \end{split}
    \end{equation*}
Consequently, multiplying \eqref{eq:Hans_Sigl1} by $\eta$, and again using \eqref{eq:young}, leads to the estimate
\begin{align}
    &-\int_0^T \int_{{\T_d}} u_\eta \partial_t u_{\eta}  \dd x \dd t \nonumber\\
    &\geq 
    \int_0^T e^{-t/\eta}\psi \int_{{\T_d}} DW(\epsilon(u_\eta)) \colon \epsilon(u_\eta) + 
    \eta
    \bigl(\tfrac 12\vert \partial_t u_{\eta} \vert^2 + 2\vert \diverg (u_\eta \otimes u_\eta) \vert^2+\tfrac 12\abs{\nabla u_\eta}^4\bigr)
    \dd x \dd t.\label{eq:ddtE}
\end{align}
Due to definition of the space $U_{\eta}$, we have $u_{\eta} \in H^1_{\loc}((0,\infty);L_2)$ and hence, the energy $E[u_\eta](\cdot)$ is differentiable. In light of \eqref{eq:ddtE} this allows us to write
\begin{align} 
    &E[u_\eta](0) - E[u_\eta](T) \nonumber\\
    &\geq 
    \int_0^T (1-e^{-t/\eta})  \int_{{\T_d}} DW(\epsilon(u_\eta)) \colon \epsilon(u_\eta)  
    +
    \eta
    \bigl(\tfrac 12\vert \partial_t u_{\eta} \vert^2 + 2\vert \diverg (u_\eta \otimes u_\eta) \vert^2+\tfrac 12\abs{\nabla u_\eta}^4\bigr) 
    \dd x \dd t,\label{energy:ineq21}
\end{align}
which proves the energy inequality \eqref{energy:ineq11}. 

\medskip
\noindent \textbf{(ii)--(iii) A-priori estimates. }
Again using \ref{it:W4}, the inequality \eqref{energy:ineq21} immediately yields $E[u_{\eta}](T) \leq E[u_\eta](0) = \Vert u^0_\eta \Vert_{L_2}^2$, which is uniformly bounded. Therefore, $u_\eta$ is uniformly bounded in $L_{\infty}((0,\infty);L_2)$.


On the interval $(\eta,\infty)$ we may bound $(1-e^{-t/\eta})$ from below by $(1-e^{-1})$. Hence, \eqref{energy:ineq21} together with the coercivity assumption \ref{it:W4}  yields
\begin{equation} \label{est:etainfty1}
    \Vert u_\eta \Vert_{L_p((\eta,\infty);V^1_p)}^p + \eta \Vert \partial_t u_\eta \Vert_{L_2((\eta,\infty);L_2)}^2 + \eta \Vert (u_\eta \cdot \nabla) u_\eta \Vert_{L_2((\eta,\infty);L_2)}^2 +\eta \Vert \nabla u_\eta\Vert^4_{L_4((\eta,\infty);L_4)}\leq C \Vert u^0_\eta \Vert_{L_2}^2.
\end{equation}
It remains to show such an estimate on the interval $(0,\eta)$. For this purpose, we may simply use the bound on the functional $I_\eta$, derived in Proposition \ref{prop:existence}, and $e^{-t/\eta}\ge \tfrac 1e$ on $(0,\eta)$.  i.e.
\begin{equation}\label{eq:apriori}
    \begin{split}  
        &\eta^{-1} \Vert u_\eta \Vert_{L_p((0,\eta);V^1_p)}^p + \Vert \partial_t u_\eta+  (u_\eta \cdot \nabla) u_\eta \Vert_{L_2((0,\eta);L_2)}^2 +\Vert \nabla u_\eta\Vert^4_{L_4((0,\eta);L_4)}
        \\
        &\leq C \int_0^\eta \int_{\T_d}
        \tfrac{1}{2} \vert \partial_t u_\eta + (u_\eta \cdot \nabla) u_\eta \vert^2+ \tfrac{1}{\eta} W(\epsilon(u_\eta)) 
        +\tfrac{C_4}{4}\abs{\nabla u_\eta}^4
        \dd x \dd t
        \\
        & \leq C I_{\eta}(u_\eta) \leq C \tfrac{1}{\eta},
    \end{split}
\end{equation}
where we use \ref{it:W3}. This finishes the proof of \ref{prop:bounds1b1}. The control 
\begin{align*}
    \Vert (u_\eta\cdot \nabla)u_\eta\Vert_{L_2((0,\eta);L_2)}^2\le C 
    \Vert \nabla u_\eta\Vert_{L_4((0,\eta);L_4)}^4
\end{align*}
in combination with \eqref{eq:apriori} yields the individual bounds in \ref{prop:bounds1c1} on the interval $(0,\eta)$.
\end{proof}


Next, we derive estimates on  $\partial_t u_\eta$ and $(u_\eta \cdot \nabla) u_\eta$ in suitable norms, that are uniform in $\eta$ together with a weak bound on $\partial_{t}^2u_\eta$. To achieve this, we use the dual formulation of the Euler--Lagrange equation in Lemma \ref{lemma:EL3}. 
Recall that we defined $\gamma = \max \{4,p\}$. Further define
\begin{equation} \label{def:tis}
    \tis \coloneq \max\{\gamma, \tfrac{2dp}{dp+2p-2d}\},
\end{equation}
and observe that $\tis = \gamma$ if $p > \tfrac{4d}{d+2}$ and $\tis = \tfrac{2dp}{dp+2p-2d}$ otherwise; the purpose of this choice will become clear below in the beginning of the proof of Proposition \ref{prop:dualbounds}. In particular $\tis \to \infty$, as $p$ approaches $\tfrac{2d}{d+2}$ from above.
Finally, suppose that $s$ is some exponent obeying
\begin{equation} \label{def:s}
    s > \tis.
\end{equation}

\begin{proposition} \label{prop:dualbounds}
Let $p> \frac{2d}{d+2}$ and choose $\tis$ as in \eqref{def:tis}. Then 
\begin{enumerate} [label=(\roman*)]
    \item \label{dualbounds:1} $\Vert (u_\eta \cdot \nabla) u_\eta \Vert_{L_p((0,\infty);(V^1_\tis)')} \leq C$;
    \item  \label{dualbounds:2} for all $T>0$ we have 
   $\Vert \partial_t u_\eta \Vert_{L_{p\wedge q}((0,T);(V^1_\tis)')}  \leq C(T)$;
   \item \label{dualbounds:3} for all $T>0$ we have $\Vert \partial_t^2 u_\eta \Vert_{L_{\frac 43 \wedge q}((0,T);(V^1_\tis)')}  \leq C(T) \tfrac 1\eta$;
    \item \label{dualbounds:4} There is a (non-relabeled) subsequence $u_\eta$ and some $u \in L_p((0,\infty);V^1_p)$  with $\partial_t u \in L_{p \wedge q}((0,T);(V^1_\tis)')$ such that 
    \begin{equation*}
        \begin{cases} u_\eta \lweakto u & \text{in } L_p((0,\infty);V^1_p) \\
        \partial_t u_\eta \lweakto \partial_t u & \text{in }  L_{p \wedge q}((0,T);(V^1_\tis)').
        \end{cases}
    \end{equation*}
    Moreover, we may refine the second weak convergence as follows. There are $\tilde{g}_\eta \in L_q((0,\infty);(V^1_p)')$ and $\tilde{h}_\eta \in L_{s'}((0,T);(V^1_s)')$ with 
    \begin{equation} \label{dualbounds:41}
    \begin{cases} \partial_t u_{\eta} -\partial_t u= \tilde{g}_\eta + \tilde{h}_\eta & \\
        \tilde{g}_\eta \lweakto 0 & \text{in } L_q((0,\infty);(V^1_p)') \\
        \tilde{h}_\eta \longrightarrow 0 &\text{strongly in }  L_{s'}((0,T);(V^1_s)').
        \end{cases} 
    \end{equation}
    \item \label{dualbounds:5}  There is a (non-relabeled) subsequence $u_\eta$ such that 
    \[
    \eta \partial_t^2 u_\eta \lweakto 0 \quad \text{in } L_{\frac{4}{3}\wedge q}((0,T);(V^1_{\tilde{s}})').
    \]
    Moreover, we may refine the convergence as follows. There are  $\tilde{\varg_\eta} \in L_q((0;T);(V^1_p)')$ and $\tilde{\varh_\eta} \in L_{s'}((0,T);(V^1_s)')$
    such that 
    \begin{equation} \label{dualbounds:42}
    \begin{cases} \eta \partial_t^2 u_{\eta}= \tilde \varg_\eta + \tilde \varh_\eta & \\
        \tilde \varg_\eta \lweakto 0 & \text{in } L_q((0,\infty);(V^1_p)') \\
        \tilde \varh_\eta \longrightarrow 0 &\text{strongly in }  L_{s'}((0,T);(V^1_s)').
        \end{cases} 
    \end{equation}
\end{enumerate}
\end{proposition}
Observe that $\gamma=\tilde{s}=p$ if $p\geq 4$, hence every object in above statement is estimated in the corresponding natural $L_p$ and $L_q$ norms, respectively.


\begin{proof}
\noindent \textbf{\ref{dualbounds:1} The nonlinear term.} Recall from Proposition \ref{prop:energyinequality1} that $u_\eta$ is uniformly bounded in $L_{\infty}((0,\infty);L_2) \cap L_p((0,\infty);V^1_p)$. If $p < d$, using the Sobolev embedding $W^1_p({\T_d}) \hookrightarrow L_{\frac{dp}{d-p}}({\T_d})$ and H\"older's inequality in space-time with $L_\infty(L_2)$ and $L_p(L_{\frac{dp}{d-p}})$, yields
\begin{equation*}
    \Vert u_\eta \otimes u_\eta \Vert_{L_p((0,\infty);L_{\frac{2dp}{dp+2d-2p}})} \leq C.
\end{equation*}
Applying the divergence operator and using \eqref{def:tis} ($\tilde{s}$ is the dual exponent of $\tfrac{2dp}{dp+2d-2p}$) we find that
\begin{equation} \label{bound:inertia1}
       \Vert \diverg(u_\eta \otimes u_\eta) \Vert_{L_p((0,\infty);(V^1_{\tis})')} \leq C.
\end{equation}
If $p\geq d$ we can directly use H\"older's inequality to obtain 
\begin{equation*}
    \Vert \diverg(u_\eta \otimes u_\eta) \Vert_{L_p((0,\infty);L_{\frac{2p}{p+2}})}\le  C.
\end{equation*}
and then use the Sobolev embeding $L_{\frac{2p}{p+2}} \hookrightarrow (V^1_\tis)'$ to obtain the desired bound \eqref{bound:inertia1}.

\medskip
\noindent\textbf{\ref{dualbounds:2} Bound on the time derivative. }
To get the desired estimate on $\partial_t u_\eta$, for any $T>0$ we bound
\begin{equation}\label{eq:claimdt}
    \Vert \veta \Vert_{L_{p\wedge q}((0,T);(V^1_\gamma)')}\le C(T), 
    \quad \text{where, as in \eqref{eq:v_eta_def}, } \quad
    \veta = \partial_t u_\eta + (u_\eta \cdot \nabla) u_\eta. 
\end{equation}
 As $\tis \geq \gamma$ one may then use the embedding $(V^1_\gamma)' \hookrightarrow (V^1_\tis)'$ and step \ref{dualbounds:1} to obtain the result. Indeed, if $p\le 2$, it is enough to use \eqref{bound:inertia1} and the triangle inequality, while for $p>2$ we additionally use that $L_p \hookrightarrow L_{p\wedge q}$ on bounded domains to finish the proof. 

We show \eqref{eq:claimdt}. The Euler--Lagrange equation \eqref{eq:EL_dual} implies that, at almost every time $t>0$, we have  
\begin{equation*}
    e^{-t/\eta} 
    \left[- \eta \partial_t \veta + \veta 
    + 
    \eta
    \bigl(R_1(u_\eta)
    + R_2(u_\eta)
    +
    R_3(u_\eta)\bigr) 
    + 
    A(u_\eta) \right]
    = 0.
\end{equation*}
Integrating this equation in time from $t$ to $T$, where $0\leq t \leq T<\infty,$ yields \rchange{for almost every $T\geq t$}
\begin{equation*}
    \veta(t) 
    =
    e^{(t-T)/\eta} \veta(T) 
    - 
    \eta^{-1} \int_t^T e^{(t-s)/\eta} 
    \bigl[
    A(u_\eta)+\eta \bigl(R_1(u_\eta) + R_2(u_\eta)+R_3(u_\eta)\bigr) 
   \bigr] \dd s.
\end{equation*}
Since $\veta \in L_2((0,\infty);L_2)$ by Proposition~\ref{prop:energyinequality1}\ref{prop:bounds1c1} and \eqref{eq:weak_bound_nonl}, and $\gamma \geq 2$ there exists a sequence $T_k \to \infty$, such that \rchange{above equation is true \emph{and}} $\Vert \veta(T_k) \Vert_{(V^1_\gamma)'} \to 0$ and therefore
\begin{equation*}
    \veta(t) 
    = 
    -\eta^{-1} \int_t^\infty e^{(t-s)/\eta} 
    \bigl[
    A(u_\eta)+\eta \bigl(R_1(u_\eta) + R_2(u_\eta) +R_3(u_\eta)\bigr)  
     \bigr]
    \dd s.
\end{equation*}
We introduce
\[
    K_\eta \coloneq 
    \begin{cases}
        \eta^{-1} e^{-t/\eta}, & t \geq 0
        \\
        0, & t < 0 .
    \end{cases}
\]
Extending $A(u_\eta)[s]=R_1(u_\eta)[s]=R_2(u_\eta)[s] = R_3(u_\eta)[s] =0$ on $(-\infty,0)$, we can rewrite $\veta$ as
\begin{equation}\label{eq:v_conv}
    \veta(t) 
    = 
    - K_\eta \ast 
    \bigl[ A(u_\eta)+\eta \bigl(R_1(u_\eta) 
    + 
    R_2(u_\eta)+R_3(u_\eta)
    \bigr) 
   \bigr],   
    \quad t \in \R.
\end{equation}
Note that for any exponent $r \geq 1$ we have
    \[
    \Vert K_\eta \Vert_{L_r(\R)} \leq C \eta^{1/r-1},
    \]
and recall from Proposition \ref{prop:energyinequality1} that
\begin{align*}
    \Vert \partial_t u_\eta + (u_\eta \cdot \nabla) u_\eta \Vert_{L_2((0,\infty);L_2)} \leq C \eta^{-1/2},
    \quad \Vert \nabla u_\eta\Vert_{L_4((0,\infty);L_4)}\le C\eta^{-1/4} \quad \text{and} \quad
    \Vert u_\eta \Vert_{L_p((0,\infty);V^1_p)}  \leq C.
\end{align*}
Due to the bound on $DW$, specified in \ref{it:W4}, $A$ satisfies
\begin{equation}\label{eq:A_bound}
    \Vert A(u_\eta) \Vert_{L_q((0,\infty);(V^1_p)')}
    \leq 
    C \bra{\Vert u_\eta \Vert_{L_p((0,\infty);V^1_p)}+1}
    \leq 
    C.
\end{equation}
Using $\gamma\ge p$ and Young's convolution inequality yields
\begin{equation}\label{eq:A_Lq}
    \Vert K_\eta \ast A(u_\eta) \Vert_{L_q(\R;(V^1_\gamma)')} 
    \leq 
    C\Vert K_\eta \ast A(u_\eta) \Vert_{L_q(\R;(V^1_p)')} 
    \leq 
    \Vert K_\eta \Vert_{L_1(\R)} \Vert A(u_\eta) \Vert_{L_q((0,\infty);(V^1_p)')} \leq C. 
\end{equation}
Moreover, using H\"older's inequality and Proposition \ref{prop:energyinequality1} \ref{prop:bounds1c1}, we have the straightforward dual bounds
\begin{equation}\label{bound:R}
    \Vert R_1 \Vert_{L_{\frac 43}((0,\infty);(V^1_\gamma)^\prime)} 
    + 
    \Vert R_2 \Vert_{L_{\frac 43}((0,\infty);(V^1_\gamma)^\prime)} 
    \leq 
    C \eta^{-\frac 34}, \quad 
    \Vert R_3 \Vert_{L_{\frac 43}((0,\infty);(V^1_\gamma)')} \leq C \eta^{-\frac 14}.
\end{equation}
Therefore, we obtain
\begin{align*}
    \| K_\eta \ast (\eta (R_1+R_2)) \Vert_{L_2(\R;(V^1_\gamma)')}  
    &\leq \eta \Vert K_\eta \Vert_{L_{\frac 43}(\R)} \Vert R_1+R_2 \Vert_{L_{\frac 43}((0,\infty);(V^1_\gamma)')} 
    \leq 
    C \eta^{\frac 34 -\frac 34} \leq C,\\
    \Vert K_\eta \ast (\eta R_3) \Vert_{L_\infty(\R;(V^1_\gamma)')} 
    &\leq 
    \eta \Vert K_\eta \Vert_{L_{4}(\R)} \Vert R_3 \Vert_{L_{\frac 43}((0,\infty);(V^1_q)')} \leq 
    C 
    \eta^{\frac 14-\frac 14} \leq C.
\end{align*}
Combining this and \eqref{eq:A_Lq} in \eqref{eq:v_conv}, and using that $L_\infty,L_q,L_2\hookrightarrow L_{p\wedge q}$,
we finally arrive at
\begin{equation}\label{eq:v_eta_bound}
    \Vert \veta \Vert_{L_{p\wedge q}((0,T);(V^1_\gamma)')} \leq C(T),
    \quad
    T > 0.
\end{equation}
 
\medskip
 \noindent\textbf{\ref{dualbounds:3} Bound on the second time derivative. }
 We use the Euler--Lagrange equation \eqref{eq:EL_dual} in the form
 \begin{align*}
    \partial_t \veta=\tfrac 1\eta\bigl(\veta+A(u_\eta)\bigr)+R_1(u_\eta)+R_2(u_\eta)+R_3(u_\eta),
 \end{align*}
 and the bounds \eqref{eq:v_eta_bound}, \eqref{eq:A_bound}, and \eqref{bound:R} to obtain
 \begin{align*}
     \norm{\partial_t \veta}_{L_{\frac 43\wedge q}((0,T);(V^1_\gamma)')}\le C\tfrac 1\eta.
 \end{align*}
 In view of the definition of $\veta$ it suffices to prove
 \begin{align}\label{eq:claim}
     \norm{\partial_t(\diverg(u_\eta\otimes u_\eta))}_{L_{\frac 43}((0,T);(V^1_\gamma)')}\le C \tfrac 1\eta.
     \end{align}
 To achieve this we combine the estimates 
 \begin{align*}
     \norm{\partial_t u_\eta}^2_{L_2((0,T);L_2)}+\norm{u_\eta}^4_{L_4((0,T);L_4)}\le C\tfrac 1\eta
 \end{align*}
 from Proposition~\ref{prop:energyinequality1} \ref{prop:bounds1c1} to see that
 \begin{align} \label{eq:claim:2}
     \norm{\partial_t u_\eta\otimes u_\eta+u_\eta\otimes \partial_t u_\eta}_{L_{\frac 43}((0,T);L_{\frac 43})}\le C \tfrac{1}{\eta^{3/4}} \le C\tfrac 1 \eta.
 \end{align}
Estimate \eqref{eq:claim} now follows by taking the divergence and using the embedding $(V^1_4)'\hookrightarrow (V^1_\gamma)'$.

\medskip
\noindent\textbf{\ref{dualbounds:4} Weak convergence statements. }
By weak compactness of the involved spaces, it is clear from \eqref{eq:uetabound} and part \ref{dualbounds:2} that there is a (non-relabeled) subsequence $u_\eta$ that converges weakly to some $u$ in $L_p((0,\infty);V^1_p)$ and so that $\partial_t u_\eta$ converges weakly in $L_{p \wedge q}((0,T);(V^1_{\tis})')$. Due to linearity of the derivative, the weak limit of $\partial_t u_\eta$ is $\partial_t u$.
Recalling \eqref{eq:v_conv},  we can split
\begin{equation}\label{eq:v_eta2}
    \veta(t) = -\left[ K_\eta \ast A(u_\eta) \right] - \eta\left[  K_\eta \ast\bigl(R_1(u_\eta)+R_2(u_\eta)+R_3(u_\eta)\bigr)\right],
    \quad
    t \in \R.
\end{equation}
Let $v$ be the weak limit of $\veta$.
Observe that $(-K_\eta \ast A(u_\eta))$ is bounded in $L_q((0,\infty);(V^1_p)')$ by \eqref{eq:A_Lq}. Hence we can extract a subsequence that
converges weakly to some $F \in L_q((0,\infty);(V^1_p)')$, and we define 
\begin{equation*}
    \tilde{g}_\eta \coloneq -K_\eta \ast A(u_\eta) - F,
\end{equation*}
and obtain $\tilde{g}_\eta \lweakto 0$ in $L_q((0,T);(V^1_p)')$, as $\eta \to 0$.

Regarding the second term in \eqref{eq:v_eta2}, slightly modifying the bounds after \eqref{bound:R}, one obtains
\begin{align}\label{eq:obelix}
\Vert \eta K_\eta \ast \left[R_1(u_\eta)+R_2(u_\eta)+R_3(u_\eta)\right]\Vert_{L_{s'}((0,T);(V^1_\gamma)')} \longrightarrow 0, \quad \text{as } \eta \to 0,
\end{align}
for all $T>0$. Consequently, the weak limit of $(-K_\eta \ast A(u_\eta))$ satisfies $F=v$.

Moreover, interpolation by Corollary \ref{coro:Draco_Malfoy} between $W^1_{p \wedge q}((0,T);(V^1_\tis)')$ and $L_p((0,T);V^1_p)$ yields on the one hand
\begin{equation}
    u_\eta \longrightarrow u \quad \text{in } L_{p-\varepsilon}((0,T);W^{1-\varepsilon}_{p-\varepsilon}),
    \quad \text{as } \eta \to 0,
\end{equation}
for any $\eps>0$, and interpolating this with $L_{\infty}((0,T);L_2)$,  on the other hand 
\begin{equation}
    u_\eta \longrightarrow u \quad \text{in } L_r((0,T);L_{2-\varepsilon}), \quad \text{as } \eta \to 0,
\end{equation}
for any $r<\infty$ and $\varepsilon>0$.
Hence, following the same argument as in step \ref{dualbounds:1}, 
\begin{equation*}
    \diverg (u_\eta \otimes u_\eta) \longrightarrow \diverg ( u \otimes u) \quad \text{in } L_{s'}((0,T);(V^1_s)'), \quad\text{as } \eta \to 0.
\end{equation*}
Consequently,
\begin{equation*}
     \tilde{h}_\eta \coloneq -\eta\left[  K_\eta \ast\left[R_1(u_\eta)+R_2(u_\eta)+R_3(u_\eta)\right]\right] - \diverg (u_\eta \otimes u_\eta) + \diverg (u \otimes u) \longrightarrow 0 \quad \text{in } L_{s'}((0,T);(V^1_s)').
\end{equation*}
It remains to check that the sum of $\tilde{g}_\eta$ and $\tilde{h}_\eta$ equals the difference between $\partial_t u_\eta$ and $\partial_t u$.
Using that $F=v$ and thus $v= -K_\eta \ast A(u_\eta) -\tilde{g}_\eta$ yields 
\[
\partial_t u_\eta -\partial_t u= (\veta-v) +\diverg(u \otimes u) - \diverg (u_\eta \otimes u_\eta) = \tilde{g}_\eta + \tilde{h}_\eta.
\]

\medskip
\noindent\textbf{\ref{dualbounds:5} Weak convergence statements: second time derivative. }
We proceed as in \ref{dualbounds:3} and \ref{dualbounds:4}. Indeed, by \ref{dualbounds:3} we can ensure (after taking a subsequence) the existence of a weak limit $\limit \in L_{\frac{4}{3}\wedge q}((0,T);(V^1_{\tilde{s}})')$. Now we use the equation \eqref{eq:EL_dual} for $\partial_t^2 u_\eta$ in the form
\[
\eta \partial_t^2 u_\eta = \bigl[-\eta \partial_t (\diverg (u_\eta \otimes u_\eta))\bigr] + \bigl[\veta + A(u_\eta)\bigr] + \eta\bigl[R_1(u_\eta) + R_2(u_\eta) + R_3(u_\eta)\bigr]. 
\]
We deal with the terms separately as suggested by the brackets. By \eqref{eq:claim:2} we have
\[
\lim_{\eta \to 0} \Vert -\eta \partial_t ( \diverg (u_\eta \otimes u_\eta)) \Vert_{L_{4/3}((0,T);(V^1_4)')} =0.
\]
and by \eqref{bound:R} also
\[
\lim_{\eta \to 0}\Vert\eta[R_1(u_\eta) + R_2(u_\eta) + R_3(u_\eta)] \Vert_{L_{4/3}((0,T);(V^1_\gamma)')} =0.
\]
We dealt with the terms $\veta$ and $A(u_\eta)$ already in the proof of \ref{dualbounds:4}. 
In more detail, we have that
\begin{align}\label{eq:asterix}
\veta + A[u_\eta] = \bigl(-K_\eta \ast A[u_\eta] + A[u_\eta]\bigr) - \eta \bigl(K \ast [R_1(u_\eta)+R_2(u_\eta)+R_3(u_\eta)]\bigr).
\end{align}
Note that $K_\eta \to \delta_0$. Using that for any test function $\psi \in L_p(\R;V^1_p)$
\[
\langle -K_\eta \ast A[u_\eta], \psi \rangle= -\langle A[u_\eta], K_\eta(- \cdot) \ast \psi \rangle,
\]
we obtain $-K_\eta \ast A[u_\eta] + A[u_\eta] \weakto 0$ in $L_q(\R;(V^1_p)')$.
For the second term on the right-hand side of \eqref{eq:asterix}, we refer to \eqref{eq:obelix}.

Combining all these observations we can decompose $\eta \partial_t \veta$ into one part, $\tilde\varg_\eta=-A[u_\eta] \ast K_{\eta} + A[u_\eta]$, that converges weakly to $0$ in $L_q((0,T);(V^1_p)')$, and another part $\tilde \varh_\eta$ that converges strongly to $0$ in $L_{s'}((0,T);(V^1_s)')$.
This shows that $\limit=0$ and consequently \eqref{dualbounds:42}.

\end{proof}

\subsection{Proof of Theorem \ref{thm:main} under an additional Lipschitz bound} \label{sec:prooffinish}

Before addressing the quite involved second part of the proof of Theorem \ref{thm:main}, for the purpose of exposition, we show how the proof finalises if we are provided with an additional $L_\infty$-bound on $\nabla u_\eta$ and the additional assumption $\eta^{1/2} \partial_t u_\eta \to 0$ strongly in $L_2((0,T);L_2)$.
We remark that such a bound is not true in reality, so the following bears no application for the proof of Theorem \ref{thm:main}.

The proof of Theorem \ref{thm:main} mainly consists of two steps:
\begin{enumerate}[label=(\alph*)]
    \item showing that the weak limit $u$ of $u_\eta$ exists and obeys the equation
    \begin{equation*}
        \partial_t u + ( u \cdot \nabla) u = -\nabla \pi + \diverg(\chi) 
    \end{equation*}
    in a suitable weak sense, where $\chi$ is the weak limit of $A(u_\eta)$;
    \item showing that $\diverg \chi = \diverg DW(\epsilon(u))$.
\end{enumerate}
Indeed, the second step is clear if $DW$ is a linear function, meaning that the fluid is Newtonian. However, in our non-Newtonian framework this convergence is not obvious, as nonlinear terms are usually not compatible with weak convergence.

Nevertheless, recall that weak convergence on a bounded domain can be attributed to two effects: oscillations and concentrations. Note that, for a \emph{concentrating} sequence we may still infer $DW(\epsilon(u_\eta)) \weakto DW(\epsilon(u))$ from $\epsilon(u_\eta) \weakto \epsilon(u)$, 
while this is not true for an \emph{oscillating} sequence.

Inspired by the works \cite{BDF,BDS}, we therefore first neglect concentrations by considering a \emph{truncated} sequence $u_\eta^L$ that is uniformly bounded in $L_{\infty}((0,T);V^1_{\infty})$, show that this sequence has no oscillations, and then pass to the limit as $L \to \infty$ in a second step.

To demonstrate our strategy, we first suppose that $u_\eta$ is uniformly bounded in $L_{\infty}((0,T);V^1_{\infty})$, show convergence for those functions and attend to the general case in the following Section \ref{sec:5}.

\begin{lemma} \label{lemma:lukaspodolski}
Let $p > \frac{2d}{d+2}$ and let $u_\eta \in U_\eta$ be a minimiser of $I_\eta$. Then the following holds true:
\begin{enumerate} [label=(\roman*)]
    \item \label{lukaspodolski:1} There exist a subsequence of $u_\eta$ (not relabeled) such that 
    \begin{equation} \label{eq:convergence:chi}
   \begin{cases}
        u_\eta \lweakto u& \text{in } L_p((0,\infty);V^1_p);
        \\
        \partial_t u_\eta \lweakto y &
        \text{in } L_{p\wedge q}((0,T);(V^1_\tis)') \quad \text{for all } T>0;
        \\
        DW(\epsilon(u_\eta)) \lweakto \chi & \text{in } L_q((0,\infty);L_q)
    \end{cases}
\end{equation}
        for some limit functions $u \in L_p((0,\infty);V^1_p),\, y \in L_{2\wedge q}((0,T);(V^1_\tis)')$ and $\chi \in L_q((0,\infty);L_q)$;
        \item \label{lukaspodolski:2} $u$ obeys the equation
            \begin{equation} \label{eq:weird:eq}
                \langle \partial_t u+ (u \cdot \nabla) u, \phi \rangle = -\langle \chi, \epsilon(\phi) \rangle 
              \end{equation}
        for all $\phi \in V^1_\tis$ for almost every $t>0$.
        \end{enumerate}
        If there is $L>0$ such that $\Vert u_\eta \Vert_{L_{\infty}((0,\infty);V^1_{\infty})} \leq L$ and, moreover, $\eta^{1/2} \partial_t u_\eta \to 0$ in $L_2((0,T);L_2)$ for any $T>0$, we have in addition
        \begin{enumerate}[resume,label=(\roman*)]
        \item \label{lukaspodolski:3} $u \in L_{\infty}((0,\infty);L_2)$ and $u$ obeys for all $t\in(0,\infty)$ the energy equality
            \begin{equation} \label{eq:lukaspodolski:3}
                E[u](0) - E[u](t) = \int_0^t \int_{{\T_d}} \chi \colon \epsilon(u) \dd x \dd s;
            \end{equation}
        \item\label{lukaspodolski:4} the limits satisfy $ \diverg \chi = \diverg DW(\epsilon(u))$ in the sense that 
        \[
            \langle \chi, \epsilon(\phi) \rangle = \langle DW(\epsilon(u)), \epsilon(\phi) \rangle 
        \]
        for all $\phi \in V^1_p$ and almost every time $t>0$ and hence $u$ is an energy solution to the non-Newtonian Navier--Stokes system.
    \end{enumerate}
\end{lemma}

\begin{remark}
    As shown by Bathory \& Stefanelli \cite{BS22}, if $p \geq \tfrac{3d+2}{d+2}$, the additional assumptions of an $L_{\infty}$-bound on $\epsilon(u_\eta)$ and the convergence of $\eta^{1/2} \partial_t u_\eta$ can be lifted (as in this case we can test the equation \eqref{eq:weird:eq} for $u$ already with $u$ itself). In the proof of \ref{lukaspodolski:4} we use the uniform $L_{\infty}$-bound once more; this can be avoided with slightly more care, cf. \cite{BS22}.
\end{remark}

\begin{proof}
\textbf{\ref{lukaspodolski:1} and \ref{lukaspodolski:2}. }By Proposition \ref{prop:existence} there exists a sequence of minimisers $u_\eta$ in $U_\eta$. 
Moreover, due to the bounds obtained in Proposition \ref{prop:energyinequality1} and Proposition \ref{prop:dualbounds} we may extract a subsequence $u_\eta$ (not relabeled) such that
\begin{equation} \label{eq:proofmain1}
    \begin{cases}
        u_\eta \lweakto u &\text{in } L_p((0,\infty);V^1_p);
        \\
        \partial_t u_\eta \lweakto y &
        \text{in } L_{p\wedge q}((0,T),(V^1_\tis)') \quad \text{for all } T>0;
        \\
        (u_\eta \cdot \nabla) u_\eta \lweakto w & \text{in } L_p((0,\infty),(V^1_\tis)');
        \\
        DW(\epsilon(u_\eta)) \lweakto \chi & \text{in } L_q((0,\infty);L_q), 
\end{cases}
\end{equation}
for limit elements $u \in L_p((0,\infty);V^1_p),\ y \in L_{p\wedge q}((0,T);(V^1_\tis)')$, $w \in L_p((0,\infty);(V^1_\tis)')$, and $\chi \in L_q((0,\infty);L_q)$. Due to linearity of the derivative and a similar interpolation argument as in Proposition \ref{prop:existence}  we may identify
\begin{equation*}
    y = \partial_t u
    \quad \text{and} \quad
    w = (u\cdot \nabla)u. 
\end{equation*}
Moreover, Proposition \ref{prop:energyinequality1}(ii) yields
\begin{equation*}
    u_\eta \overset{\ast}{\lweakto} u  \quad \text{in } L_{\infty}((0,\infty);L_2).
\end{equation*}
From Lemma \ref{lemma:EL3} we know that any $u_\eta$ obeys the Euler--Lagrange equation
\begin{equation*}
    \int_0^{\infty} 
    \langle \eta \veta,\partial_t \phi \rangle 
    + 
    \langle \veta,\phi \rangle 
    + 
    A[u_\eta,\phi]
    + 
    \eta \bigl(R_1[u_\eta,\phi] 
    + 
    R_2[u_\eta,\phi] 
    + 
    R_3[u_\eta,\phi]\bigr)
    \dd t 
    =
    0
\end{equation*} 
for any $\phi \in C_c^{\infty}((0,T);V^1_\tis)$, where $\veta= \partial_t u_\eta + (u_\eta \cdot \nabla) u_\eta$. Note that \eqref{bound:R} implies 
\begin{equation*}
\begin{split}
    &
    \Vert \eta R_1(u_{\eta}) \Vert_{L_{\frac 4 3}((0,\infty);(V^1_{\gamma})')} 
    + \Vert \eta R_2(u_{\eta}) \Vert_{L_{\frac 4 3}((0,\infty);(V^1_{\gamma})')} 
    \leq 
    C \eta^{1/4}
    \longrightarrow 0, 
    \quad \text{as } \eta \to 0,
    \\
    &
    \|\eta R_3(u_\eta)\|_{L_{\frac 4 3}((0,\infty);(V^1_{\gamma})')}
    \leq
    C \eta^{3/4}
    \longrightarrow 0, 
    \quad 
    \text{as } \eta \to 0.
\end{split}
\end{equation*}
Moreover, we have 
\begin{equation*}
    \|\eta \veta\|_{L_{p\wedge q}((0,T);(V^1_\tis)')}
    \longrightarrow 0,
    \quad \text{as } \eta \to 0,
\end{equation*}
by Proposition~\ref{prop:dualbounds}.
Thus, using weak convergence of 
$\veta$, $\epsilon(u_\eta)$ and $DW(\epsilon(u_\eta))$ we obtain that $u$ satisfies
\begin{equation*}
    \int_0^{\infty} \langle \partial_t u+ (u \cdot \nabla) u,\phi \rangle+ \langle \chi,\epsilon(\phi) \rangle \dd t =0,
    \quad
    \phi \in C_c^{\infty}((0,\infty);V^1_\tis).
\end{equation*}
Consequently, $u$ satisfies equation \eqref{eq:weird:eq} for almost every time $t>0$.

\medskip
\noindent \textbf{\ref{lukaspodolski:3}. } Observe that due to the $L_{\infty}$-bound on $\nabla u_{\eta}$ we can already infer that the weak limit $u$ must also satisfy $\Vert u_\eta \Vert_{L_{\infty}((0,\infty);V^1_{\infty})} \leq L$. Consequently, $u$ is in the dual space of all terms appearing in \eqref{eq:weird:eq} on any interval $(0,t)$, and we may test equation \eqref{eq:weird:eq} with $u$ itself and integrate in time from $0$ to $t$ to obtain
\begin{equation*}
    E[u](0)-E[u](t) 
    = 
    -\int_0^t \langle \partial_t u, u \rangle \dd s = \int_0^t \int_{{\T_d}} \chi \colon \epsilon(u) \dd x \dd s,
\end{equation*}
where we have again used the identity $\langle (u\cdot \nabla)u,u\rangle = 0$ for $u\in V^1_{\infty}$.

\medskip
\noindent \textbf{\ref{lukaspodolski:4}. }
In order to identify the limit $\diverg \chi = \diverg DW(\epsilon(u)) \in L_q((0,\infty);(V^1_p)')$, we apply
Minty's trick with an adjustment to the present setting.

To this end, fix $T>0$, take $\phi \in L_p((0,T);V^1_p)$ and observe that due to convexity of $W$ we have
\begin{equation} \label{eq:Minty:1}
    \begin{split}
        0 &\leq \int_0^T \langle DW(\epsilon(u_\eta)) - DW(\epsilon(\phi)), \epsilon(u_\eta) - \epsilon(\phi) \rangle \dd s \\
        & = \int_0^T \langle DW(\epsilon(u_\eta)), \epsilon(u_\eta) \rangle - \langle DW(\epsilon(u_\eta)), \epsilon(\phi) \rangle - \langle DW(\epsilon(\phi)), \epsilon(u_\eta) \rangle + \langle DW(\epsilon(\phi)), \epsilon(\phi) \rangle 
        \dd s.
    \end{split}
\end{equation}
We handle all four summands on the right-hand side of \eqref{eq:Minty:1} separately: the fourth term depends only on $\phi$.
Using the weak convergence $u_\eta \weakto u$ in $L_p((0,\infty);V^1_p)$, we find for the third term that
\begin{equation*}
    \langle DW(\epsilon(\phi)), \epsilon(u_\eta) \rangle
    \longrightarrow 
    \langle DW(\epsilon(\phi)), \epsilon(u) \rangle, \quad \text{ as } \eta \to 0.
\end{equation*}
 Due to the weak convergence $DW(\epsilon(u_\eta)) \weakto \chi \text{ in } L_q((0,\infty);L_q)$, we see that
\begin{equation*}
    \langle DW(\epsilon(u_\eta)),\epsilon(\phi)\rangle 
    \longrightarrow 
    \langle \chi, \epsilon(\phi)\rangle,
    \quad \text{as } \eta \to 0.
\end{equation*}
For the first term on the right-hand side of \eqref{eq:Minty:1} we claim that 
\begin{equation} \label{claim:baumgart}
    \lim_{\eta \to 0} \int_0^T \langle DW(\epsilon(u_\eta)), \epsilon(u_\eta) \rangle \dd s = \int_0^T \langle \chi, \epsilon(u) \rangle \dd s.
\end{equation}
At this point we explicitly mention that \eqref{claim:baumgart} fails for $p < \tfrac{3d+2}{d+2}$ if the additional $L_{\infty}((0,\infty);V^1_{\infty})$ bound is dropped. In particular, this type of proof for the identification of the limit $\chi$ fails in the setting without Lipschitz bound. However, we assume for the moment that, (under the additional assumption $ u_\eta \in L_{\infty}((0,\infty);V^1_{\infty})$,) the identity \eqref{claim:baumgart} is true. In this case, the proof proceeds in the standard manner:

\noindent Taking the limit $\eta \to 0$ in \eqref{eq:Minty:1}, we may infer that
\[
    0 \leq \int_0^T \langle\chi - DW(\phi), \epsilon(u) - \epsilon(\phi) \rangle \dd s 
\]
for all $\phi \in L_p((0,T);V^1_p)$ and, in particular,
choosing $\phi = u \pm \lambda \tilde\phi$ for $\tilde \phi \in L_p((0,T);V^1_p)$ and $\lambda > 0$, we obtain
\[
    0 \leq  \mp\int_0^T \langle\chi - DW(\epsilon(u \pm \lambda \tilde\phi)), \epsilon(\tilde\phi) \rangle \dd s 
\]
Taking in both cases the limit as $\lambda \to 0$,
we infer 
\[
    0
    = 
    \int_0^T \langle\chi - DW(\epsilon(u)), \epsilon(\tilde \phi) \rangle \dd s,
\]
which is the assertion. That $u$ is an energy solution then follows from \ref{lukaspodolski:3}.
\medskip

\noindent It remains to show the claim \eqref{claim:baumgart}. To this end, we observe the following facts:
\begin{itemize}
    \item As $\Vert u_\eta \Vert_{L_{\infty}((0,T);V^1_{\infty})} \leq L$ we have $\eta^{1/2} \diverg(u_\eta \otimes u_\eta) \to 0$ in $L_2((0,T);L_2)$. Moreover, we assumed that $\eta^{1/2} \partial_t u_\eta \to 0$ in $L_2((0,T);L_2)$.
    \item Since  $\Vert u_\eta \Vert_{L_{\infty}((0,T);V^1_{\infty})} \leq L$, we have $u_\eta \weakto u$ in $L_r((0,T);V^1_r)$ for any $r<\infty$, and in combination with \ref{prop:dualbounds}\ref{dualbounds:2}, $u_\eta \weakto u$ in $W^1_{p\wedge q}((0,T);(V^1_\tis)')$. Therefore, by interpolation (cf. Corollary~\ref{coro:Draco_Malfoy}), we have $u_\eta \to u$ strongly in $L_2((0,\infty);L_2)$ and, in particular, up to extraction of a subsequence, $E[u_\eta](s) \to E[u](s)$ for almost every $s \in (0,\infty)$.
\end{itemize}
Now consider 
\[
    \mathcal D \coloneq \int_0^T \langle DW(\epsilon(u_\eta)), \epsilon(u_\eta) \rangle \dd t.
\]
Using the Euler--Lagrange equation/energy equality \eqref{eq:Hans_Sigl1} we may rewrite $\mathcal D$ as 
\begin{align*}
    \mathcal D 
    &= E[u_\eta](0) - E[u_\eta](T) \\
    &\quad- \eta \int_0^{\infty} \psi(t) e^{-t/\eta} \int_{{\T_d}} 3 \partial_t u_\eta \diverg( u_\eta \otimes u_\eta)+ 2 |\diverg (u_\eta \otimes u_\eta)|^2 +  C_4 \vert \nabla u_\eta \vert^4 
    + |\partial_t u_\eta|^2
    \dd x \dd t \\
    &\quad+ \int_0^T e^{-t/\eta} \bigl(\langle DW(\epsilon(u_\eta)), \epsilon(u_\eta) \rangle \bigr) \dd t \\
    &\quad- \int_T^{\infty} (e^{T/\eta}-1) e^{-t/\eta}  \langle DW(\epsilon(u_\eta)), \epsilon(u_\eta) \rangle \dd t \\
    &\eqcolon \mathrm{(I)} + \mathrm{(II)} + \mathrm{(III)} + \mathrm{(IV)}.
\end{align*}
Due to the previous observation we know that $\mathrm{(I)}$ converges to $E[u](0)-E[u](T)$ for almost every $T>0$, and that $\mathrm{(II)}$ converges to zero. Moreover, the $L_{\infty}$-bound of $\epsilon(u_\eta)$ (in space-time) and the uniform bound on $\partial_t u_\eta$ in $L_{s'}((0,T);(V^1_s)')$ imply together with the use of H\"older's inequality that $\mathrm{(III)}$ and $\mathrm{(IV)}$ tend to zero as $\eta \to 0$.
Therefore, 
    \[
    \lim_{\eta \to 0}\int_0^T \langle DW(\epsilon(u_\eta)), \epsilon(u_\eta) \rangle \dd t = E[u](0) - E[u](T) = \int_0^T \langle \chi, \epsilon(u) \rangle \dd t,
    \]
where we use the energy equality \eqref{eq:lukaspodolski:3} for the limit $u$. This shows \eqref{claim:baumgart} and the proof is complete.
\end{proof}

\section{Solenoidal Lipschitz truncation \& the proof of Theorem \ref{thm:main} under no additional assumptions.} \label{sec:5}

This section is devoted to the proof of Theorem \ref{thm:main}. First, we may show (cf. Lemma \ref{lemma:lukaspodolski} \ref{lukaspodolski:1}, \ref{lukaspodolski:2}) that the sequence $u_\eta$ of minimisers converges weakly to a solution of

\begin{equation*}
\begin{cases}
    \partial_t u + (u \cdot \nabla) u = - \nabla \pi + \diverg \chi \\
    \diverg u=0, 
\end{cases}
\end{equation*}
and satisfies an appropriate energy inequality involving the weak limit $\chi$ of $DW(\epsilon(u_\eta))$. 
The main difficulty is to identify the weak limit $\diverg \chi = \diverg DW(\epsilon(u))$.
In Lemma \ref{lemma:lukaspodolski} we have seen that this is true under the additional uniform regularity assumption $\Vert u_\eta \Vert_{L_\infty((0,\infty);V^1_\infty)}\le L$, but this is in general not satisfied.

Instead, recall that $DW(\cdot)$ is nonlinear (apart from the case of Newtonian fluids) and hence not compatible with oscillations of $\epsilon(u_\eta)$. In Lemma \ref{lemma:lukaspodolski} we demonstrated that, when ruling out concentration effects, we directly obtain $\chi= DW(\epsilon(u))$ and, as we see in Section \ref{sec:definitely_Bergdoktor}, the sequence $u_\eta$ is in fact strongly convergent in $L_p((0,\infty);V^1_p)$.

To obtain the result, we aim to apply Minty's trick as in Lemma \ref{lemma:lukaspodolski}. 
This is however not possible since we cannot test the limiting equation with $u$ itself and therefore do not get an energy equality. To handle this problem we follow the idea of \cite{BDF}. That is, we introduce truncated sequences $u_\eta^L$ and $u^L$ and show that the differences  $(u_\eta^L-u_\eta)$ and $(u^L-u)$ have suitable regularity. This allows us to test the Euler--Lagrange equation  for $u_\eta$ and the limiting equation for $u$ with $(u_\eta^L-u_\eta)$ and $(u^L-u)$, respectively. In the following chapter, we state the abstract truncation result and use it to prove the convergence of the nonlinear viscosity term. Due to its technicality, we postpone the proof of the truncation statement to Section \ref{sec:truncation}.

The major obstruction in the proof compared to the proof of Lemma \ref{lemma:lukaspodolski} is to show the convergence 
\[
    \lim_{\eta \to 0} \int_0^t \int_{{\T_d}} \partial_t u_\eta \cdot u_\eta^L \dd s = \int_0^t \int_{{\T_d}} \partial_t u \cdot u^L \dd s.
\]
of the energy terms. 
Recall that the according convergence in Lemma \ref{lemma:lukaspodolski} was obtained by using $\partial_t u_\eta u_\eta = \tfrac{1}{2} \partial_t |u_\eta|^2$ and the fundamental theorem to observe that $E[u_\eta] \to E[u]$ pointwise almost everywhere.
This argument cannot be directly employed here.
The proof is very roughly organised as follows:
\begin{enumerate} [label=(\roman*)]
    \item\label{it:i} we show that $u_\eta$ converges to some $u$ that satisfies a differential equation involving the weak limit $\chi$ of $DW(\epsilon(u_\eta))$;
    \item\label{it:ii} we construct a truncated sequence $u_\eta^L$ that converges to some $u^L$, and use a test function for the Euler--Lagrange equation that is connected to this truncated sequence;
    \item\label{it:iii} we show that we may use Minty's trick by proving convergence of suitable energies.
\end{enumerate}
Throughout this section we assume as previously that the exponents $s$ and $\tis$ are given by 
\begin{equation} \label{eq:def_s}
    s > \tis 
    = 
    \max\{\gamma,\tfrac{2dp}{dp+2p-2d}\} = \max\{p,4,\tfrac{2dp}{dp+2p-2d}\}.
\end{equation}

\subsection{Weak convergence of $u_\eta$}
We accomplish Step \ref{it:i}, i.e. we briefly recall the weak-convergence results of Lemma \ref{lemma:lukaspodolski} for the minimising sequence $u_\eta$ and observe that the weak limit $u$ satisfies a differential equation as well as an energy inequality, both involving the weak limit $\chi$ of the nonlinear term $DW(\epsilon(u_\eta))$. 

\begin{lemma} \label{lemma:jonashector}
Let $p > \frac{2d}{d+2}$ and let $u_\eta \in U_\eta$ be a minimiser of $I_\eta$. Then the following holds true: 
\begin{enumerate} [label=(\roman*)]
    \item \label{jonashector:1} There is a subsequence (not relabeled) $u_\eta$ and a limit $u \in L_p((0,\infty);V^1_p)$ such that  
    \begin{equation} \label{eq:convergence:chi2}
    \begin{cases}
        u_\eta \lweakto u 
        &
        \text{in } L_p((0,\infty);V^1_p);
        \\
        \partial_t u_\eta \lweakto \partial_t u 
        &
        \text{in } L_{p \wedge q}((0,T);(V^1_\tis)') \quad \text{for all } T>0;
        \\
        DW(\epsilon(u_\eta)) \lweakto \chi 
        &
        \text{in } L_q((0,\infty);L_q)
    \end{cases}
    \end{equation}
    for some $\chi \in L_q((0,\infty);L_q)$;
    \item \label{jonashector:3/2} For every $T > 0$ we have 
    \begin{equation*}
        \|\partial_t^2 u_\eta \|_{L_{\tis'}((0,T);(V^1_\tis)')}  \leq C(T) \tfrac 1\eta;
    \end{equation*}
    \item \label{jonashector:2} $u$ satisfies the equation
            \begin{equation} \label{eq:weird:eq2}
                \langle \partial_t u+ (u \cdot \nabla) u, \phi \rangle = -\langle \chi, \epsilon(\phi) \rangle 
              \end{equation}
        for all $\phi \in V^1_\tis$ for almost every $t>0$;
        \item \label{jonashector:3} $u \in L_{\infty}((0,\infty);L_2)$ and $u$ obeys the energy inequality
        \begin{equation} \label{eq:jonashector:3}
            E[u](0) - E[u](t) \geq 0
        \end{equation}
        for almost every $t>0$.
    \end{enumerate}
\end{lemma}
\begin{proof}
    \textbf{\ref{jonashector:1}, \ref{jonashector:3/2} and \ref{jonashector:2}}
    are already shown in Proposition \ref{prop:dualbounds} and Lemma \ref{lemma:lukaspodolski}.\\
    \textbf{\ref{jonashector:3}. } This follows from the energy inequality for $u_\eta$, i.e. 
        \begin{align*}
        E[u_\eta](0) \geq E[u_\eta](t) + \int_0^t (1-e^{-s/\eta})\langle DW(\epsilon(u_\eta)), \epsilon(u_\eta) \rangle \dd s. 
        \end{align*}
    The integral on the right-hand side is non-negative and, by lower semi-continuity of the energy, we have 
    \begin{equation*}
        \liminf_{\eta \to 0} E[u_\eta](t) \geq E[u](t)
    \end{equation*}
    for almost every $t>0$. This yields \eqref{eq:jonashector:3}.
\end{proof}

\subsection{Lipschitz truncation and concentrating sequences}
In this subsection, we formulate the abstract truncation result. Throughout this section $u_\eta \in U_\eta$ is a minimiser of $I_\eta$, 
i.e. it enjoys all the bounds verified in Section \ref{sec:shearthinning}.
By $u$ we denote the weak limit of (a subsequence of) $u_\eta$.

Starting out, let $\hat{u} \in L_p(\R;V^1_p)$ be the function
    \[
        \hat{u}(t) \coloneq 
        \begin{cases} 
        u(t), & t \geq 0, 
        \\ 
        u(-t), & t <0.
        \end{cases}
    \] 
In the following we abuse notation by writing $u$ instead of $\hat{u}$. Let $\phi \in C_c^{\infty}(\R)$, $\mathrm{spt}(\phi) \subset (-1,1)$, be a standard mollifier and $\phi_\eta(t) \coloneq \eta^{-1} \phi(\eta^{-1}t)$.
Observe that
\begin{equation*}
    (\phi_\eta \ast u) \longrightarrow u
    \quad
    \text{strongly in }
    L_p((0,\infty);V^1_p).
\end{equation*} 
Moreover, thanks to Proposition \ref{prop:energyinequality1} \ref{prop:bounds1c1} and Lemma \ref{lemma:jonashector} \ref{jonashector:3/2} we have
    \[
        \Vert \partial_t (\phi_\eta \ast u) \Vert_{L_2((0,\infty);L_2)} \leq C \eta^{-1/2} 
        \quad 
        \text{and}
        \quad
        \Vert \partial_t^2 (\phi_\eta \ast u) \Vert_{L_{s'}((0,T);(V^1_s)')} \leq C \eta^{-1}.
    \]
We now define
\begin{equation} \label{def:w}
    w_\eta \coloneq u_\eta - \phi_\eta \ast u.
\end{equation}
For further reference, we formulate basic properties of $w_\eta$ in the following lemma.
\begin{lemma} \label{prop:weta}
    Let $w_\eta$ be as in \eqref{def:w}. Then
        \begin{enumerate} [label=(\roman*)]
            \item \label{prop:weta1} $w_\eta \weakto 0$ weakly in $L_p((0,\infty);V^1_p)$ and $\int w_\eta(t,x) \dd x =0$ for almost every $t \in \R$;
            \item \label{prop:weta2} $\eta^{1/2} \partial_t w_\eta$ is uniformly bounded in $L_2((0,\infty);L_2)$;
            \item  \label{prop:weta3} for all $T>0$, the sequence $\eta \partial_t^2 w_\eta$ is uniformly bounded in $L_{s'}((0,T);(V^1_s)')$ ;
            \item\label{prop:weta4}  We may write 
                \[
                \partial_t w_\eta = \hat{g}_\eta + \hat{h}_\eta
                \]
            with $\hat{g}_\eta \weakto 0$ weakly in $L_q((0,\infty);(V^1_p)')$ and $\hat{h}_\eta \to 0$ strongly in $L_{s'}((0,T);(V^1_s)')$ for any $T>0$;
            \item \label{prop:weta5}
            We may write 
            \[
            \eta \partial_t^2 w_\eta = \hat{\varg}_\eta + \hat{\varh}_\eta
            \]
            with $\hat{\varg}_\eta \weakto 0$ weakly in $L_q((0,\infty);(V^1_p)')$ and $\hat{\varh}_\eta \to 0$ strongly in $L_{s'}((0,T);(V^1_s)')$ for any $T>0$;
        \end{enumerate}
\end{lemma}

\begin{proof}
\textbf{\ref{prop:weta1}. }
This statement follows from the weak convergence $u_\eta \weakto u$ in $L_p((0,\infty);V^1_p)$ and the strong convergence of $(\phi_\eta \ast u) \to u$ in $L_p((0,\infty);V^1_p)$. The zero spatial average follows from the fact that both $u_\eta$ and $\phi_\eta \ast u_\eta$ have the same spatial average.

\medskip
\noindent\textbf{\ref{prop:weta2} and \ref{prop:weta3}. }
We may employ the bounds obtained in Proposition \ref{prop:energyinequality1} \ref{prop:bounds1c1} and Proposition \ref{prop:dualbounds} \ref{dualbounds:3}, respectively. 

\medskip
\noindent\textbf{\ref{prop:weta4}. } Observe that $(\partial_t (\phi_\eta \ast u) -\partial_t u) \to 0$ strongly in $L_{s'}((0,T);(V^1_s)')$
and that $\partial_t u_\eta -\partial_t u= \tilde{g}_\eta + \tilde{h}_\eta$, where $\tilde{g}_\eta \weakto 0$ in $L_q((0,\infty);(V^1_p)')$ and $\tilde{h}_\eta \to 0$ in $L_{s'}((0,\infty);(V^1_s)')$, cf. Proposition \ref{prop:dualbounds} \ref{dualbounds:4}.

\medskip
\noindent \textbf{\ref{prop:weta5}} We write
\[
\eta \partial_t^2 w_\eta = \eta \partial_t^2 u_\eta - \eta \partial_t \phi_\eta \ast \partial_t u.
\]
According to Lemma \ref{prop:dualbounds} \ref{dualbounds:5} we can split $\eta \partial_t^2 u_\eta = \tilde{\varg}_\eta + \tilde{\varh}_\eta$ with the desired convergence properties. Moreover, observe that 
\[
\Vert \eta \partial_t \phi_\eta \Vert_{L^1(\R)} \leq C 
\quad \text{and} \quad \int_{\R} \eta \partial_t \phi_{\eta} \dd s=0,
\]
i.e. it is a mollifier with zero mass, implying that 
\[
\eta \partial_t \phi_{\eta} \ast \partial_t u \longrightarrow 0 \quad \text{in } L_{s'}((0,T);(V^1_s)')
\]
as $\partial_t u \in L_{s'}((0,T);(V^1_s)')$. Letting $\hat{\varg}_\eta \coloneq \tilde{\varg}_\eta$ and $\hat{\varh}_\eta \coloneq \tilde{\varh}_\eta - \eta \partial_t \phi_\eta\ast \partial_t u$ then proves \ref{prop:weta5}.

\end{proof}

The key auxiliary result in proving Theorem \ref{thm:main} is the following truncation lemma in the spirit of \cite{BDS}. As the proof of Lemma \ref{lemma:trunc} is quite involved, we postpone it to Section \ref{sec:truncation}.

\begin{lemma}[Solenoidal Lipschitz truncation] \label{lemma:trunc}
    Let $w_\eta$ be defined as in \eqref{def:w} and $0<T<\infty$. There exists a constant
    \begin{equation*}
    C' = C'(d, \sup_{\eta >0} \Vert w_\eta \Vert_{L_p((0,T);V^1_p)}, \sup_{\eta >0} \Vert \partial_t w_\eta \Vert_{L_{s'}((0,T);(V^1_s)')}),
    \end{equation*}
    with $s$ defined in \eqref{eq:def_s}, such that the following holds. If $L>C'$, then there exists a sequence $w_\eta^L \in L_\infty((0,T);V^1_\infty)\cap W^1_\infty((0,T);(V^1_r)')$\, for all $1< r <\infty$, such that
        \begin{enumerate} [label=(T\arabic*)]
            \item \label{trunc1}$\Vert w_\eta^L \Vert_{L_\infty((0,T);V^1_\infty)} \leq CL$;
            \item \label{trunc2} $\Vert w_\eta^L \Vert_{W^1_\infty((0,T);(V^1_r)')} \leq CL^{p-1}$;
            \item \label{trunc3} for any $T>0$ we may write $w_\eta^L - w_\eta = G^L_\eta+ H^L_\eta$, where
            \begin{align}
                    & \lim_{L \to \infty} \sup_{\eta>0} \Vert G^L_\eta \Vert_{L_p((0,T);V^1_p)} =0,
                    \label{GLeta} 
                    \\
                    &
                    \lim_{\eta \to 0} \Vert H^L_\eta \Vert_{L_{\expo}((0,T);V^1_{\expo})} =0 \quad \text{for any } \expo \text{ with } \expo<p \text{ and for any } L>C'; 
                    \label{HLeta}
            \end{align}
            Moreover, we can write $w_\eta^L = w_{\eta,1}^L + w_{\eta,2}^L$ such that $w_{\eta,1}^L$ is uniformly (in $\eta$ \emph{and} $L$) bounded in $L_p((0,T);V^1_p)$ and $w_{\eta,2}^L \to 0$ in $L_p((0,T);V^1_p)$ as $\eta \to 0$ (for fixed $L>C'$).
            \item \label{trunc4} for any $T>0$ we may write $\partial_t w_\eta^L - \partial_t w_\eta = g^L_\eta + h^L_\eta$, where
                \begin{align}
                    & \lim_{L \to \infty} \sup_{\eta>0}   \Vert g^L_\eta \Vert_{L_q((0,T);(V^1_p)')} =0, \label{gleta} 
                    \\
                    &\lim_{\eta \to 0} \Vert h^L_\eta \Vert_{L_{\expo'}((0,T);(V^1_\expo)')} =0 \quad \text{for some } \expo<\infty \text{ and for any } L>C'; \label{hleta}
                \end{align} 
            \item \label{trunc4.5} We have $\Vert \eta \partial_t^2 w_\eta^L \Vert_{L_{\infty}((0,T);(V^1_r)')} \leq CL^{p-1}$ and, moreover, we may write $\eta(\partial_t^2 w_\eta^L-\partial_t^2 w_\eta) = \varg_\eta^L + \varh_\eta^L$ with
             \begin{align}
                    &\lim_{L \to \infty} \sup_{\eta>0} \Vert \varg^L_\eta \Vert_{L_q((0,T);(V^1_p)')} =0, \label{vargleta} 
                    \\
                    &\lim_{\eta \to 0} \Vert \varh^L_\eta \Vert_{L_{\expo'}((0,T);(V^1_\expo)')} =0 \quad \text{for some } \expo<\infty \text{ and for any } L>C'; \label{varhleta}
                \end{align} 
            \item \label{trunc5} for fixed $L>0$, we have the uniform bound
            \begin{equation*}
                \|\eta^{1/2} \partial_t w^L_\eta\|_{L_2((0,T);L_2)}
                \leq C,
            \end{equation*}
            with a constant $C > 0$ that does not depend on $\eta$. 
        \end{enumerate}  
\end{lemma}
\begin{remark} 
\begin{enumerate} [label=(\roman*)]
    \item We only state Lemma \ref{lemma:trunc} for finite time to avoid further problems with the integrability of $\partial_t u$ in time. In principle (with a more sophisticated version of \ref{trunc4}), it is also possible to obtain such a statement for $T=\infty$.
    \item Requirements \ref{trunc1} and \ref{trunc2} are not needed in this form, an integrability $w_\eta^L \in L_r((0,T);V^1_r) \cap W^1_\infty((0,T);(V^1_r)')$, for some $r<\infty$ sufficiently large, is enough. In particular, it is sufficient to choose $r$ sufficiently large such that we can test the limiting Navier--Stokes equation with functions in $L_r((0,T);V^1_r)$. 
    \end{enumerate}
\end{remark}

For proving Lemma \ref{lemma:trunc} and for later discussions, the following observation will prove invaluable. Recall that a sequence $f_\eta \in L_p(\Omega)$, $\Omega \subset \R^{d+1}$, is called \textbf{$p$-equi-integrable}, if 
    \[
        \lim_{\varepsilon \to 0} \sup_{\eta} \sup_{\mathcal{L}^n(E) < \varepsilon} \int_E \vert f_\eta \vert^p \dd x =0.
     \]
\begin{lemma}[e.g. \cite{FMP}] \label{lemma:timohorn}
    Suppose that $f_\eta \in L_p(\Omega)$, $\Omega \subset \R^{d+1}$, is bounded. Then there exists a splitting
        \[
        f_\eta = f_\eta^{\equi} + f_\eta^{\conc},
        \]
    such that $f_\eta^{\equi}$ is $p$-equi-integrable and $f_\eta^{\conc} \to 0$ strongly in
    $L_r(\Omega)$ for any $r<p$. Moreover, $f_\eta^{\equi}$ and $f_\eta^{\conc}$ have disjoint support and $\mathcal{L}^{d+1}(\{f_\eta^{\conc} \neq 0\}) \to 0$ as $\eta \to \infty$.
\end{lemma}
Note that this splitting is convenient for the consideration of nonlinear weak limits. In particular the weak limit of $DW(f_\eta)$ agrees with the one of $DW(f_\eta^{\equi})$ since $f_\eta^{\equi}$ and $f_\eta^{\conc}$ have disjoint support and thus for any $\phi\in C^\infty_c(\Omega)$ we have
\begin{align}\label{eq:DW_weakconv}
\begin{aligned}
    \lim_{\eta \to 0} \int_\Omega DW(f_\eta)\phi \dd x&=\lim_{\eta \to 0} \int_\Omega DW(f_\eta^{\equi})\phi\dd x+\lim_{\eta \to 0}\int_\Omega DW(f_\eta^{\conc})\phi\dd x\\
    &=\lim_{\eta \to 0} \int_\Omega DW(f_\eta^{\equi})\phi\dd x,
\end{aligned}
\end{align}
where in the last step we use that $f_\eta^{\conc}\to 0$ strongly in $L_r(\Omega)$ and the growth bound \ref{it:W4} for $DW$. 

\subsection{The truncated sequence $u_\eta^L$ and auxiliary results.}
In this subsection, we accomplish Step \ref{it:ii}. Instead of testing the Euler--Lagrange equation for $u_\eta$ with $\psi u_\eta$, (as done in Proposition \ref{prop:energyinequality1}),
we want to test the equation with a truncated version $\psi u_\eta^L$. To this end, for fixed $L>0$, we define 
    \begin{equation} \label{def:uetaL}
        u_\eta^L = \phi_\eta \ast u + w_\eta^L.
     \end{equation}
 Observe that, for any $\psi \in C_c^{\infty}((0,\infty))$, the function $\psi u_\eta^L$ is a valid test function for the Euler--Lagrange equation \eqref{eq:EL1}. For fixed $L>0$, we denote by $w^L$ the weak limit (up to a subsequence) of $w_\eta^L$ and by $u^L = u +w^L$ the weak limit of $u_\eta^L$. Moreover, by $\chi^L$ we denote the weak limit of $DW(\epsilon(u_\eta^L))$.

\begin{lemma}[Properties of $u^L$ and $\chi^L$] \label{lemma:marcelrisse}
Let $L>0$ and let $u^L$ and $\chi^L$ be as above. Then we have
\begin{enumerate} [label=(\roman*)]
        \item \label{marcelrisse:1} $u^L \to u$ strongly in $L_p((0,T);V^1_p)$, as $L \to \infty$;
        \item \label{marcelrisse:2} $(\partial_t u^L- \partial_t u) \to 0$ strongly in $L_q((0,T);(V^1_p)')$, as $L \to \infty$;
        \item \label{marcelrisse:3} $\chi^L \to \chi$ strongly in $L_q((0,T);L_q)$, as $L \to \infty$.
\end{enumerate}
\end{lemma}
At this point we emphasise 
that \ref{marcelrisse:2} does \emph{not} imply that $\partial_t u^L \to \partial_t u$ in $L_q((0,T);(V^1_p)')$, as both functions are not necessarily elements of this space; we only show that their difference is.
\begin{proof}
    \textbf{\ref{marcelrisse:1}. } It suffices to show that $w^L \to 0$, as $L \to \infty$. To this end, recall that $w_\eta^L- w_\eta =G_\eta^L + H_\eta^L$ as in Lemma \ref{lemma:trunc} \ref{trunc3}. In particular,
    \[
    \wlim_{\eta \to 0} w_\eta^L=\wlim_{\eta \to 0} w_\eta^L -w_\eta = \wlim_{\eta \to 0} G_\eta^L + H_\eta^L,
    \]
    where the weak limit is taken in $L_{\expo}((0,T);V^1_ {\expo})$, and where we used Lemma~\ref{prop:weta}\ref{prop:weta1} in the first equality. Recall that $\|H^L_\eta\|_{L_{\expo}((0,T);V^1_\expo)} \to 0$ strongly as $\eta\to 0$, and hence \[
    \wlim_{\eta \to 0} w_\eta^L  = \wlim_{\eta \to 0} G_\eta^L,
    \]
    which is also initially an equation in $L_{\expo}((0,T);V^1_ {\expo})$. Due to weak lower-semicontinuity of the $L^p((0,T);V^1_p)$ norm, however, we have
        \[
            \Vert w^L \Vert_{L_p((0,T);V^1_p)} \leq \liminf_{ \eta \to 0} \Vert G^L_\eta \Vert_{L_p((0,T);V^1_p)}. 
        \]
    As the right-hand side converges to zero, as $L \to \infty$, by \eqref{GLeta}, we find that $w^L \to 0$ in $L_p((0,T);V^1_p)$, as $L \to \infty$. 

    \noindent \textbf{\ref{marcelrisse:2}. } Observe that 
    \begin{equation*}
        \partial_t u_\eta^L = \partial_t (\phi_\eta \ast u) + \partial_t w_\eta^L.
    \end{equation*}
    Moreover, $\partial_t (\phi_\eta \ast u) \to \partial_t u$ in $L_{p \wedge q}((0,T);(V^1_\gamma)')$.
    Therefore, (up to choosing a subsequence), the difference $(\partial_t u^L-\partial_t u)$ coincides with the weak limit of $\partial_t w_\eta^L$ in $L_{p \wedge q}((0,T);(V^1_\gamma)')$, as $\eta$ tends to zero. Therefore, it suffices to show that
        \[
        \lim_{L \to \infty} \Vert \partial_t w^L \Vert_{L_q((0,T);(V^1_p)')} =0.
        \]
    To see this, observe that due to \ref{trunc4} and Lemma~\ref{prop:weta}\ref{prop:weta4}, $\partial_t w^L = \wlim_{\eta \to 0} g_\eta^L+h_\eta^L = \wlim_{\eta \to 0} g_\eta^L$ (as an equation initially in $L_{\expo'}((0;T);(V^1_\expo)')$ and, due to lower-semicontinuity of the norm, we have
    \[
        \lim_{L \to \infty} \Vert \partial_t w^L \Vert_{L_q((0,T);(V^1_p)')} \leq \lim_{L \to \infty} \sup_{\eta >0} \Vert g_\eta^L \Vert_{L_q((0,T);(V^1_p)')}=0.
    \]
    \noindent \textbf{\ref{marcelrisse:3}. } $\chi$ is the weak limit of    
        $DW(\epsilon(u_\eta))$. Observe that this coincides with the weak limit of $DW(\epsilon(u_\eta)^{\equi})$
    where we use the decomposition
    \begin{equation*}
        \epsilon(u_\eta) = \epsilon(u_\eta)^{\equi} + \epsilon(u_\eta)^{\conc}, 
    \end{equation*}
    cf. Lemma \ref{lemma:timohorn} and \eqref{eq:DW_weakconv}. In view of \eqref{def:uetaL} and \eqref{def:w} we find the decomposition
    \begin{equation*}
        \epsilon(u_\eta^L) = \epsilon(u_\eta) ^{\equi} + \epsilon(u_\eta)^{\conc} + \bigl(\epsilon(w_\eta^L)-\epsilon(w_\eta)\bigr).
    \end{equation*}
    Recall that $\epsilon(u_\eta)^{\conc} \to 0$ in $L_r((0,T);L_r)$ for $r<p$, and that $w_\eta^L-w_\eta = G_\eta^L + H_\eta^L$, where $H_\eta^L \to 0$ in some $L_{\expo}((0,T);V^1_{\expo})$ with $p-1<\expo<p$. Therefore, we obtain the convergence
    \[
    DW(\epsilon(u_\eta)^{\equi} + \epsilon(G_\eta^L)) \lweakto \chi^L
    \]
    and consequently,
    \begin{equation} \label{eq:skhiri}
    \Vert \chi^L-\chi \Vert_{L_q((0,T);L_q))} \leq \liminf_{\eta\to 0}
    \Vert DW(\epsilon(u_\eta)^{\equi} + \epsilon(G_\eta^L)) - DW(\epsilon(u_\eta)^{\equi}) \Vert_{L_q((0,\infty);L_q)} 
    \end{equation}
    We claim that the right-hand side of the above equation converges to zero, as $L \to \infty$. To see this, take $\lambda >0$ and consider each of the sets  
    \begin{equation*}
        X_1
        = 
        \{\vert \epsilon(u_\eta)^{\equi} \vert > \lambda\},
        \quad
        X_2
        = 
        \{ \vert \epsilon(G_\eta^L) \vert > \lambda \},
        \quad \text{and} \quad
        X_3
        =
        (X_1\cup X_2)^C,
    \end{equation*}    
    separately. On $X_1 \cup X_2$ we use the bound
    \begin{equation} \label{X1X2}
    \left \vert DW(\epsilon(u_\eta)^{\equi} + \epsilon(G_\eta^L)) - DW(\epsilon(u_\eta)^{\equi})\right \vert \leq C \left(\vert \epsilon(u_\eta)^{\equi} \vert^{p-1} + \vert  \epsilon(G_\eta^L) \vert^{p-1} \right).
    \end{equation}
    Due to the equi-integrability of $\epsilon(u_\eta)^{\equi}$ and since $G_\eta^L \to 0$ uniformly in $\eta$, as $L \to \infty$, we have
    \begin{equation*} 
        \limsup_{L\to \infty} \liminf_{\eta \to 0} \int_{X_1 \cup X_2} \left \vert DW(\epsilon(u_\eta)^{\equi} + \epsilon(G_\eta^L)) - DW(\epsilon(u_\eta)^{\equi}) \right \vert^q \dd x \leq C(\lambda)
    \end{equation*}
    with $C(\lambda) \to 0$, as $\lambda \to \infty$.
    On the complement we may use uniform continuity of $DW$ on the $2\lambda$-ball and dominated convergence to conclude that
    \begin{align}\label{eq:LetaDW}
   \lim_{L \to \infty} \sup_{\eta >0} \int_{X_3} \left \vert DW(\epsilon(u_\eta)^{\equi} + \epsilon(G_\eta^L)) - DW(\epsilon(u_\eta)^{\equi}) \right \vert^q \dd x =0.
    \end{align}
    This implies 
    \begin{equation} \label{dietermueller}
    \limsup_{L \to \infty}  \liminf_{\eta >0} \Vert DW(\epsilon(u_\eta)^{\equi} + \epsilon(G_\eta^L)) - DW(\epsilon(u_\eta)^{\equi}) \Vert_{L_q((0,\infty);L_q)} \leq C(\lambda).
    \end{equation}
    Letting $\lambda \to \infty$, whence $C(\lambda) \to 0$, we obtain \ref{marcelrisse:3}.
\end{proof}

\subsection{Passing to the limit in the nonlinear viscosity term}
Lemma \ref{lemma:marcelrisse} shows that the approximated $u^L$ and $\chi^L$ enjoy nice convergence properties. To apply Minty's trick as in Lemma \ref{lemma:lukaspodolski} we would like to test the Euler--Lagrange equation with $u_\eta$, which is however not expedient, as it does not converge in the right space. Instead we test the equation with $w_\eta^L$.

\begin{proposition} \label{prop:hansschaefer}
Let $0<T$, let $u_\eta \in U_\eta$ be a sequence of minimisers of $I_\eta$ with weak limit $u \in L_p((0,\infty);V^1_p)$ and let $u_\eta^L$ be as in \eqref{def:uetaL}. Suppose that $\chi$ is the weak limit of $DW(\epsilon(u_\eta))$ in $L_q((0,\infty);L_q)$. Then
\begin{enumerate} [label=(\roman*)]
    \item \label{hs:1} There is a sequence $S_L^1$ that tends to zero as $L \to \infty$ such that 
    \[
    \lim_{\eta \to 0} \int_0^t \langle DW(\epsilon(u_\eta^L)),\epsilon(u_\eta^L)\rangle \dd s = \int_0^t \langle \chi^L, \epsilon(u)\rangle \dd s+ S_L^1,\quad 0<t<T.
    \]
    \item \label{hs:2} For almost every time $t>0$ and all $\phi \in V^1_p$ we have
    \[
    \langle \chi, \epsilon(\phi)\rangle= \langle DW(\epsilon(u)),\epsilon(\phi)\rangle.
    \]
    \item \label{hs:3} For fixed $L>0$ we have 
    \[
    \eta^{1/2} \partial_t w_\eta^L \longrightarrow 0 \quad \text{in } L_2((0,T);L_2).
    \]
\end{enumerate}
\end{proposition}
As a consequence of this result we get that $u$ actually is a Leray--Hopf solution of the non-Newtonian Navier Stokes system.
\begin{corollary} \label{coro:finale}
     Let $u_\eta \in U_\eta$ be a minimiser of $I_\eta$ with weak limit $u \in L_p((0,\infty);V^1_p)$. Then $u$ is a Leray--Hopf solution to the non-Newtonian Navier--Stokes system.
\end{corollary}
We first prove Corollary \ref{coro:finale}, as it finishes the proof of the main theorem.
    \begin{proof}
    Proposition \ref{prop:hansschaefer} and Lemma \ref{lemma:jonashector} show that $u$ is a weak solution to the equation 
        \[
            \langle \partial_t u + (u \cdot \nabla) u, \phi \rangle = -\langle DW(\epsilon(u)),\epsilon(\phi)) \rangle
        \]
    for almost every $t>0$ and for all $\phi \in V^1_\gamma$. It remains to prove that $u$ obeys the energy inequality. \schange{
    For this, we prove first that   \begin{equation} \label{estimate:rs2}
           \liminf_{\eta \to 0} \int_0^t \langle DW(\epsilon(u_\eta)),\epsilon(u_\eta) \rangle \dd s \geq \int_0^t \langle DW(\epsilon(u)),\epsilon(u)\rangle \dd s.
         \end{equation}
    Indeed, observe that $W$ is convex, hence
    \[
    W( \epsilon(u)) \geq W(\epsilon(u_\eta)) + DW(\epsilon(u_\eta)) \cdot (\epsilon(u) - \epsilon(u_\eta)).
    \]
    Rearranging this inequality yields
    \begin{equation*}
         \int_0^t \langle DW(\epsilon(u_\eta)),\epsilon(u_\eta) \rangle \dd s \geq \int_0^t \int_{\T_d} W(\epsilon(u_\eta)) - W(\epsilon(u)) \dd x \dd s + \int_0^t \langle DW(\epsilon(u_\eta)),\epsilon(u)\rangle \dd s.
    \end{equation*}
    Taking the liminf in the above inequality, the first term on the right-hand side is non-negative as $W$ is convex and thus $W(\cdot)$ is weakly lower-semicontinuous. The second part of Proposition \eqref{prop:hansschaefer} shows that the latter summand converges to the right-hand side of \eqref{estimate:rs2}, establishing this inequality.
     }

    This lower-semicontinuity, the fact that $\liminf_{\eta \to 0} E[u_\eta](t) \geq E[u](t)$ and $\lim_{\eta \to 0} E[u](0) = E[u](0)$ and the energy inequality for $u_\eta$, \eqref{energy:ineq11}, yield
        \[
            E[u](t) + \int_0^t \langle DW(\epsilon(u)),\epsilon(u) \rangle \dd s \leq E[u](0),
        \]
    and, consequently, $u$ is a Leray--Hopf solution to the non-Newtonian Navier--Stokes problem. 
\end{proof}

The rest of the effort is to show Proposition \ref{prop:hansschaefer}, in particular \ref{hs:1}.
\begin{proof}[Proof of Proposition \ref{prop:hansschaefer}]
\textbf{On \ref{hs:1}.} First we prove that \ref{hs:1} holds with $\leq$ instead of $=$. We subdivide this proof into multiple smaller steps. \\
\textbf{Step 1:}  We show that
\begin{equation} \label{hs:step1}
    \lim_{\eta\to 0}\int_0^t \langle DW(\epsilon(u_\eta^L)), \epsilon(u_\eta^L) \rangle \dd s  
    =\int_0^t \langle \chi^L,\epsilon(u) \rangle
    +  \lim_{\eta \to 0} \int_0^t  \langle DW(\epsilon(u_\eta^L)),\epsilon(w_\eta^L) \rangle \dd s.
\end{equation}
Indeed, observe that $\epsilon(u_\eta^L) = \epsilon(\phi_\eta \ast u) + \epsilon(w_\eta^L)$. Using that $\phi_\eta \ast u \to u$ \emph{strongly} in $L_p((0,\infty);V^1_p)$ and $DW(\epsilon(u_\eta^L)) \weakto \chi^L$ weakly in $L_p$ yields
\begin{equation*}
    \lim_{\eta \to 0} \int_0^t \langle DW(\epsilon(u_\eta^L)), \epsilon(\phi_\eta \ast u) \rangle \dd s = \int_0^t \langle \chi^L, \epsilon(u) \rangle \dd s.
\end{equation*}
This directly implies the statement of the first step.
\smallskip

\noindent \textbf{Step 2:} We show that there is a sequence $S_L^2$ with $S_L^2 \to 0$ as $L \to \infty$ such that
\begin{equation} \label{hs:step2}
    \lim_{\eta \to 0} \int_0^t  \langle DW(\epsilon(u_\eta^L)),\epsilon(w_\eta^L) \dd s = \lim_{\eta \to 0} \int_0^t \langle DW(\epsilon(u_\eta)),\epsilon(w_\eta^L) \dd s +S_L^2.
\end{equation}
We have $\epsilon(u_\eta^L) -\epsilon(u_\eta) = \epsilon(w_\eta^L) - \epsilon(w_\eta)$. Furthermore, recall the decomposition
\begin{equation*}
    \epsilon(u_\eta) = \epsilon(u_\eta)^{\equi} + \epsilon(u_\eta)^{\conc}.
\end{equation*}
According to the truncation lemma \ref{lemma:trunc} \ref{trunc3} we can further split $w_\eta^L - w_\eta = G_\eta^L + H_\eta^L$, i.e. 
\begin{align*}
    \epsilon(u_\eta^L)&= \bigl(\epsilon(u_\eta)^{\equi} + \epsilon(G_\eta^L) \bigr) 
    + \bigl(\epsilon(u_\eta)^{\conc} + \epsilon(H_\eta^L) \bigr) 
\end{align*}
Arguing as in the proof of Lemma \ref{lemma:marcelrisse} \ref{marcelrisse:3}, as $\epsilon(u_\eta)^{\conc} \to 0$ and $\epsilon(H^L_\eta) \to 0$ in $L_r((0,T);V^1_r)$ for $(p-1)<r<p$,
\begin{align*}
    DW(\epsilon(u_\eta)) - DW(\epsilon(u_\eta)^{\equi}) &\longrightarrow 0 \quad \text{strongly in } L_{\tilde{r}}((0,T);L_{\tilde{r}}) \text{ as } \eta \to 0,~1<\tilde{r}<q \\
    DW(\epsilon(u_\eta^L)) - DW(\epsilon(u_\eta)^{\equi}+\epsilon(G_\eta^L)) &\longrightarrow 0 \quad \text{strongly in }  L_{\tilde{r}}((0,T);L_{\tilde{r}}) \text{ as } \eta \to 0,~1<\tilde{r}<q.
\end{align*}
For fixed $L$, however, $w_\eta^L$ is uniformly bounded in $L_{\infty}((0,T);V^1_{\infty})$ and thus we conclude by using H\"older's inequality
\begin{equation} \label{helmutrahn}
    \begin{split}
      \lim_{\eta \to 0} \int_0^t  \langle DW(\epsilon(u_\eta)),\epsilon(w_\eta^L)\rangle \dd s &=
      \lim_{\eta \to 0} \int_0^t  \langle DW(\epsilon(u_\eta)^{\equi}),\epsilon(w_\eta^L)\rangle \dd s, \\
      \lim_{\eta \to 0} \int_0^t  \langle DW(\epsilon(u_\eta^L)),\epsilon(w_\eta^L)\rangle \dd s &=
      \lim_{\eta \to 0} \int_0^t  \langle DW(\epsilon(u_\eta)^{\equi}+\epsilon(G_\eta^L)),\epsilon(w_\eta^L)\rangle \dd s.
    \end{split}
\end{equation}
We have established in \eqref{dietermueller} that 
\begin{align*}
    \lim_{L \to \infty} \sup_{\eta >0} \Vert DW(\epsilon(u_\eta)^{\equi}) - DW(\epsilon(u_\eta)^{\equi}+\epsilon(G_\eta^L)) \Vert_{L_q((0,T);L_q)} =0
\end{align*}
Now, recall that due to \ref{trunc3}, $w_\eta^L=w_{\eta,1}^L+w_{\eta,2}^L$, where the first term is bounded in $L_p((0,T);V^1_p)$ and the second tends to zero in $L_p((0,T);V^1_p)$ \eqref{dietermueller} implies (with H\"older's inequality)
\begin{align*}
    \lim_{L \to \infty} \lim_{\eta \to 0} \int_0^t  \langle DW(\epsilon(u_\eta)^{\equi}),\epsilon(w_\eta^L) \rangle \dd s
        = 
    \lim_{L \to \infty}
    \lim_{\eta \to 0} \int_0^t  \langle DW(\epsilon(u_\eta)^{\equi}+\epsilon(G_\eta^L)),\epsilon(w_\eta^L) \rangle \dd s.
\end{align*}
With the previous observation \eqref{helmutrahn} this yields \eqref{hs:step2}.

\smallskip
\noindent \textbf{Step 3:} We now use the Euler--Lagrange equation for $u_\eta$ tested with $w_\eta^L$ multiplied with a certain cut-off in time. To this end, take $\psi$ as before
\begin{equation*}
    \psi(s) \coloneq 
    \begin{cases}
        e^{s/\eta} -1, 
        & 0\leq s <t 
        \\
        e^{t/\eta} -1, 
        & 
        s \geq t
    \end{cases}           
\end{equation*}
and consider a cut-off $\zeta \in C_c^{\infty}([0,T);[0,1])$ with $\zeta=1$ on $[0,t)$. Recall that $w_\eta^L \in L_{\infty}((0,T);V^1_{\infty}) \cap W^1_{\infty}((0,T);L_{\infty})$; hence $\psi \zeta w_\eta^L \in TU_\eta$ and we can use it as a test function in \eqref{eq:EL1}, i.e. we obtain
\begin{equation} \label{hs:step3:1}
    \begin{split} 
    0 =& \eta\int_0^{\infty}  e^{-s/\eta} \psi \zeta \langle \partial_t u_\eta + \diverg (u_\eta \otimes u_\eta), \partial_t w_\eta^L \rangle \dd s \\
    & + \eta \int_0^{\infty} e^{-s/\eta}  \partial_t(\psi\zeta) \langle \partial_t u_\eta + \diverg( u_\eta \otimes u_\eta),w_\eta^L \rangle \dd s \\
    & + \int_0^{\infty} e^{-s/\eta} \psi \zeta \langle DW(\epsilon(u_\eta)), \epsilon(w_\eta^L) \rangle \dd s \\
    & +\eta \int_0^{\infty} e^{-s/\eta} \psi \zeta \langle \partial_t u_\eta + \diverg (u_\eta \otimes u_\eta), \diverg(u_\eta \otimes w_\eta^L+w_\eta^L \otimes u_\eta) \rangle \dd s \\
    & + \eta C_4\int_0^{\infty} e^{-s/\eta} \psi \zeta \langle \vert \nabla u_\eta \vert^2 \nabla u_\eta, \nabla w_\eta^L \rangle \dd s \\
    =& \mathrm{(I)} + \mathrm{(II)} + \mathrm{(III)} + \mathrm{(IV)} + \mathrm{(V)}
    \end{split}
\end{equation}
We now argue that \eqref{hs:step3:1} in combination with \eqref{hs:step1} and \eqref{hs:step2} yields the result if we take the limit $\eta \to 0$ and then the limit $L \to \infty$.
\smallskip

\noindent \textbf{Step 3a:} For fixed $L>0$, the terms $\mathrm{(IV)}$ and $\mathrm{(V)}$ vanish as $\eta \to 0$.

To this end recall that $e^{-s/\eta} \psi \zeta$  is bounded by $1$, and from Proposition~\ref{prop:energyinequality1} and Lemma~\ref{lemma:trunc}\ref{trunc1}, that 
\begin{equation} \label{hs:obs1}
    \Vert \partial_t u_\eta + \diverg (u_\eta \otimes u_\eta) \Vert_{L_2((0,T);L_2)} \leq C\eta^{-1/2}, \quad \Vert u_\eta \Vert_{L_4((0,T);V^1_4)} \leq C \eta^{-1/4}, \quad \Vert w_\eta^L \Vert_{L_{\infty}((0,T);V^1_{\infty})} \leq CL,
\end{equation}
such that H\"older's inequality implies
\[
\left \vert \int_0^{\infty} e^{-s/\eta} \psi \zeta\langle \partial_t u_\eta + \diverg (u_\eta \otimes u_\eta), \diverg(u_\eta \otimes w_\eta^L+w_\eta^L \otimes u_\eta) \rangle \dd s \right \vert \leq C L \eta^{-3/4},
\]
i.e. $\mathrm{(IV)} \to 0$ as $\eta \to 0$.
The observations \eqref{hs:obs1} also lead to 
\[
\left \vert  \int_0^{\infty} e^{-t/\eta} \psi \zeta \langle \vert \nabla u_\eta \vert^2 \nabla u_\eta, \nabla w_\eta^L \rangle \dd s \right \vert \leq C L \eta^{-3/4}
\]
via the use of H\"older's inequality and hence $\mathrm{(V)} \to 0$ as $\eta \to 0$.
\smallskip

\noindent \textbf{Step 3b:} For fixed $L$ we can approximate $\mathrm{(III)}$ as follows:
\begin{equation} \label{hs:step3b}
    \lim_{\eta \to 0} \int _0^{\infty} e^{-s/\eta} \psi \zeta \langle DW(\epsilon(u_\eta),\epsilon(w_\eta^L)\rangle \dd s = \lim_{\eta \to 0} \int_0^t  \langle DW(\epsilon(u_\eta), \epsilon(w_\eta^L) \rangle \dd s.
\end{equation}
We only need to show that the difference between both sides tends to zero as $\eta \to 0$. Recalling the definition of $\psi$ and $\zeta$ we obtain that the difference equals
\begin{equation*}
    -\int_0^t e^{-s/\eta} \langle DW(\epsilon(u_\eta),\epsilon(w_\eta^L)\rangle \dd s +(e^{t/\eta}-1) \int_t^T e^{-s/\eta} \zeta(s) \langle DW(\epsilon(u_\eta),\epsilon(w_\eta^L)\rangle \dd s.
\end{equation*}
Now $\epsilon(u_\eta)$ is uniformly bounded in $L_p((0,T);L_p)$ and therefore due to \ref{it:W4}, $DW(\epsilon(u_\eta))$ is uniformly bounded in $L_q((0,T);L_q)$. Moreover, for fixed $L$, $\epsilon(w_\eta^L)$ is uniformly bounded in $L_{\infty}((0,T);L_{\infty})$. Using H\"older's inequality and that
\[
\Vert e^{-s/\eta} \Vert_{L_p((0,t))} \longrightarrow 0
\quad \text{and} \quad \Vert e^{-s /\eta} \zeta \Vert_{L_p((t,T))} \longrightarrow 0, \quad \text{as } \eta \to 0,
\]
implies that the difference tends to zero for fixed $L>0$ as $\eta \to 0$. Therefore, \eqref{hs:step3b} is established.
\smallskip

\noindent \textbf{Step 4:} In this step we handle the terms $\mathrm{(I)}$ and $\mathrm{(II)}$. We have 
\begin{equation} \label{hs:step3c1}
\begin{split}
    \mathrm{(I)} + \mathrm{(II)} 
    &= \eta\int_0^{\infty}  e^{-s/\eta} \psi \zeta \langle \partial_t u_\eta + \diverg (u_\eta \otimes u_\eta), \partial_t w_\eta^L \rangle \dd s \\
    &\quad + \eta \int_0^{\infty} e^{-s/\eta}  \partial_t(\psi\zeta) \langle \partial_t u_\eta + \diverg( u_\eta \otimes u_\eta),w_\eta^L \rangle \dd s \\
    & \geq \eta\int_0^{\infty}  e^{-s/\eta} \psi \zeta \langle \partial_t u_\eta + \diverg (u_\eta \otimes u_\eta) - \partial_t w_\eta^L, \partial_t w_\eta^L \rangle \dd s \\
    &\quad + \eta \int_0^{\infty} e^{-s/\eta}  \partial_t(\psi\zeta) \langle \partial_t u_\eta + \diverg( u_\eta \otimes u_\eta),w_\eta^L \rangle \dd s
\end{split}
\end{equation}
For later reference, let us denote by 
    \begin{equation} \label{def:QL}
    Q_L = \limsup_{\eta \to 0} \eta \int_0^{\infty} e^{-s/\eta} \psi \zeta \langle \partial_t w_\eta^L,\partial_t w_\eta^L \rangle \geq 0,
    \end{equation}
which is exactly the difference in the above inequality when $\eta \to 0$.
Integration by parts in the first integral yields
\begin{align*}
     \mathrm{(I)} + \mathrm{(II)} 
     & \ge \int_0^{\infty} e^{-s/\eta} \psi \zeta  \langle \partial_t u_\eta - \partial_t w_\eta^L + \diverg (u_\eta \otimes u_\eta), w_\eta^L \rangle \dd s \\
     &\quad - \eta \int_0^{\infty} e^{-s/\eta} \psi \zeta \langle (\partial_t^2 u_\eta - \partial_t^2 w_\eta^L), w_\eta^L \rangle \dd s \\
     &\quad -\eta \int_0^{\infty} e^{-s/\eta} \psi \zeta \langle \partial_t (\diverg (u_\eta \otimes u_\eta)), w_\eta^L \rangle \dd s \\
     &\quad +\eta \int_0^{\infty} e^{-s/\eta} \partial_t (\psi \zeta) \langle \partial_t w_\eta^L, w_\eta^L \rangle \dd s \\
     &= \mathrm{(VI)} + \mathrm{(VII)} + \mathrm{(VIII)} + \mathrm{(IX)}.
\end{align*}
Again, we handle the terms separately, taking the limit $\eta \to 0$ and, if needed, then taking the limit $L \to \infty$.
\smallskip

\noindent \textbf{Step 4a:} Recall that for fixed $L$, the family $w_\eta^L$ is uniformly bounded in $L_{\infty}((0,T);V^1_{\infty})$ and converges weakly to $w^L$. Moreover, 
\begin{enumerate} [label=(\roman*)]
    \item $e^{-\cdot/\eta} \psi \zeta \to 1_{(0,t)}$ strongly in $L_s((0,T))$, as $\eta \to \infty$; 
    \item due to compact Sobolev embedding we established that $\diverg (u_\eta \otimes u_\eta) \to \diverg (u \otimes u)$ strongly in $L_{s'}((0,T);(V^1_s)')$;
    \item we have $\partial_t u_\eta = \partial_t (\phi_\eta \ast u) + \partial_t w_\eta$ and hence
    \[
    \partial_t u_\eta - \partial_t w_\eta^L = \partial_t (\phi_\eta \ast u) + \partial_t w_\eta - \partial_t w_\eta^L = \partial_t (\phi_\eta \ast u) - g_\eta^L -h_\eta^L.
    \]
    Recall that $h_\eta^L \to 0$ strongly in $L_{\expo'}((0,T);(V^1_\expo)')$ for some $\expo <\infty$ as $\eta \to 0$ (cf. \ref{trunc4}) and that $\partial_t (\phi_\eta \ast u) \to \partial_t u$ strongly in $L_{s'}((0,T);(V^1_s)')$.
\end{enumerate}
Again, applying H\"older's inequality yields
\begin{align*}
    \lim_{L \to \infty} \lim_{\eta \to 0} \mathrm{(VI)} 
        &= \lim_{L \to \infty} \int_0^t \langle \partial_t u + \diverg(u \otimes u), w^L \rangle \dd s + \lim_{L \to \infty} \lim_{\eta \to 0} \int_0^t \langle g_\eta^L, w_\eta^L \rangle \dd s.
\end{align*}
As $g_\eta^L \to 0$ uniformly in $\eta$ in $L_q((0,T);(V^1_p)')$ as $L \to \infty$, and $w_\eta^L = w_{\eta,1}^L + w_{\eta,1}^L$ (cf. \ref{trunc3}), the second limit is zero as $L \to \infty$.

Now using the equation for $u$ from Lemma~\ref{lemma:jonashector}\ref{jonashector:2} ($w^L$ has sufficient regularity) gives
\begin{align*}
    \lim_{L \to \infty} \lim_{\eta \to 0} \mathrm{(VI)} 
        &= \lim_{L \to \infty} \int_0^t \langle \partial_t u + \diverg(u \otimes u), w^L \rangle \dd s \\
        &= -\lim_{L \to \infty} \int_0^t \langle \chi,\epsilon(w^L) \rangle \dd s \\
        &= 0
\end{align*}
as $w^L \to 0$ strongly in $L_p((0,\infty);V^1_p)$( Lemma \ref{lemma:marcelrisse} \ref{marcelrisse:1}).
We conclude that there is a sequence $S_L^3 \to 0$ as $L \to \infty$ such that 
\begin{align} \label{sl3}
    \lim_{\eta \to 0} \int_0^{\infty} e^{-s/\eta} \psi \zeta \langle \partial_t u_\eta - \partial_t w_\eta^L + \diverg (u_\eta \otimes u_\eta), w_\eta^L \rangle \dd s = S_L^3.
\end{align}

\smallskip
\noindent \textbf{Step 4b:} We estimate $\mathrm{(VII)}$ via the decomposition, i.e. 
\[
\eta(\partial_t^2 u_\eta - \partial_t^2 w_\eta^L) = \eta \partial_t^2 (\phi_\eta \ast u) + \eta (\partial_t^2 w_\eta - \partial_t^2 w_\eta^L) = \eta (\partial_t \phi_\eta \ast \partial_t u) - \varg_\eta^L - \varh_\eta^L.
\]
Now
\begin{enumerate} [label=(\roman*)]
    \item $\eta (\partial_t \phi_\eta \ast \partial_t u) \to 0$ strongly in $L_{s'}((0,T);(V^1_s)')$ (cf. proof of Lemma \ref{prop:weta} \ref{prop:weta5});
    \item $\varg_\eta^L \to 0$ strongly in $L_q((0,T);(V^1_p)')$ uniformly in $\eta$ as $L \to \infty$;
    \item $\varh_\eta^L \to 0$ strongly in $L_{\expo'}((0,T);(V^1_{\expo})')$ for fixed $L$ as $\eta \to 0$.
\end{enumerate}
With the same H\"older estimates as in the previous step we then get that there is a sequence $S_L^4 \to 0$ as $L \to \infty$ such that 
\begin{align} \label{sl4}
    \lim_{\eta \to 0} \mathrm{(VII)}=\lim_{\eta \to 0} - \eta \int_0^{\infty} e^{-s/\eta} \psi \zeta \langle (\partial_t^2 u_\eta - \partial_t^2 w_\eta^L), w_\eta^L \rangle \dd s = S_L^4
\end{align}

\smallskip 
\noindent \textbf{Step 4c:} We have already seen that $\eta \partial_t (\diverg (u_\eta \otimes u_\eta) \to 0$ strongly in $L_{4/3}((0,T);(V^1_s)')$, cf. \eqref{eq:claim:2}. Therefore, again using boundedness of $w_\eta^L$ in $L_{\infty}((0,T);V^1_{\infty})$ and H\"older's inequality we obtain
\[
\lim_{\eta \to 0} 
\mathrm{(VIII)}=\lim_{\eta \to 0}  -\eta \int_0^{\infty} e^{-s/\eta} \psi \zeta \langle \partial_t (\diverg (u_\eta \otimes u_\eta)), w_\eta^L \rangle \dd s =0.
\]

\smallskip 
\noindent \textbf{Step 4d:} 
For $\mathrm{(IX)}$ we get
\begin{align*}
     \eta \int_0^{\infty} e^{-s/\eta} \partial_t(\psi \zeta) \langle \partial_t w_\eta^L, w_\eta^L \rangle \dd s 
    =& \int_0^t  \langle \partial_t w_\eta^L, w_\eta^L \rangle \dd s
    \\
    & + (e^{t/\eta}-1) \int_t^{\infty} e^{-s/\eta} \zeta' \langle \partial_t w_\eta^L, w_\eta^L \rangle \dd s.    
\end{align*}
By the same arguments as before, the second integral tends to zero as $\eta \to 0$. The first integral can be rewritten in terms of the energy:
\[
 \int_0^t  \langle \partial_t w_\eta^L, w_\eta^L \rangle \dd s = E[w_\eta^L](t) - E[w_\eta^L](0).
\]
We have $w_\eta^L \weakto w^L$ weakly (-$*$) both in $L_\infty((0,T);V^1_{\infty})$ and in $W^1_r((0,T);(V^1_{r'})')$ for all $r < \infty$. Interpolation gives $w_\eta^L \to w^L$ strongly in $W^{1/2-\varepsilon}_r((0,T);L_2)$ for any $r<\infty$. Hence, $w_\eta^L \to w^L$ in some $C^{\alpha}((0,T);L_2)$ and therefore
\[
\lim_{\eta \to 0} \int_0^t  \langle \partial_t w_\eta^L, w_\eta^L \rangle \dd s = E[w^L](t) - E[w^L](0) =  \int_0^t  \langle \partial_t w^L, w^L \rangle \dd s 
\]
As $\partial_t w^L \to 0$ in $L_q((0,T);(V^1_p)')$ and $w^L \to 0$ in $L_p((0,T);V^1_p)$ (cf. Lemma \ref{lemma:marcelrisse} \ref{marcelrisse:1} and \ref{marcelrisse:2}), we get convergence of the right-hand-side to zero, i.e.
\begin{align} \label{sl5}
     \lim_{\eta \to 0} \mathrm{(IX)}=\lim_{\eta \to 0} \eta \int_0^{\infty} e^{-s/\eta} \partial_t(\psi \zeta) \langle \partial_t w_\eta^L, w_\eta^L \rangle \dd s = S_L^5
\end{align}
for some sequence $S_L^5 \to 0$ as $L \to \infty$.
\smallskip

\noindent\textbf{Step 5:} Combining all previous steps, and taking the limit in \eqref{hs:step3:1} we realise
\[
\lim_{\eta \to 0} \int_0^{\infty} \langle DW(\epsilon(u_\eta),\epsilon(w_\eta^L) \rangle \dd s = - S_L^3-S_L^4 - S_L^5 - Q_L.
\]
In combination with \eqref{hs:step1} and \eqref{hs:step2} this gives
\begin{equation} \label{hs:step5}
    \lim_{\eta \to 0} \int_0^t \langle DW(\epsilon(u_\eta^L)),\epsilon(u_\eta^L) \rangle \dd s = \int_0^t \langle \chi,\epsilon(u^L) \rangle \dd s + S_L^2 -S_L^3 - S_L^3-S_L^4 - S_L^5 - Q_L.
\end{equation}
Defining $S_L^1= S_L^2 -S_L^3 - S_L^3-S_L^4 - S_L^5$ and using that $Q_L \geq 0$ finally proves \ref{hs:1} with $\leq$.
\medskip

\noindent \textbf{\ref{hs:2}.} We use Minty's trick. For any $\phi \in C_c^{\infty}([0,\infty);V^1_p)$ the monotonicity of $DW$ (i.e. convexity of $W$) implies
    \begin{equation} \label{proof:hs:1}
      0 \leq  \int_0^t \langle DW(\epsilon(u_\eta^L)) -DW(\epsilon(\phi)), \epsilon(u_\eta^L) - \epsilon(\phi) \rangle \dd s.
    \end{equation}
Now, due to Step \ref{hs:1} of the proof we have
    \begin{align}\label{eq:usc}
    \lim_{\eta \to 0}  \int_0^t \langle DW(\epsilon(u_\eta^L)), \epsilon(u_\eta^L) \rangle = \int_0^t \langle \chi^L, \epsilon(u^L) \rangle \dd s + S^1_L-Q_L.
    \end{align}
Using weak convergence of $DW(\epsilon(u_\eta^L))$ to $\chi^L$ and of $\epsilon(u_\eta^L)$ to $\epsilon(u^L)$, we obtain the convergence of the mixed terms
\begin{align}\label{eq:mixedt}
\lim_{\eta \to 0} -\int_0^t \langle DW(\epsilon(u_\eta^L)),\epsilon (\phi)\rangle + \langle DW(\epsilon(\phi),\epsilon(u_\eta^L) \dd s
=
- \int_0^{t}  \langle \chi^L, \epsilon(\phi) \rangle + \langle DW(\epsilon(\phi), \epsilon(u^L) \rangle \dd s.
\end{align}
Combining \eqref{eq:usc} and \eqref{eq:mixedt} in \eqref{proof:hs:1} yields
    \begin{equation*}
      -S^1_L + Q_L \leq  \int_0^t \langle \chi^L -DW(\epsilon(\phi)), \epsilon(u^L) - \epsilon(\phi) \rangle \dd s.
    \end{equation*}
Letting $L \to \infty$, using \emph{strong} convergence of $\chi^L$ and $\epsilon(u^L)$ due to Lemma~\ref{lemma:marcelrisse} and $Q_L \geq 0$, yields
    \begin{equation} \label{proof:hs:2}
        0 \leq \int_0^t \langle \chi -DW(\epsilon(\phi)), \epsilon(u) - \epsilon(\phi) \rangle \dd s.
    \end{equation}
By a density argument, \eqref{proof:hs:2} holds for all $\phi \in L_p((0,\infty);V^1_p)$ and, in particular, for $\phi= u \pm \lambda \tilde{\phi}$ for any $\tilde{\phi} \in L_p((0,\infty);V^1_p)$. Letting $\lambda \to 0$ yields
    \begin{equation} \label{proof:hs:3}
        0 \leq \int_0^t \langle \chi -DW(\epsilon(u)),\epsilon(\tilde{\phi}) \rangle \dd s
        \quad
        \text{and}
        \quad
        0 \geq \int_0^t \langle \chi -DW(\epsilon(u)),  \epsilon(\tilde{\phi}) \rangle \dd s
    \end{equation}
which proves \ref{hs:2}.

\medskip
\noindent \textbf{\ref{hs:3} and \ref{hs:1}. } In \eqref{proof:hs:2} we just used $Q_L\geq 0$. Now plugging in $u=\phi$, however, yields $Q_L \leq 0$, i.e. (as it was already positive) $Q_L=0$. Recalling that 
\[
Q_L =  \limsup_{\eta \to 0} \eta \int_0^{\infty} e^{-s/\eta} \psi \zeta \langle \partial_t w_\eta^L,\partial_t w_\eta^L \rangle,
\]
and proceeding as in Step 3b to replace $e^{-s/\eta} \psi \zeta$ by the characteristic function on $(0,t)$ implies
\[
\lim_{L \to \infty} \limsup_{\eta \to 0} \eta \Vert \partial_t w_\eta^L \Vert_{L_2((0,t);L_2)}^2 =0
\]
and $Q_L \to 0$. By \eqref{hs:step5},  \ref{hs:1} thus holds with equality.
\end{proof}
\section{Energy dissipation, strong convergence and uniqueness}
\label{sec:definitely_Bergdoktor}
\schange{
The previous section dealt with the second part of the proof of Theorem \ref{thm:main}; in particular this second part is concerned with the convergence of the nonlinear term}
\begin{equation} \label{eq:Bergdoktor1}
    DW(\epsilon(u_\eta)) \lweakto DW(\epsilon(u)).
\end{equation}
\schange{Taking a variational viewpoint, and investigating the minimising sequence, this weak convergence is quite revealing. If the fluid is non-Newtonian, any \emph{oscillation} effect will likely destroy \eqref{eq:Bergdoktor1}, whereas \emph{concentration} effects do not influence \eqref{eq:Bergdoktor1}.}

\schange{The PDE side of this phenomenon is well-known and has been investigated both from the side of \emph{existence results} using classical techniques (cf. \cite{Lady1,Lady2,Lady3,Lionsbook}) or, similar to here, truncation methods to rule out oscillations (cf. \cite{FMS,BDF,BDS} etc.) as well as \emph{non-uniqueness results} \cite{BV19,BMS}. Using the approach of the WIDE functional, we are able to recover some of this classical insights via the variational methods. In particular we  demonstrate that the absence of concentrations in our sequence of minimisers directly yields a strong convergence result.}

\begin{lemma} \label{lemma:elmau}
Let $1 < p < \infty$ and suppose that $u_\eta \in L_p((0,\infty);V^1_p)$ is a bounded sequence with $u_\eta \weakto u$ in $L_p((0,\infty);V^1_p)$. Moreover, assume that 
    \begin{enumerate} [label=(\roman*)]
        \item in addition to \ref{it:W1}--\ref{it:W4}, $W$ is strictly convex \footnote{A function $W \colon \R^{d \times d}_{\sym,0} \to \R$ is strictly convex if for all $\xi_1\neq \xi_2$ and $\lambda \in (0,1)$ we have $W(\lambda \xi_1 + (1-\lambda)\xi_2) < \lambda W(\xi_1) + (1-\lambda) W(\xi_2)$}, i.e. $DW$ is strictly monotone;
        \item the dissipation of $u_\eta$ converges to the dissipation of $u$, i.e. for all $0<t<\infty$
            \begin{equation} \label{eq:elmau}
                \lim_{\eta \to 0} \int_0^t \langle DW(\epsilon(u_\eta)),\epsilon(u_\eta) \rangle \dd s 
                =
                \int_0^t \langle DW(\epsilon(u)),\epsilon(u) \rangle \dd s;
            \end{equation}
        \item $DW(\epsilon(u_\eta)) \weakto DW(\epsilon(u))$ in $L_q((0,\infty);L_q)$.
    \end{enumerate} 
    Then $u_\eta \to u$ strongly in $L_p((0,\infty);V^1_p)$.
\end{lemma}

We briefly comment on the additional assumption of \emph{strict convexity}. Assumptions \ref{it:W1}--\ref{it:W4} also allow for a convex potential $W$ with $W(\epsilon)=0$ for $\vert \epsilon \vert <R$, $R>0$. In such a setting, for low strain rate $\epsilon$, the shear stress $\sigma$ is zero. Hence, for low strain rate, the non-Newtonian Navier--Stokes equation is the incompressible Euler equation and it is well known that the Euler equation allows for oscillations, cf. \cite{DL09} (although only with lower regularity). We therefore restrict to strictly convex $W$.

Lemma \ref{lemma:elmau} not only rules out oscillations, but more importantly also shows that there are no concentrations. In other words, we obtain \emph{strong} convergence.
\begin{proof}
    Due to strict convexity of $W$ it suffices to show that 
        \begin{equation} \label{eq:claim:elmau}
        \lim_{\eta \to 0} \int_0^t \langle DW(\epsilon(u_\eta))- DW(\epsilon(u)), \epsilon(u_\eta)-\epsilon(u) \rangle \dd s =0.
        \end{equation}
    To this end, observe that, due to weak convergence,
        \[
        \lim_{\eta \to 0} -\int_0^t \langle DW(\epsilon(u_\eta)), \epsilon(u) \rangle + \langle DW(\epsilon(u)), \epsilon(u_\eta) \rangle \dd s = -2\int_0^t \langle DW(\epsilon(u)), \epsilon(u) \rangle \dd s. 
        \]
    Using \eqref{eq:elmau} yields \eqref{eq:claim:elmau}.
\end{proof}

\begin{theorem}[Strong convergence] \label{prop:strong}
    Suppose that $p > \tfrac{3d+2}{d+2}$ and that $W$ is strictly convex. Let $u_\eta \in U_\eta$ be a minimiser of $I_\eta$. Then $u_\eta$ converges strongly to an energy solution $u$ in $L_p((0,\infty);V^1_p)$. 
\end{theorem}

\begin{proof}
We have established in the previous section that $u_\eta$ converges \emph{weakly} in $L_p((0,\infty);V^1_p)$ to a solution $u$ of the Navier--Stokes system. The regularity of $u$ now allows to test the equation with $u$ itself and to obtain the \emph{energy equality} for $u$, cf. Proposition \ref{prop:energy}. Consequently, passing to the limit $\eta \to 0$ in the energy equality \eqref{eq:Hans_Sigl1}, we find that
\[
    \lim_{\eta \to 0} \int_0^t \langle DW(\epsilon(u_\eta)), \epsilon(u_\eta)\rangle \dd s 
    = 
    \int_0^t \langle DW(\epsilon(u)), \epsilon(u) \rangle \dd s. 
\]
Thus, applying Lemma \ref{lemma:elmau}, we obtain that $u_\eta \to u$ \emph{strongly} in $L_p((0,\infty);V^1_p)$.
\end{proof}

Moreover, one can show that under certain additional assumptions, the solution $u$ obtained in Theorem \ref{prop:strong} is unique. We refer to \cite{MNR93,MNRR} for results on the torus and to \cite{MN01} for results on bounded domains in the case of non-degenerate viscosities.

The exponents covered by Theorem \ref{prop:strong} are, as mentioned in the proof, those for which the energy equality is satisfied, i.e. a classical solution theory exists. For exponents $\tfrac{2d}{d+2}<p<\tfrac{3d+2}{d+2}$ such an energy equality is not automatically satisfied, as the flow might develop anomalous dissipation. The proof of Corollary \ref{coro:finale}, however, reveals for which sequences $u_\eta$ such an anomalous dissipation is to be expected, which is directly linked to strong convergence of $u_\eta$.
\begin{theorem}[Strong convergence under $p$-equi-integrability] \label{strong:equiint}
    Suppose that $p > \tfrac{2d}{d+2}$ and that $W$ is strictly convex. Let $u_\eta \in U_\eta$ be a minimiser of $I_\eta$ and suppose that $\epsilon(u_\eta)$ is $p$-equi-integrable. Then $u_\eta$ converges strongly to a Leray--Hopf solution $u$ in $L_p((0,\infty);V^1_p)$. If, in addition, $\eta^{1/2} \partial_t u_\eta \to 0$ strongly in $L_2((0,T);L_2)$ for all $T>0$, then $u$ is an energy solution.
\end{theorem}

\begin{proof}
Observe that whenever $\epsilon(u_\eta)$ is $p$-equi-integrable, we may interchange the limits
\[
\lim_{\eta \to 0} \lim_{L \to \infty} \int_0^t \langle DW(\epsilon(u_\eta^L), \epsilon(u_\eta^L) \dd s = \lim_{L \to \infty} \lim_{\eta \to 0} \int_0^t \langle DW(\epsilon(u_\eta^L), \epsilon(u_\eta^L) \dd s 
\]
of Proposition 4.7. This then shows
\[
\lim_{\eta \to 0} \int_0^t \langle DW(\epsilon(u_\eta), \epsilon(u_\eta) \rangle \dd s =  \int_0^t \langle DW(\epsilon(u), \epsilon(u) \rangle \dd s.
\]
By Lemma \ref{lemma:elmau} we conclude that $u_\eta \to u$ \emph{strongly} in $L_p((0,\infty);V^1_p)$.

To prove the energy equality, consider the energy equality derived for $u_\eta$, i.e. \eqref{eq:Hans_Sigl1}. On the one hand, note that 
\[
E[u_\eta](t) \longrightarrow E[u](t) \quad \text{for a.e. } t>0,
\]
as, due to interpolation between $L_p((0,T);V^1_p)$ and $W^1_{s'}((0,T);(V^1_s)')$, cf. Corollary \ref{coro:Draco_Malfoy}, we get $u_\eta \to u$ in some $L_r((0,T);L_2)$.
We further assumed that $\eta^{1/2} \partial_t u_\eta \to 0$ in $L_2((0,T);L_2)$. Finally,
\begin{enumerate} [label=(\roman*)]
    \item $\eta^{1/4} \nabla u_\eta$ is bounded in $L_4((0,T);L_4)$;
    \item via H\"older's inequality, using that $u_\eta$ is bounded in $L_{\infty}((0,T);L_2)$ and that $\eta^{1/4} u_\eta$ is bounded in $L_4((0,T);L_{4^\ast})$ (where $4^{\ast}$ is the exponent such that $W^1_4 \hookrightarrow L_{4^{\ast}}$), we obtain that $\eta^{1/4}u_\eta \to 0$ strongly in $L_4((0,T);L_4)$;
    \item consequently, we obtain that $\eta^{1/2} \diverg (u_\eta \otimes u_\eta) \to 0$ strongly in $L_2((0,T);L_2)$.
\end{enumerate}
Now taking the limit $\eta \to 0$ in \eqref{eq:Hans_Sigl1} yields the energy \emph{equality} for $u$ and hence $u$ is a energy solution.
\end{proof}
\begin{remark}[Some comments and open questions]
    \begin{enumerate} [label=(\roman*)]
        \item For $p \geq 4$ we may prove every result \emph{without} additional $L_4$-stabilising term $\tfrac{C_4}{4}$. Essentially, this term might be absorbed into the dissipation. 
        \item The validity of the energy equality for solutions to the Navier--Stokes equations for $p > \tfrac{3d+2}{d+2}$ suggests that in principle the stabiliser, that enforces an $L_{\infty}(L_2)$-bound and an energy inequality for weak solutions of the approximate problem, can be dropped.
        \item For $p<\tfrac{3d+2}{d+2}$ it is still an open problem whether Leray--Hopf solutions are unique, e.g. \cite{BV19b}. In addition, it is highly unclear whether the functional possesses a unique minimiser. For $p > \tfrac{3d+2}{d+2}$, while (for strictly convex $W$) the Leray--Hopf solution is unique (cf. \cite[Section 5.4.1]{MNRR} and \cite{Bulicek}, any functional $I_\eta$ might possess multiple minimisers -- with the distance of solutions to each other tending to zero as $\eta \to 0$.  A natural question is whether the validity (or non-validity) of this can be shown for small $p$.
        \item Even though for \emph{Newtonian} fluids weak convergence of $DW(\epsilon(u))=\epsilon(u)$ is, due to its linearity, trivial, Theorem \ref{strong:equiint} also applies in this case, i.e. for $2$-equi-integrable sequences we also get a stronger convergence result.
        \item We have shown that under certain conditions, minimisers $u_\eta$ of $I_\eta$ converge to Leray--Hopf solutions of the non-Newtonian Navier--Stokes equations. Another natural question is whether the reverse statement is also true, i.e. if Leray--Hopf solutions are approximated by a sequence of minimisers/critical points of the functional.
        \item In this paper we use the truncation statement proved in the following section to show convergence of the nonlinear term. However, due to the variational structure of the problem also other techniques might be appropriate to rule out oscillation effects.
    \end{enumerate}
\end{remark}
\section{The proof of the truncation statement} \label{sec:truncation}
In this section we provide a proof of the truncation statement Lemma \ref{lemma:trunc}. We reformulate it here with abstract assumptions. To this end, assume that $w_\eta$ is a sequence that obeys the following properties
\begin{enumerate} [label=(P\arabic*)]
    \item \label{P:1} $w_\eta \weakto 0$ in $L_p((0,\infty);V^1_p)$;
    \item \label{P:2} $\partial_t w_\eta = \hat{g}_\eta +\hat{h}_\eta$, where $\hat{g}_\eta \weakto 0$ in $L_q((0,T);(V^1_p)')$ and $\hat{h}_\eta \to 0$ in $L_{s'}((0,T);(V^1_s)')$ for some $s<\infty$ and all $T>0$;
    \item \label{P:3} $\eta^{1/2} \partial_t w_\eta$ is uniformly bounded in $L_2((0,\infty);L_2)$;
    \item \label{P:4} $\eta \partial_t^2 w_\eta$ is uniformly bounded in $L_{s'}((0,\infty);(V^1_s)')$ for some $s<\infty$ and $\eta \partial_t^2 w_\eta = \hat{\varg}_\eta +\hat{\varh}_\eta$ with  $\hat{\varg}_\eta \weakto 0$ in $L_q((0,T);(V^1_p)')$ and $\hat{\varh}_\eta \to 0$ in $L_{s'}((0,T);(V^1_s)')$ for some $s<\infty$ and all $T>0$;
    \item \label{P:5} We have $\int_{\T_d} w_\eta(t,x) \dd x =0$ for almost every $t \in [0,\infty)$.
\end{enumerate}
The sequence $w_\eta$ (or, more precisely the sequence $w_\eta$ multiplied with a cut-off in time) constructed in the previous section obeys the properties \ref{P:2}--\ref{P:5}, as we have shown in Lemma \ref{prop:weta} for $w_\eta = u_\eta - \varphi_\eta \ast u$. 

\begin{lemma}[Solenoidal Lipschitz truncation with abstract assumptions] \label{lemma:trunc:v2}
    Let $T > 0, L > C'$ for some fixed $C'>0$ and let $w_\eta$ satisfy \ref{P:1}--\ref{P:5}. For each $s'<r<\infty$ there exists a constant $C=C(d,r)$ and a sequence $w_\eta^L \in L_\infty((0,\infty);V^1_\infty)\cap W^1_\infty((0,\infty);(V^1_r)')$ with $\int_{\T_d} w_\eta^L \dd x =0$ for almost every $t$ such that 
        \begin{enumerate} [label=(T\arabic*)]
            \item \label{trunc1:1}$\Vert w_\eta^L \Vert_{ L_\infty((0,T);V^1_\infty)} \leq CL$;
            \item \label{trunc2:1} $\Vert w_\eta^L \Vert_{W^1_\infty((0,T);(V^1_r)')} \leq CL^{p-1} $;
            \item \label{trunc3:1} we can decompose $w_\eta^L - w_\eta = G^L_\eta+ H^L_\eta$, where 
                \begin{align}
                    &
                    \lim_{L \to 0} \sup_{\eta>0} \Vert G^L_\eta \Vert_{L_p((0,T);V^1_p)} =0, \label{GLeta_v2} 
                    \\
                    &
                    \lim_{\eta \to 0} \sup_{L>0} \Vert H^L_\eta \Vert_{L_{\expo}((0,T);V^1_{\expo})} =0 \quad \text{for all 
                    } \expo \text{ with } \expo<p; \label{HLeta_v2}
                \end{align}
            Moreover, we can write $w_\eta^L = w_{\eta,1}^L + w_{\eta,2}^L$ such that $w_{\eta,1}^L$ is uniformly (in $\eta$ \emph{and} $L$) bounded in $L^p((0,T);V^1_p)$ and $w_{\eta,2}^L \to 0$ in $L^p((0,T);V^1_p)$ as $\eta \to 0$ (for fixed $L>C'$).
            \item \label{trunc4:1} we can decompose $\partial_t w_\eta^L -\partial_t w_\eta = g^L_\eta + h^L_\eta$, where
                \begin{align}
                    &
                    \lim_{L \to 0} \sup_{\eta>0} \Vert g^L_\eta \Vert_{L_q((0,T);(V^1_p)')} =0, \label{gleta_v2} \\
                    &
                    \lim_{\eta \to 0} \sup_{L>C'} \Vert h^L_\eta \Vert_{L_{\expo'}((0,T);(V^1_\expo)')} =0 \quad \text{for some } \expo<\infty; \label{hleta_v2}
                \end{align} 
            \item \label{trunc45:1} we can decompose $\eta(\partial_t^2 w_\eta^L -\partial_t^2 w_\eta) = \varg^L_\eta + \varh^L_\eta$, where
                \begin{align}
                    &
                    \lim_{L \to 0} \sup_{\eta>0} \Vert \varg^L_\eta \Vert_{L_q((0,T);(V^1_p)')} =0, \label{vargleta_v2} \\
                    &
                    \lim_{\eta \to 0} \sup_{L>C'} \Vert \varh^L_\eta \Vert_{L_{\expo'}((0,T);(V^1_\expo)')} =0 \quad \text{for some } \expo<\infty; \label{varhleta_v2}
                \end{align}     
            \item \label{trunc5:1} for fixed $L>C'$, $\eta^{1/2} \partial_t w^L_\eta$ is bounded (uniformly for $\eta>0$) in $L_2((0,T);L_2)$.
        \end{enumerate}  
\end{lemma}

In the decomposition \ref{trunc3:1}, the part $H_\eta^L$ reflects the concentrating part which might still be large in $L_p((0,T);V^1_p)$ but is shown to converge to zero in spaces with less integrability. $G_\eta^L$ reflects the difference between the $p$-equi-integrable part and the truncation. This difference converges to zero as $L \to \infty$.

Let us remark that without the assumptions \ref{P:3} and \ref{P:4}, and the resulting requirements \ref{trunc45:1} \& \ref{trunc5:1}, such a result has been achieved in \cite{BDS}.
Therefore, our approach is heavily inspired by the works \cite{BDF,BDS,DSSV}. We explicitly highlight that, due to the additional requirement of \ref{trunc5:1}, one \emph{cannot} take the parabolic, divergence-free truncation of \cite{BDS} and further mollify in time to obtain \ref{trunc5:1}, as this destroys the careful bounds \ref{trunc3:1} and \ref{trunc4:1} that are needed in the proof in Section \ref{sec:5}.

As mentioned before, the Euler--Lagrange equation for the functional $I_\eta$ corresponds to an \emph{elliptic regularisation} of a parabolic equation, consequently, a purely parabolic Lipschitz truncation (cf. \cite{KL00,BDS,DSSV}) is not sufficient. Instead, we construct a truncation that is specifically tailored to an elliptic regularisation of a parabolic problem.

We also mention that the statement of Lemma \ref{lemma:trunc:v2} is by no means optimal (e.g. in terms of \ref{trunc5:1}),  which is one of the reasons, why the estimates in the proof of Lemma \ref{prop:hansschaefer}, especially those in the fourth step, require further integration by parts.

\subsection{Elliptic-parabolic truncation for higher regularity}
As $\partial_t w_\eta$ is contained in a negative (spatial) Sobolev space, we consider functions $v_\eta \colon \R \times \R^d \to \R^{d \times d}_{\skw}$ satisfying
\[
    \curl^{\ast} v_\eta = w_\eta \quad \text{in } (0,\infty) \times \T_d
\]
instead. This allows us to work in spaces with improved spatial differentiability.
For more details we refer to Subsection \ref{section:sol}.
For notational simplicity, as $v_\eta$ denotes a function defined on $\R \times \R^d$, we use the shorthand $W^k_p = W^k_p(\R^d)$. In the following (cf. Subection \ref{section:sol}), $v \colon \R\times \R^d\to  \R^{d \times d}_{\skw}$ can, in principle, be replaced by $v_\eta \colon \R \times \T_d\to  \R^{d \times d}_{\skw}$  which are periodic functions. As certain estimates, for instance for the maximal function, are slightly more standard on the fullspace, we stick to $v_\eta$ defined on $\R \times \R^d$ for the time being.

If not stated otherwise, we now fix some $0<\eta <1$, $s<\infty$, and assume the following counterparts to \ref{P:1}--\ref{P:4}: 
\begin{enumerate} [label=(P\arabic*')]
    \item \label{Q:1} $v_\eta \in L_p(\R;W^2_p)$;
    \item \label{Q:2} $\partial_t v_\eta \in L_{s'}(\R;L_{s'})$ for some $s<\infty$ and $\partial_t v_\eta = \barg_\eta + \barh_\eta$, with $\barg_\eta \weakto 0$ in $L_q((0,T);L_q)$ and $\barh_\eta \to 0$ in $L_{s'}((0,T);L_{s'})$ for some $s<\infty$ and all $T>0$;
    \item \label{Q:3} $\eta^{1/2} \partial_t v_\eta \in L_2(\R;W^1_2)$;
    \item \label{Q:4} $\eta \partial_t^2 v_\eta \in L_{s'}(\R;L_{s'})$ for some $s<\infty$ and, in addition $\eta \partial_t^2 v_\eta= \barvarg_\eta + \barvarh_\eta$, where $\barvarg_\eta \weakto 0$ in $L_q((0,T);L_q)$ and $\barvarh_\eta \to 0$ in $L_{s'}((0,T);L_{s'})$ for some $s<\infty$ and all $T>0$;
\end{enumerate}
\begin{remark}
    We assume a global bound on the $L_{s'}(\R;L_{s'})$-norm of $\partial_t v_\eta$; but the truncation result also holds if we only have local bounds. This might for instance be shown by considering a cut-off in time.
\end{remark}
For brevity we drop the index $\eta$ and write for instance $v=v_\eta$.
We define
\begin{equation*}
    \alpha = (p-1) = \frac{p}{q}
    \quad \text{and}
    \quad
    \beta = \max \left\{\frac{p}{2},\frac{3p-4}{2} \right\}.
\end{equation*}

The goal of this and the following subsections is to prove the following truncation statement for $v$.
\begin{proposition}[Lipschitz truncation for an elliptic regularisation] \label{prop:hrtruncation}
Suppose that $v$ satisfies \ref{Q:1}--\ref{Q:4}. There is a constant $C(d,v )$ such that for any $L > C(d, v )$, there is a function $v^L$ such that 
\begin{enumerate} [label=(\roman*)]
    \item \label{hr:1}$\Vert v^L \Vert_{L_{\infty}(\R;W^{2}_{\infty})} \leq C_0(d) L$ and $\Vert \partial_t v^L \Vert_{L_{\infty}(\R;L_{\infty})} \leq C_0(d)L^\alpha$;
    \item \label{hr:2} $\eta^{1/2} \Vert \partial_t v^L \Vert_{L_\infty(\R;W^1_\infty)} \leq C_0(d)L^\beta$;
    \item \label{hr:3} $\eta \Vert \partial_t^2 v^L \Vert_{L_\infty(\R;L_\infty)} \leq C_0(d)L^\alpha$;
    \item \label{hr:4} We have the following estimate on the set where $v$ and $v^L$ \textbf{do not} coincide: 
        \begin{equation} \label{eq:hr:4}
            \mathcal{L}^{d+1}(\{v \neq v^L\}) \leq C_1(d,v) L^{-p} + C_2(d,v) L^{-ps'/q}.
        \end{equation}
\end{enumerate}
\end{proposition}
Observe that the statement of Proposition \ref{prop:hrtruncation} is slightly suboptimal in the sense that $\beta>p/2$ might happen in \ref{hr:2}. This issue can, in principle, be fixed if we slightly weaken statement \eqref{eq:hr:4}. However, as the $L_2$-distance between $\partial_t w_\eta$ and $\partial_t w_\eta^L$ is irrelevant in the proofs of Section \ref{sec:5} (e.g. Step 4 in the proof of Lemma \ref{prop:hansschaefer}), \ref{hr:2} is sufficient in our setting.
Furthermore, we explicitly state the constant appearing in \eqref{eq:hr:4} later (cf. Lemma \ref{lemma:ronweasley}). As it turns out, if we take a sequence $v_\eta$ as before, then $C_1(d,v_\eta)$ is uniformly bounded and $C_2(d,v_\eta) \to 0$ as $\eta \to 0$.
\smallskip

The proof of Proposition \ref{prop:hrtruncation} roughly consists of the following steps:
    \begin{itemize} 
        \item identify a large set (called the good set) on which the function $v$ is  `Lipschitz continuous', cf. Section \ref{section:Lipschitz};
        \item show the bound for its complement (the bad set) featured in \ref{hr:4}, cf. Section \ref{section:max};
        \item construct an extension operator that leaves $v$ unchanged on the good set and replaces it by its extension on the bad set, cf. Section \ref{section:Whitney};
        \item show that the extension has the properties required in Proposition \ref{prop:hrtruncation}, cf. Sections \ref{section:Lipschitz} and \ref{section:finish}.
    \end{itemize}

\subsection{The maximal function} \label{section:max}
We consider the following two metrics on $\R \times \R^d$ with a parameter $\kappa$, which is chosen dependent on $L$ such that
\begin{equation} \label{def:kappa}
    \kappa = L^{(p-2)/2}.
\end{equation}
Observe that $\kappa^2 L = L^{p-1}=L^\alpha$ and that $\kappa^3 L\leq L^\beta$ \schange{, whenever $L \geq 1$}. First, the \emph{parabolic metric} $d_{\pa}$ is defined by 
    \[
        d_\pa((t,x),(t',x')) = \max \left\{ \kappa \vert t-t' \vert^{1/2}, \vert x-x' \vert \right\}
    \]
and the \emph{elliptic metric} $d_\el$ with ellipticity parameter $\eta$ is given by
    \[
    d_\el((t,x),(t',x')) = \max \left\{ \eta^{-1/2} \kappa \vert t-t' \vert, \vert x-x' \vert \right\}.
    \]
For some $\rho>0$ we then define $B^{\pa}_\rho(t,x)$ and $B^{\el}_\rho(t,x)$ to be the $\rho$-ball around $(t,x) \in \R \times \R^{d}$ with respect to the parabolic and the elliptic metric, respectively. If we write $B_\rho(x)$ or $B_\rho(t,x)$ we mean the ball of the Euclidean metric.

We may now define the (centred) maximal function with respect to both those metrics. For a function $v \in L_{1,\loc}(\R \times \R^d)$ we define 
    \[
    \M_c^{\pa} v (t',x') \coloneq \sup_{\rho>0} \fint_{B^{\pa}_\rho(t',x')} \vert v \vert \dd x \dd t
    \quad 
    \text{and}
    \quad
    \M^{\pa} v (t',x')
    \coloneq
    \sup_{\rho>0,(s,z) \colon \atop (t',x') \in B^{\pa}_\rho(s,z)} \fint_{B^{\pa}_\rho(s,z)} \vert v \vert \dd x \dd t.
    \]
Likewise, we may define the centred and non-centred maximal function with respect to the elliptic metric, $\M_c^{\el}$ and $\M^{\el}$. Due to the scaling properties of those metrics, we have the pointwise bounds
    \begin{equation} \label{maximalf:pointwise}
    \M_c^{\pa} v \leq \M^{\pa } v \leq 2^{d+2} \M^{\pa}_c v
    \quad 
    \text{and}
    \quad
    \M_c^{\el} v \leq \M^{\el} \leq 2^{d+1}\M^{\el}_c,
    \end{equation}
so, up to dimensional constants, the following observations are valid both for the centred- and non-centred maximal function.

The following statement is well-known, we shortly remind the reader of the proof (e.g. \cite[p. 5, Thm. 1b)]{Stein}) to see that the bounds \emph{do not} depend on the parameters $\kappa$ and $\eta$.

\begin{lemma} \label{weakl1}
       Let $1<p \leq \infty$. Both $\M^{\pa}$ and $\M^{\el}$ are bounded sublinear operators from $L_p(\R \times \R^d)$ to $L_p(\R \times \R^d)$. In particular, the operator norms of $\M^{\pa}$ and $\M^{\el}$ may be bounded independently of $\kappa$ and $\eta$.
\end{lemma}

\begin{proof}
We prove the estimate for the centred maximal function, for the non-centred it follows by \eqref{maximalf:pointwise}. Observe that both $\M^{\pa}$ and $\M^{\el}$ are obviously sublinear bounded operators from $L_{\infty}(\R \times \R^d)$ to $L_{\infty}(\R \times \R^d)$ with 
    \[
    \Vert \M^{\pa} v \Vert_{L_\infty} \leq \Vert v \Vert_{L_\infty} 
    \quad
    \text{and}
    \quad
    \Vert \M^{\el} v \Vert_{L_\infty} \leq \Vert v \Vert_{L_\infty}. 
    \]
We show that $\M^{\pa}_c$ and $\M^{\el}_c$ are bounded from $L_1(\R \times \R^d)$ into (the weak $L_1$ space) $L_{1,\infty}(\R \times \R^d)$ with bounds independent of $\kappa$ and $\eta$. The Marcinkiewicz interpolation theorem then directly gives the lemma.

Let $\lambda>0$. We may cover $\{\M^{\pa}_c > \lambda \}$ and $\{ \M^{\el}_c >\lambda\}$ with balls centred at points $(t,x)$ with radii $\rho$, and at points $(t',x')$ with radii $\rho'$, respectively, such that
    \[
    \fint_{B_\rho^{\pa}(t,x)} \vert v \vert > \lambda \quad \text{and} \quad  \fint_{B_{\rho'}^{\el}(t',x')} \vert v \vert > \lambda.
    \]
By Vitali's covering lemma, we can choose a collection $\mathscr{C}$ and $\mathscr{C}'$, respectively, of disjoint balls, such that $B_{5\rho}^{\pa}(t,x)$ still covers $\{\M^{\pa}_c > \lambda \}$ and $B_{5\rho'}^{\el}(t',x')$ still covers
$\{\M^{\el}_c > \lambda \}$. Moreover, we have (independently of $\kappa$ and $\eta$) 
    \[
    \mathcal{L}^{d+1}(B^{\pa}_{5\rho}(t,x)) = 5^{d+2} \mathcal{L}^{d+1}(B^{\pa}_{\rho}(t,x)) \quad \text{and} \quad  \mathcal{L}^{d+1}(B^{\el}_{5\rho'}(t',x')) =5^{d+1} \mathcal{L}^{d+1}(B^{\el}_{\rho'}(t',x'))
    \]
Consequently,
    \begin{align*}
        \mathcal{L}^{d+1} (\{\M^{\pa}_c >\lambda\}) &\leq \sum_{B^{\pa}_\rho \in \mathscr{C}} \mathcal{L}^{d+1}(B_{5\rho}^{\pa}) \leq 5^{d+2} \sum_{B^{\pa}_\rho \in \mathscr{C}} \mathcal{L}^{d+1}(B_{\rho}^{\pa})   \\
        & \leq 5^{d+2}\sum_{B^{\pa}_\rho \in \mathscr{C}} \lambda^{-1} \int_{B_{\rho}^{\pa}(t,x)} \vert v \vert \leq 5^{d+2}\lambda^{-1} \Vert v \Vert_{L^1}.
    \end{align*}
Therefore, the operator $\M^{\pa}$ is bounded from $L_1(\R \times \R^d)$ to $L_{1,\infty}(\R \times \R^d)$ with operator norm bounded by $5^{d+2}$. The same calculation for the elliptic balls shows that $\M^{\el}$ is bounded from $L_1(\R \times \R^d)$ to $L_{1,\infty}(\R \times \R^d)$ with operator norm bounded by $5^{d+1}$.
\end{proof} 
Now fix $L>0$.
We first introduce the set $\B_L$, where the function is large, as
\begin{align}\label{eq:badset}
        \B_L&\coloneq \B_L^1 \cup \B_L^2 \cup \B_L^3 \cup \B_L^4;\\
        \B_L^1&\coloneq \{ \vert v \vert > L \} \cup \{\vert \nabla v \vert > L \} \cup \{\vert \nabla^2 v \vert > L \};\nonumber\\
        \B_L^2&\coloneq \{ \vert \partial_t \nabla v \vert > L^{p/2} \eta^{-1/2}\};\nonumber\\
        \B_L^3&\coloneq \{ \vert \barg \vert > L^{\alpha} \} \cup \{\vert \barvarg \vert >  L^{\alpha}\}; \nonumber\\
        \B_L^4&\coloneq \{ \vert \barh \vert > L^{\alpha} \} \cup \{\vert \barvarh \vert > L^{\alpha}\};\nonumber
        \end{align}
We then define the \emph{bad set} as superlevel set of the \emph{maximal function}. In more detail, 
    \begin{equation} \label{def:badset}
        \bad_L \coloneq \bad_L^{\el} \cup \bad_L^{\pa},
    \end{equation}
where $\bad_L^{\pa}$ is defined via 
    \begin{align*}
        \bad_L^{\pa}&\coloneq \bad_L^{\pa,1} \cup \bad_L^{\pa,2} \cup \bad_L^{\pa,3} \cup \bad_L^{\pa,4};\\
        \bad_L^{\pa,1}&\coloneq \{ \M^{\pa} v > 2L \} \cup \{\M^{\pa} (\nabla v) > 2L \} \cup \{\vert \M^{\pa} (\nabla^2 v) \vert > 2L \}
        \\
        \bad_L^{\pa,2}&\coloneq \{ \M^{\pa} (\partial_t \nabla v) \vert > 2L^{p/2} \eta^{-1/2}\};
        \\
        \bad_L^{\pa,3}&\coloneq \{ (\M^{\pa} \barg) \vert > 2L^\alpha  \} \cup \{ (\M^{\pa} \barvarg) \vert > 2L^\alpha\};
        \\
        \bad_L^{\pa,4}&\coloneq \{ (\M^{\pa} \barh) \vert > 2L^\alpha \} \cup \{ (\M^{\pa} \barvarh) \vert > 2L^\alpha  \}
     \end{align*}
The set $\bad^{\el}_L$ is defined likewise, with the parabolic maximal function $\M^{\pa}$ replaced by the elliptic maximal function $\M^{\el}$. We also denote the complement of this bad set, the \emph{good set}, by
    \begin{align*}
        \good_L= (\R \times \R^n) \setminus \bad_L.
    \end{align*}
The following estimate, making use of the $L_1$ and $L_{1,\infty}$ bound, is inspired by \cite[Lemma 3.1]{Zhang}. It reveals why we explicitly choose the exponents $\alpha$ and $p/2$ in the definition of the bad set.
\begin{lemma} \label{lemma:ronweasley}
    Let $L>0$, $v$ and $\bad_L$ and $\B_L$ be as in \eqref{def:badset} and \eqref{eq:badset}. Then
         \begin{equation} \label{lemma:ronweasly:eq}
            \begin{split}
            \mathcal{L}^{d+1}(\bad_L) \leq& C_d L^{-p}\int_{\B_L^1 \cup \B_L^2 \cup \B_L^3} \vert v \vert^p + \vert \nabla v \vert^p + \vert \nabla^2 v \vert^p + \eta^{-1} \vert \partial_t \nabla v \vert^{2}
            + \vert \barg \vert^q + \vert \barvarg \vert^q 
             \dd x \dd t \\           
             &+C_d L^{-ps'/q} \int_{\B_L^4} \vert \barh \vert^{s'} + \vert \barvarh \vert^{s'} \dd x \dd t.
            \end{split}
        \end{equation}
\end{lemma}
\begin{remark} \label{remark:choiceL}
Recall that $\kappa = L^{(p-2)/2}$ and therefore $\kappa^2$ grows slower than $L^p$ and $L^{ps'/{q}} =L^{(p-1)s'}$ . Therefore,  we can choose $L$ large enough, so that $\mathcal{L}^{d+1}(\bad_L) \leq \tfrac{1}{c_d} \kappa^2$, with a lower bound on $L$, depending on norms for $v$ appearing in \ref{Q:1}--\ref{Q:4} (for a dimensional constant $c_d$ to be specified later).
\end{remark}
\begin{remark} \label{remarkgh}
Recall that the maximal function is sublinear, hence we have
\begin{align*}
    \{ \M^{\pa} (\partial_t v) >4L^{\alpha} \} \subset \left(\{
    \M^{\pa} (\barg) >2L^{\alpha} \} \cup 
     \{
    \M^{\pa} (\barh) >2L^{\alpha} \}\right) =\bad_L^{\pa,4}
\end{align*}
and a similar statement holds true for the maximal function of $\eta \partial_t^2 v.$
\end{remark}
\begin{proof}
    We first  show that for any function $f \in L_r(\R\times\R^d)$, $1\leq r <\infty$, we have
       \begin{equation} \label{eq:ronweasley}
            \mathcal{L}^{d+1}( \{ \M^{\pa} f \geq 2 \lambda\}) \leq C_d \int_{\{\vert f \vert \geq \lambda \}} \lambda^{-r} \vert f \vert^r \dd x \dd t,
        \end{equation} 
    and the same for the elliptic maximal function. 
    To show \eqref{eq:ronweasley}, consider $f_\lambda =\max\{0, \vert f \vert - \lambda\}$. Then, due to the $L_r$-$L_{r,\infty}$ bound on the maximal operator (cf. Lemma \ref{weakl1}) and its sublinearity, we obtain
        \begin{align*}
            \mathcal{L}^{d+1}( \{ \M^{\pa} f \geq 2 \lambda\}) &\leq \mathcal{L}^{d+1}( \{ \M^{\pa} f_{\lambda} \geq  \lambda\}) \\
            &\leq C_d \lambda^{-r} \Vert f_{\lambda} \Vert_{L_r} \\
            & =C_d \lambda^{-r} \int_{\{\vert f \vert \geq \lambda\}} 
            (\vert f \vert- \lambda)^r \dd x \dd t \\
            &\leq C_d \lambda^{-r} \int_{\{\vert f \vert \geq \lambda\}} \vert f \vert^r \dd x \dd t.
        \end{align*}
    Now applying this result to the following cases yields the result:
     \begin{enumerate} [label=(\roman*)]
         \item for $f=v$, $f = \nabla v$, $f = \nabla^2 v$, $r=p$ and $\lambda= L$;
         \item for $f=\partial_t \nabla v$, $r=2$ and $\lambda =\eta^{-1/2} L^{p/2}$;
         \item for $f= \barg$ and $f= \barvarg$, $r=q$ and $\lambda= L^{\alpha} =L^{p-1}$;
         \item for $f= \barh$ and $f= \barvarh$, $r=s'$ and $\lambda =L^{\alpha}$.
     \end{enumerate}
\end{proof}
\subsection{Decomposition into Whitney cubes} \label{section:Whitney}
In this section, we discuss a decomposition of the bad set $\bad_L$ into a suitable modification of Whitney/ Calder\'on-Zygmund cubes. Recall that $\good_L$ is closed and $\bad_L$ is open. For any $(t,x) \in \B_L$ we have the implication
    \begin{equation} \label{est:1}
    \begin{split}
        d_{\pa}((t,x),\good_L) \geq \kappa \eta^{1/2} \quad \Longrightarrow \quad d_{\pa}((t,x),\good_L) \leq d_{\el}((t,x),\good_L), \\
        d_{\el}((t,x),\good_L) \leq  \kappa \eta^{1/2} \quad \Longrightarrow \quad 
        d_{\pa}((t,x),\good_L) \geq d_{\el}((t,x),\good_L).
    \end{split}
    \end{equation}
In particular, the metrics coincide at distance exactly $\kappa \eta^{1/2}$ and, moreover, are comparable if the distance is comparable to $\kappa \eta^{1/2}$.

\medskip
\noindent \textbf{Parabolic cubes. } Following \cite{KL00,BDS} we define dyadic parabolic cubes as cubes of the form $(t,x)+[0,\kappa^{-2}2^{-2m}] \times [0,2^{-m}]^{d},\, m \in \Z$. We start from the unit grid (i.e. from cubes $\Z^{d+1}+[0,1]^{d+1}$), then rescale time by $\kappa$ and then successively subdivide a cube $(t,x)+[0,\kappa^{-2}2^{-2m}] \times [0,2^{-m}]^{d}$ into $2^{d+2}$ disjoint rectangles of sidelength $\kappa^{-2}2^{-2(m+1)}$ in time and $2^{-(m+1)}$ in space. For such a parabolic cube $Q$ we denote by $l_t(Q)$ the length of the side in time and by $l_x(Q)$ the length in space such that $l_t(Q)= \kappa^{-2}l_x(Q)^2$.

For each $(t,x) \in \bad_L$ we then may take the largest parabolic dyadic cube $Q$ containing $(t,x)$, such that $2Q$ (the cube with the same centre but twice the sidelength w.r.t. the parabolic metric) is still contained in $\bad_L$. We then get a countable collection $\Qu^{\pa,\ast}$ of parabolic dyadic cubes $(Q^{\ast}_i)_{i \in \N}$, such that
\begin{enumerate} [label=(\roman*)]
    \item $\bad_L = \bigcup_{i \in \N} Q^{\ast}_i$;
    \item the interiors of $Q^{\ast}_i$ and $Q^{\ast}_j$ are disjoint for $i\neq j$;
    \item for each $i \in \N$, $Q^{\ast}_i$ only touches at most $C(d)$ cubes;
    \item if $Q^{\ast}_i$ and $Q^{\ast}_j$ touch, then $\tfrac{1}{4} l_x(Q^{\ast}_i) \leq l_x(Q^{\ast}_j) \leq 4 l_x(Q^{\ast}_i)$.
    \item $\tfrac{1}{c_d} d_{\pa}(Q^{\ast}_i,\good_L) \leq l_x(Q^{\ast}_i) \leq c_d d_{\pa}(Q^{\ast}_i,\good_L)$.
\end{enumerate}
The dimensional constant $c_d$ might be chosen to be $4 (\sqrt{d}+1)$.
Consider a small parameter $\varepsilon< \tfrac{1}{10}$ and define
    \[
    Q_i = (1+\varepsilon)(Q_i^{\ast})^{\circ},
    \]
such that $Q_i$ is the open cube with the same centre as $Q_i^{\ast}$ and such that $l_x(Q_i) = (1+\varepsilon)l_x(Q_i^{\ast})$ and $l_t(Q_i)=(1+\varepsilon)^2 l_t(Q_i^{\ast})$.
Then the collection of those cubes $Q_i$ has the following properties:
    \begin{enumerate} [label=(\roman*),resume]
        \item $\bad_L = \bigcup_{i \in \N} Q_i$;
        \item for fixed $i \in \N$, $Q_i \cap Q_j \neq \emptyset$ for at most $C(d)$ cubes;
        \item for any $(t,x) \in \bad_L$ there is a small neighbourhood $O$ such that $O \cap Q_i \neq \emptyset$ for at most $C(d)$ cubes;
        \item if $Q_i \cap Q_j \neq \emptyset $, then $\tfrac{1}{4} l_x(Q_i) \leq l_x(Q_j) \leq 4 l_x(Q_i)$;
        \item $\tfrac{1}{c_d} d_{\pa}(Q_i,\good_L) \leq l_x(Q_i) \leq 2c_d d_{\pa}(Q_i,\good_L)$.
    \end{enumerate}
Let us consider a cut-off $\phi^{\ast} \in C_c^{\infty}((-\varepsilon/2,1+\varepsilon/2)^{d+1};[0,1])$ that equals one in $[0,1]^{d+1}$. Consider $\phi_i^{\ast}$ the rescaled (with respect to the parabolic scaling) and displaced version that is supported on $Q_i$ and equals one on $Q_i^{\ast}$. Then $\phi_i^{\ast}$ enjoys the following bounds for the derivatives:
    \begin{enumerate}[label=(\roman*),resume]
        \item $\Vert \nabla \phi_i^{\ast} \Vert_{L_{\infty}} \leq C l_x(Q_i)^{-1}$;
        \item $\Vert \nabla^2 \phi_i^{\ast} \Vert_{L_{\infty}} \leq C l_x(Q_i)^{-2}$;
        \item $\Vert \partial_t \phi_i^{\ast} \Vert_{L_{\infty}} \leq C \kappa^2 l_x(Q_i)^{-2}$;
        \item $\Vert \partial_t \nabla \phi_i^{\ast} \Vert_{L_{\infty}} \leq C \kappa^2 l_x(Q_i)^{-3}$;
        \item $\Vert \partial_t^2 \nabla \phi_i^{\ast} \Vert_{L_{\infty}} \leq C \kappa^4 l_x(Q_i)^{-4}$.
    \end{enumerate}
\medskip
\noindent \textbf{Elliptic cubes. } Likewise, we may define a covering with elliptic cubes, i.e. cubes that have the form $(t,x) + [0,\kappa^{-1}\eta^{1/2}2^{-m}] \times [0,2^{-m}]^d,\, m \in \Z$. In this case, the `elliptic' cubes are nothing else but usual Calder\'on-Zygmund cubes when rescaling the time with factor $\kappa\eta^{-1/2}$, i.e. $l_t(Q) = \kappa^{-1}\eta^{1/2} l_x(Q)$. With the same notation as before, we introduce a cover $\Qu^{\el,\ast}$ of such elliptic cubes $Q'_i$, such that (e.g. \cite{Stein})
    \begin{enumerate}  [label=(\roman*')]
    \setcounter{enumi}{6}
         \item $\bad_L = \bigcup_{i \in \N} Q'_i$;
        \item for fixed $i \in \N$, $Q'_i \cap Q'_j \neq \emptyset$ for only $C(d)$ cubes;
        \item for any $(t,x) \in \bad_L$ there is a small neighbourhood $O$ such that $O \cap Q'_i \neq \emptyset$ for only $C(d)$ cubes;
        \item if $Q'_i \cap Q'_j \neq \emptyset $, then $\tfrac{1}{4} l_x(Q'_i) \leq l_x(Q'_j) \leq 4 l_x(Q'_i)$;
        \item $\tfrac{1}{c_d} d_{\el}(Q_i,\good_L) \leq l_x(Q_i) \leq 2c_d d_{\el}(Q_i,\good_L)$.
    \end{enumerate}
Moreover, scaling the cut-off $\phi^{\ast}$ in an elliptic fashion, one obtains $\phi^{\ast}_i$ supported on $Q'_i$ with the following bounds on the derivatives:
     \begin{enumerate}[label=(\roman*'),resume]
        \item $\Vert \nabla\phi_i^{\ast}\Vert_{L_{\infty}} \leq C l_x(Q_i')^{-1}$;
        \item $\Vert \nabla^2 \phi_i^{\ast} \Vert_{L_{\infty}} \leq C l_x(Q_i')^{-2}$;
        \item $\Vert \partial_t \phi_i^{\ast} \Vert_{L_{\infty}} \leq C \kappa \eta^{-1/2} l_x(Q_i')^{-1}$;
        \item $\Vert \partial_t \nabla\phi_i^{\ast}\Vert_{L_{\infty}} \leq C \kappa \eta^{-1/2} l_x(Q_i')^{-2}$;
        \item $\Vert \partial_t^2 \phi_i^{\ast} \Vert_{L_{\infty}} \leq C \kappa^2 \eta^{-1} l_x(Q_i')^{-2}$.
    \end{enumerate}
    
\medskip 
\noindent \textbf{Combining elliptic and parabolic cubes. }
We now define the combined cover being used in the following. For this purpose, consider both covers constructed before and take only parabolic cubes with $l_x(Q) \geq \tfrac{\kappa}{2c_d} \eta^{1/2}$ and elliptic cubes with $l_x(Q) \leq 2c_d\kappa \eta^{1/2}$.
We then obtain a cover 
    \[
    \Qu = \Qu^{\el} \cup \Qu^{\pa}
    \]
consisting of  elliptic cubes $Q'_i$ in $\Qu^{\el,\ast}$ with small sidelengths as well as parabolic cubes $Q_i$ in $\Qu^{\pa,\ast}$ with large sidelengths. In particular, note that the cubes in $\Qu^{\el}$ cover $\{(t,x) \in \bad_L \colon d_{\el}((t,x),\good_L) \leq \kappa \eta^{1/2}\}$ and the cubes in $\Qu^{\pa}$ cover $\{(t,x) \in \bad_L \colon d_{\pa}((t,x),\good_L) \geq \kappa \eta^{1/2}\}$. Due to \eqref{est:1}, $\Qu$ covers $\B_L$.

We relabel the cubes in the cover (and  also the $\phi_i^{\ast}$'s in the partition of unity), such that $\Qu$ consists of cubes $Q_i$, $i \in \N$.
As there is a certain change of type of cubes at scale $\kappa \eta^{1/2}$, we also introduce the union 
    \[
    \widehat{\Qu^{\el}} = \Qu^{\el} \cup \{Q \in \Qu^{\pa} \colon \exists Q' \in \Qu^{\el} \colon Q \cap Q' \neq \emptyset \}
    \]
of all elliptic cubes and all parabolic cubes that touch an elliptic cube. Similarly, the union
    \[
    \widehat{\Qu^{\pa}} = \Qu^{\pa} \cup \{Q' \in \Qu^{\el} \colon \exists Q \in \Qu^{\pa} \colon Q' \cap Q \neq \emptyset \}
    \]
consists of all parabolic cubes and all elliptic cubes that touch a parabolic cube.
We first collect some properties that directly follow from this definition and the properties of $\Qu^{\el,\ast}$ and $\Qu^{\pa,\ast}$, respectively.

\begin{lemma}
Let $\bad_L$ be a closed set and let $\Qu$ be as constructed a cover consisting of $(Q_i)_{i \in \N}$. Then
\begin{enumerate} [label=(Q\arabic*)] 
    \item $\bad_L = \bigcup_{i \in \N} Q_i$;
    \item for fixed $i \in \N$, $Q_i \cap Q_j \neq \emptyset$ for at most $C(d)$ cubes;
    \item for any $(t,x) \in \bad_L$, there is a small neighbourhood $O$, such that $O \cap Q_i \neq \emptyset$ of at most $C(d)$;
    \item if $Q_i \cap Q_j \neq \emptyset$, then $ \tfrac{1}{4} l_x(Q_i) \leq l_x(Q_j) \leq 4 l_x(Q_i)$;
    \item If $Q_i \in \Qu^{\el}$, then $\tfrac{1}{c_d} d_{\el}(Q_i,\good_L) \leq l_x(Q_i) \leq 2c_d d_{\el}(Q_i,\good_L)$;
    \item If $Q_i \in \Qu^{\pa}$, then $\tfrac{1}{c_d} d_{\pa}(Q_i,\good_L) \leq l_x(Q_i) \leq 2c_d d_{\pa}(Q_i,\good_L)$;
    \item There is a constant, such that if $Q_i \in \widehat{\Qu^{\el}} \cap \widehat{\Qu^{\pa}}$, then
    \begin{itemize}
        \item $C^{-1} l_x(Q_i) \leq \eta^{1/2}\kappa \leq C l_x(Q_i)$;
        \item $C^{-1} l_t(Q_i) \leq \eta \leq C l_t(Q_i)$.
    \end{itemize}
\end{enumerate}
\end{lemma}
Using the functions $\phi_i^{\ast}$ introduced above, we finally define a partition of unity by
    \begin{equation}
        \phi_i(t,x) = \frac{\phi_i^{\ast}((t,x))}{\sum_{j \in \N}  \phi_j^{\ast}((t,x))}.
    \end{equation}
Due to the properties of $\phi_i^{\ast}$ (both in the parabolic and the elliptic regime) and the previous lemma we obtain that $\phi_i$ is still smooth and $\sum_{i \in \N} \phi_i = \chi_{\bad_L}$. Moreover, by applying the quotient rule and the estimates for the turnover region from the previous lemma (in particular for cubes in $\widehat{\Qu^{\pa}} \cap \widehat{\Qu^{\el}}$) we get the following:

\begin{lemma}
    Let $Q_i \in \Qu$ and let $(t,x) \in \bad_L$.
    \begin{enumerate} [label=(Q\arabic*)]
    \setcounter{enumi}{8}
        \item \label{Q:9} If $Q_i \in \widehat{\Qu^{\pa}}$, $(t,x) \in Q_i$, then 
         \begin{enumerate}[label=(\roman*)]
        \item $\vert \nabla \phi_i(t,x) \vert \leq C l_x(Q_i)^{-1}$;
        \item $\vert \nabla^2 \phi_i(t,x) \vert \leq C l_x(Q_i)^{-2}$;
        \item $\vert \partial_t \phi_i(t,x) \vert \leq C \kappa^2 l_x(Q_i)^{-2}$;
        \item $\vert \partial_t \nabla \phi_i (t,x) \vert \leq C \kappa^2 l_x(Q_i)^{-3}$;
        \item $\vert \partial_t^2 \phi_i(t,x) \vert \leq C \kappa^4 l_x(Q_i)^{-4}$.
    \end{enumerate}
    \item \label{Q:10} If $Q_i \in \widehat{\Qu^{\el}}$, $(t,x) \in Q_i$, then 
         \begin{enumerate}[label=(\roman*)]
        \item $\vert \nabla \phi_i(t,x) \vert \leq C l_x(Q_i)^{-1}$;
        \item $\vert \nabla^2 \phi_i(t,x) \vert \leq C l_x(Q_i)^{-2}$;
        \item $\vert \partial_t \phi_i(t,x) \vert \leq C \eta^{-1/2} \kappa l_x(Q_i)^{-1}$;
        \item $\vert \partial_t \nabla \phi_i (t,x) \vert \leq C \eta^{-1/2} \kappa l_x(Q_i)^{-2}$;
        \item $\vert \partial_t^2  \phi_i(t,x) \vert \leq C \eta^{-1} \kappa^2 l_x(Q_i)^{-2}$.
    \end{enumerate}
    \end{enumerate}
\end{lemma}

\begin{remark} \label{remark:choice2}
    The lower bound on $L$ that we derived in Remark \ref{remark:choiceL} yields an upper bound on the sidelength for any cube in $\Qu$. In particular, for $L$ as in Remark \ref{remark:choiceL} and a cover $\Qu$ of $\bad_L$ we have for any $Q_i \in \Qu$
    \begin{equation}
        l_x(Q_i) \leq 1.
    \end{equation}
    More precisely, if the cube $Q_i$ is elliptic, then $Q_i \subset \bad_L$ and $l_x(Q_i) \leq 2c_d\eta^{1/2}\kappa$ (cf. \eqref{est:1}). Writing $l_t(Q_i)=\kappa^{-1} \eta^{1/2} l_x(Q)$, we get $\mathcal{L}^{d+1}(Q_i) = \kappa^{-1} \eta^{1/2} l_x(Q_i)^{d+1}$. If $l_x(Q_i)$ was larger than one, one would obtain $ \mathcal{L}^{d+1}(\bad_L) \geq \mathcal{L}^{d+1}(Q_i) \geq \tfrac{1}{2c_d} \kappa^{-2}$, leading to a contradiction to the upper bound of $\mathcal{L}^{d+1}(\bad_L)$ obtained in Remark \ref{remark:choiceL}.
    A similar lower bound can be achieved for parabolic cubes, i.e. we obtain $\mathcal{L}^{d+1}(\bad_L) \geq \kappa^{-2} l_x(Q_i)^{d+2} \geq \kappa^{-2}$, also leading to a contradiction to the lower bound in Remark \ref{remark:choiceL}.
\end{remark}

\subsection{Definition of the truncation}
We consider the bad set $\bad_L$ as defined in \eqref{def:badset} and the cover $\Qu$ with cubes $Q_i$ and with the corresponding partition of unity $\phi_i$ as constructed in the previous section. We now define functions $v_i$ and $\tilde{v}_i$ as averaged Taylor polynomials of order one with respect to a cube $Q_i$, i.e.
\begin{equation} \label{def:vi}
    v_i(s,y)= \fint_{Q_i} v(t,x) + \nabla v(t,x) \cdot (y-x) + \partial_t v(t,x) \cdot (s-t) \dd t \dd x
\end{equation}
and, for later use,
\begin{equation} \label{def:vitilde}
    \tilde{v}_i(s,y) = \fint_{Q_i} v(t,x) + \nabla v(t,x) \cdot (y-x) \dd t \dd x.
\end{equation}
Then we can define the truncation $v^L$ as follows.

\begin{definition}
    Let $v$, $\bad_L$, the cover $\Qu$, the partition of unity $\phi_i$ and $v_i$ be as above. We then define the truncation $v^L$ by
        \begin{equation} \label{def:vL}
            v^L(s,y) = 
            \begin{cases} v(s,y) & \text{if } (s,y) \in \good_L \\
            \sum_{i \in \N} \phi_i(s,y) v_i(s,y) & \text{if } (s,y) \in \bad_L.
            \end{cases}
        \end{equation}
\end{definition}
We claim that $v^L$ as defined in \eqref{def:vL} enjoys all the properties of Proposition \ref{prop:hrtruncation}. 


\subsection{Lipschitz estimates} \label{section:Lipschitz}
In this section we prove an estimate on $(v_i-v_j)$ for $(s,y) \in Q_i \cap Q_j$. This turns out to be a crucial step when proving that the truncation $v^L$ enjoys all the properties of Proposition \ref{prop:hrtruncation}.
More precisely, we prove the following lemma.

\begin{lemma} \label{lemma:harrypotter}
Let $v$, $\bad_L$, the cover $\Qu$, the partition of unity $\phi_i$ and $v_i$ be as above.
    Suppose that $(s,y) \in Q_i \cap Q_j \subset \bad_L$. Then
    \begin{enumerate} [label=(\roman*)]
            \item \label{hp0} $\vert v_i(s,y) \vert + \vert \nabla v_i(s,y) \vert \leq CL$.
            \item \label{hp05} $ \vert \partial_t v_i(s,y) \vert  \leq C L^\alpha$.
           \item \label{hp1} $ \vert (v_i-v_j)(s,y) \vert \leq CL l_x(Q_i)^2$; 
           \item \label{hp2} $\vert \nabla (v_i-v_j) (s,y) \vert \leq CL l_x(Q_i)^1$;
           \item \label{hp3} If $Q_i,Q_j \in \widehat{\Qu^{\el}}$, then $\vert \partial_t (v_i-v_j)(s,y) \vert \leq CL \eta^{-1/2} \kappa l_x(Q_i)^1$;
           \item \label{hp4} If $Q_i,Q_j \in \widehat{\Qu^{\pa}}$, then $\vert \partial_t (v_i-v_j)(s,y) \vert \leq CL \kappa^2$;
    \end{enumerate}
\end{lemma}
Observe that, as long as $Q_i$ and $Q_j$ are elliptic, the estimate on $\partial_t (v_i-v_j)$ is meaningful, whereas the estimate for parabolic cubes also just follows from \ref{hp05}.
\begin{proof}
We may distinguish the case where both $Q_i$ and $Q_j$ are contained in $\widehat{\Qu^{\el}}$ and the case where both are contained in $\widehat{\Qu^{\pa}}$. In both cases, the proof follows the lines of the elliptic/parabolic case, e.g. \cite{Liu} and \cite{BDS}, respectively.
\smallskip

The first and second assertion directly follow from the fact that $4c_d Q_i \cap \good_L \neq \emptyset$. In particular, for $f=v$, $f=\nabla v$ and $f=\partial_t v$, the definition of $\good_L$ as the sublevel set of the maximal function yields
    \[
    \fint_{Q_i}  \vert f \vert \dd x \dd t \leq C_d \fint_{4c_dQ_i} \vert  f \vert \dd x \dd t \leq 2C_dL.
    \] 
Moreover, observe that $v_i$ is an affine function (and so is $(v_i-v_j)$). As the bounds are valid \emph{for all} $(s',y')$ in $Q_i \cap Q_j$ and we have
    \begin{itemize}
        \item $\vert y- x \vert \leq l_x(Q_i)$ for all $(s,y),(t,x) \in Q_i$;
        \item $\vert s- t \vert \leq 2c_d \kappa^{-1} \eta^{1/2} l_x(Q_i)$ for all $(s,y),(t,x) \in Q_i$, if $Q_i \in \widehat{\Qu^{\el}}$;
        \item $\vert s- t \vert \leq 2c_d \kappa^{-2} l_x(Q_i)^2$ for all $(s,y),(t,x) \in Q_i$, if $Q_i \in \widehat{\Qu^{\pa}}$;
    \end{itemize}
for the derivative of the affine function we then may infer \ref{hp2}, \ref{hp3} and \ref{hp4}, respectively, from \ref{hp1}.
It remains to prove \ref{hp1}. We proceed by case distinction.
\smallskip

\noindent \textbf{Elliptic case: $Q_i,Q_j \in \widehat{\Qu^{\el}}$. } 
We first suppose that $v \in C^2(\R \times \R^d)$, the general case follows by a density argument. Let $\rho \coloneq \inf\{l_x(Q_i),l_x(Q_j)\}$ and consider $B^{\el}_{\rho}(s,y)$. We claim that
    \begin{equation} \label{hp:claim1}
        \left \vert v_i(s,y) - \fint_{B^{\el}_{\rho}(s,y)} v(s',y') \dd s' \dd y' \right \vert \leq CL\rho^2.
    \end{equation}
Using in \eqref{hp:claim1} the triangle inequality for $Q_i$ and $Q_j$ then directly yields \ref{hp1}.

To prove \eqref{hp:claim1}, first fix $(s',y') \in B^{\el}_{\rho}(s,y)$. As $Q_i \in \widehat{\Qu^{\el}}$, the sidelengths in time-scale and space-scale obey $c^{-1} \eta^{-1/2} \kappa l_t(Q_i) \leq l_x(Q_i) \leq c \eta^{-1/2} \kappa l_t(Q_i)$. Therefore, fixing a (purely dimensional) constant $c>0$, we have
    \[
        Q_i \subset B^{\el}_{c\rho}(s',y')
        \quad
        \text{and}
        \quad
        \mathcal{L}^{d+1}(Q_i) \geq c^{-1} \mathcal{L}^{d+1}(B^{\el}_{c\rho}(s',y')).
    \]
Consequently,
    \begin{align*}
        \vert v_i(s',y') - v(s',y') \vert 
        &= 
        \left \vert \fint_{Q_i} \bigl(v(t,x) + \nabla v(t,x) \cdot (y'-x) +\partial_t v(t,x) (s'-t)\bigr) - v(s',y')  \dd t \dd x \right \vert
        \\& 
        \leq C  \fint_{B^{\el}_{c\rho}(s',y')}
        \left \vert
        \bigl(v(t,x) + \nabla v(t,x) \cdot (y'-x) +\partial_t v(t,x) (s'-t)\bigr) - v(s',y')
        \right \vert
        \dd t \dd x.
    \end{align*}
Applying the fundamental theorem of calculus twice  then yields (for $t_h= ht+(1-h)s'$ and $x_h=hx+(1-h)y'$)
    \begin{align*}
        \vert v_i(s',y') - v(s',y') \vert& \leq  C \fint_{B^{\el}_{c\rho}(s',y')} \int_0^1 
            h \Bigl(\vert \nabla^2 v(t_h,x_h) (y'-x)^2 \vert+ \vert \partial_t \nabla v(t_h,x_h) \cdot (y'-x) (s'-t) \vert \\& \hspace{3cm} + \vert \partial_t^2 v(t_h,x_h) (s'-t)^2 \vert \Bigr) \dd h \dd t \dd x \\
            &= C \fint_{B^{\el}_{c\rho}(s',y')} \int_0^1 
            h^{-1} \Bigl(\vert \nabla^2 v(t_h,x_h) (y'-x_h)^2 \vert+ \vert \partial_t \nabla v(t_h,x_h) \cdot (y'-x_h) (s'-t_h) \vert
            \\& \hspace{3cm} + \vert \partial_t^2 v(t_h,x_h) (s'-t_h)^2 \vert \Bigr) \dd h \dd t \dd x
            \\
            &=(\ast).
    \end{align*}
Change of variables (rewriting in terms of $t_h$ and $x_h$) gives
    \begin{align*}
        (\ast) &= C\fint_{\B^{\el}_{c\rho}(s',y')} \Bigl(\vert \nabla^2 v(t,x) (y'-x)^2 \vert+ \vert \partial_t \nabla v(t,x) \cdot (y'-x) (s'-t) \vert
     + \vert \partial_t^2 v(t,x) (s'-t)^2 \vert \Bigr) \\& \hspace{2cm}
     \int_{d_{\el}((s',y'),(t,x))/(c\rho)}^{1} h^{-d-2} \dd h \dd t \dd x 
      \\
      &\leq C \rho^{d+1} \fint_{\B^{\el}_{c\rho}(s',y')}  \Bigl(\vert \nabla^2 v(t,x) (y'-x)^2 \vert+ \vert \partial_t \nabla v(t,x) \cdot (y'-x) (s'-t) \vert
     + \vert \partial_t^2 v(t,x) (s'-t)^2 \vert \Bigr) \\
     &\hspace{2cm} Cd_{\el}((s',y'),(t,x))^{-d-1} \dd t \dd x.
    \end{align*}
   We now use that \begin{enumerate} [label=(\roman*)]
       \item $(y'-x) \leq d_{\el}((s',y'),(t,x))$;
       \item $(s'-t) \leq \eta^{1/2} \kappa^{-1} d_{\el}((s',y'),(t,x))$;
   \end{enumerate}
    to obtain
   \begin{align*}
       (\ast) \leq 
       C \rho^{d+1} \fint_{\B^{\el}_{c\rho}(s',y')} d_{\el}((s',y'),(t,x))^{-d+1} \left( \vert \nabla^2 v(t,x) \vert + \eta^{1/2} \kappa^{-1} \vert \partial_t \nabla v(t,x) \vert + \eta \kappa^{-2} \vert \partial_t^2 v(t,x) \vert\right) \dd t \dd x.
   \end{align*}
Finally, we integrate over all $(s',y') \in B^{\el}_{\rho}(s,y)$, use Fubini and the estimate
    \begin{align*}
        \int_{B^{\el}_{\rho}(s,y)} d_{\el}((s',y'),(t,x))^{-d+1} \dd s' \dd y'  
        \leq C_d \rho^2
    \end{align*}
to obtain
    \begin{align*}
        &\left \vert \fint_{B^{\el}_{\rho}(s,y)} v_i(s',y') - v(s',y') \dd s' \dd y' \right \vert \\
        & \leq  C \rho^2 \fint_{B^{\el}_{(c+1)\rho}(s,y)} \left( \vert \nabla^2 v(t,x) \vert + \eta^{1/2} \kappa^{-1} \vert \partial_t \nabla v(t,x) \vert + \eta \kappa^{-2}\vert \partial_t^2 v(t,x) \vert\right) \dd t \dd x.        
    \end{align*}
Using that $B^{\el}_{(c+1)\rho}(y,s) \cap \good_L$ is non-empty and the definition of $\good_L$ as a sublevel set of the maximal function
gives that the right-hand side is bounded by 
    \[
    C\rho^2 \bigl(L + \eta^{1/2} \kappa^{-1} \eta^{-1/2} L^{p/2} + \eta \kappa^{-2} L^\alpha \eta^{-1}\bigr) 
    \leq C\rho^2 L.
    \]
As $v_i$ is a first-order polynomial, we have
    \[
        \fint_{B^{\el}_{\rho}(s,y)} v_i(s',y') \dd s' \dd y' = v_i(s,y),
    \]
finishing the proof of \eqref{hp:claim1} and \ref{hp1} is established (for the elliptic case). 

\medskip

\noindent \textbf{Parabolic case: $Q_i,Q_j \in \widehat{\Qu^{\pa}}$. } 
    In this case we have to separate the time-scale and the space-scale. Suppose that $v \in C^2(\R \times \R^d)$ and write $Q_i = (t_i,x_i) + [0,l_t(Q_i)] \times [0,l_x(Q_i)]^d$ (similarily for $Q_j$). Let $\rho \coloneq \inf\{l_x(Q_i),l_x(Q_j)\}$. We claim three estimates. First, we may forget about the additional time derivative in the definition of $v_i$:
    \begin{equation} \label{hp:claim2}
       \left \vert   \fint_{Q_i}  \partial_t v(t,x) \cdot (s-t) \dd t \dd x -   \fint_{Q_j}  \partial_t v(t,x) \cdot (s-t) \dd t \dd x \right \vert \leq CL l_x(Q_i)^2.  
    \end{equation}
    Second, we may show the following estimate in space (recall \eqref{def:vitilde}):
    \begin{equation} \label{hp:claim3}
        \left \vert \tilde{v}_i - \fint_{t_i}^{t_i+l_t(Q_i)} \fint_{B_\rho(y)} v(s',y') \dd y' \dd s' \right \vert  \leq C L l_x(Q_i)^2.
    \end{equation}
    Third, we compare in time, i.e.
    \begin{equation} \label{hp:claim4}
       \left \vert \fint_{t_i}^{t_i+l_t(Q_i)} \fint_{B_\rho(y)} v(s',y') \dd y' \dd s' - \fint_{t_j}^{t_j+l_t(Q_j)} \fint_{B_\rho(y)} v(s',y') \dd y' \dd s' \right \vert \leq CL l_x(Q_i)^2.
    \end{equation}
    If \eqref{hp:claim2}--\eqref{hp:claim4} are established, we may infer \ref{hp1} by using triangle inequality. It remains to prove those claims.
    
\medskip
\noindent 
\textbf{Proof of \eqref{hp:claim2}. } Observe that $\vert s-t \vert \leq C \kappa^{-2}l_x(Q_i)^2$, as we are in the parabolic regime. Therefore,
    \begin{align*}
    \left \vert \fint_{Q_i}  \partial_t v(t,x) \cdot (s-t) \dd t \dd x \right \vert &\leq C \kappa^{-2} l_x(Q_i)^2 \fint_{Q_i} \vert \partial_t v(t,x) \vert \dd t \dd x \\
    & \leq C \kappa^{-2} l_x(Q_i)^2  \fint_{2c_dQ_i} \vert \partial_t v(t,x) \vert \dd t \dd x 
    \leq C L^\alpha \kappa^{-2} l_x(Q_i)^2,
    \end{align*}
as $2c_dQ_i \cap \good_L \neq \emptyset$ and $\good_L$ is a sublevel set of the maximal function. Using $L^\alpha \kappa^{-2} = L$ and the same estimate for $Q_j$ yields \eqref{hp:claim2}.

\medskip
\noindent \textbf{Proof of \eqref{hp:claim3}. } 
We proceed as in the proof of \eqref{hp:claim1}, but only within space and not in space-time. For this purpose, fix a time $t$ and consider the space cube $\tilde{Q}_i= x_i + [0,l_x(Q_i)]^d$. We estimate the term
\begin{equation} \label{term}
    \left \vert \fint_{\tilde{Q}_i} v(t,x) + \nabla v(t,x) \cdot (x-y) \dd x - \fint_{B_{\rho}(y)} v(t,y') \dd y' \right \vert = (\heartsuit)
\end{equation}
and integrate in time afterwards. To estimate \eqref{term}, proceed as for \eqref{hp:claim1}, in particular, using that
    \[
        \fint_{\tilde{Q}_i} v(t,x) + \nabla v(t,x) \cdot (x-y) \dd x = \fint_{B_{\rho}(y)}  \fint_{\tilde{Q}_i} v(t,x) + \nabla v(t,x) \cdot (x-y') \dd x \dd y'. 
    \]
Then, for fixed $y'$ we may estimate 
\begin{align*}
    \left \vert \fint_{\tilde{Q}_i} v(t,x) + \nabla v(t,x) \cdot (x-y') \dd x - v(t,y')  \right \vert 
    &\leq
    \fint_{B_{c\rho}(y')} \vert \nabla^2 v(t,x) (y'-x)^2 \vert \int_{\vert y'-x \vert/(c\rho)}^1 h^{-d-1} \dd h \dd x \\
    & \leq 
    c \rho^{d}
    \fint_{B_{c\rho}(y')} \vert \nabla^2 v(t,x) \vert \cdot \vert x-y' \vert^{-d+2} \dd x.
\end{align*}
Integrating over $y'$ and using \schange{Fubini's theorem} and
    \[
        \int_{B_\rho(y)} \vert x- y' \vert^{-d+2} \dd x \leq C_d \rho^2
    \]
yield
    \begin{align*}
        (\heartsuit) \leq C \rho^2 \fint_{B_{(c+1) \rho}(y)} \vert \nabla^2 v(t,x) \vert \dd x.
    \end{align*}
Now, integrating in time for $t \in [t_i,t_i+l_t(Q_i)]$ yields
    \begin{align*}
         \left \vert \tilde{v}_i - \fint_{t_i}^{t_i+l_t(Q_i)} \fint_{B_\rho(y)} v(s',y') \dd y' \dd s' \right \vert 
         &\leq 
         C\rho^2 \fint_{t_i}^{t_i+l_t(Q_i)}  \fint_{B_{(c+1)\rho}(y)} \vert \nabla^2 v(t,x) \vert \dd x \dd t \\
        & \leq C\rho^2 \fint_{B^{\pa}_{(c+1)\rho}(s,y)} \vert \nabla^2 v(t,x) \vert \dd t \dd x.
    \end{align*}
    As ${B^{\pa}_{(c+1)\rho)}(s,y)} \cap \good_L \neq \emptyset$, and due to the definition of $\good_L$ as a sublevel set of the maximal function, we infer that the right-hand side of above equation is controlled by $C\rho^2L$.

\medskip
\noindent \textbf{Proof of \eqref{hp:claim4}. } This works as before, but even simpler as all integrals are one-dimensional. Indeed, one can show that, for $t' = \min\{t_i,t_j\}$ and $t''= \max\{t_i+l_t(Q_i),t_j+l_t(Q_j)\}$,
\begin{align*}    
    &
    \left \vert \fint_{t_i}^{t_i+l_t(Q_i)} \fint_{B_\rho(y)} v(s',y') \dd y' \dd s' - \fint_{t_j}^{t_j+l_t(Q_j)} \fint_{B_\rho(y)} v(s',y') \dd y' \dd s' \right \vert
    \\
    &\leq \fint_{t'}^{t''} \fint_{B_\rho(y)} \vert \partial_t v(s',y')\vert\, \vert t''-t'\vert \dd y' \dd s' \\
    &\leq 
    C \rho^2 \kappa^{-2} \fint_{B_{c\rho}^{\pa}(s,y)}\vert \partial_t v(s',y') \vert 
    \leq  
    C \rho^2 \kappa^{-2} L^\alpha \leq C \rho^2 L.
 \end{align*}
\end{proof}

\subsection{Proof of the high regularity truncation} 
\label{section:finish}
We now finish the proof of Proposition \ref{prop:hrtruncation} using the estimates derived in Lemma \ref{lemma:harrypotter} as a crucial ingredient.
\begin{proof}[Proof of Proposition \ref{prop:hrtruncation}]
Observe that $v^L$ is locally a finite sum of smooth functions such that $v^L \in C^{\infty}(\bad_L)$. Moreover, in a small neighbourhood of $\good_L$, $v^L$ coincides with (a rescaled version of) the usual Whitney extension (cf. \cite{Stein}). Therefore, we conclude $v^L \in W^{2}_{\infty}(\R \times \R^d)$. 

Moreover, Lemma \ref{lemma:ronweasley} gives the bound on the set where $v$ and $v^L$ do not coincide, i.e. \eqref{eq:hr:4}, as the truncation $v^L$ leaves $v$ untouched on the good set $\good_L$. It remains to prove the bounds Proposition \ref{prop:hrtruncation} \ref{hr:1}--\ref{hr:3}.

To this end, we may check if these bounds hold pointwise almost everywhere. Due to the definition of $\good_L$ and the observation that $\vert f \vert \leq \min\{ \M^{\pa} f, \M^{\el} f\}$ for any function $f$, we have 
    \begin{align*}
       & \Vert v^L \Vert_{L_{\infty}(\good_L)} +  \Vert \nabla v^L \Vert_{L_{\infty}(\good_L)} + \Vert \nabla^2 v^L \Vert_{L_{\infty}(\good_L)} \leq C(d)L,
       \\
       & \Vert \partial_t v^L \Vert_{L_{\infty}(\good_L)}  \leq L^\alpha, \\
       & \eta^{1/2} \Vert \partial_t \nabla v^L \Vert_{L_{\infty}(\good_L)} \leq C(d) L^{p/2} \leq C(d) L^\beta, \\
      & \eta \Vert \partial_t^2 v^L \Vert_{L_{\infty}(\good_L)} \leq C(d)L^\alpha,
    \end{align*}
so it remains to show the bounds on $\bad_L$. Recall that $L$ is chosen large enough, such that any cube $Q_i$ has sidelength $l_x(Q_i) \leq 1$, cf. Remark \ref{remark:choice2}. We consider the elliptic and the parabolic case separately: 

\medskip
\noindent \textbf{Elliptic regime: $d_{\el}((s,y),\good_L) \leq \kappa \eta^{1/2}$. } Then any cube containing $(s,y)$ is element of $\widehat{\Qu^{\el}}$. As $\sum_{i \in \N} \phi_i =1$, we have, due to Lemma \ref{lemma:harrypotter} \ref{hp0},
\begin{equation} \label{proof:hr1}
    \vert v^L(s,y) \vert \leq \sup_{i \in \N \colon \atop (s,y) \in Q_i} \vert v_i(s,y) \vert \leq CL.
\end{equation}

\medskip
\noindent \textit{First space derivative. } For the derivatives, we introduce an additional $v_j$, using that $\phi_j$ and $\phi_i$ are partitions of unity,
\begin{equation} \label{proof:hr2}
    \nabla v^L= \sum_{i \in \N} \nabla \phi_i \otimes v_i + \phi_i \nabla v_i = \sum_{i,j \in \N} \phi_j \nabla \phi_i \otimes (v_i-v_j) + \sum_{i \in \N} \phi_i \nabla v_i.
\end{equation}
The second sum can be estimated as \eqref{proof:hr1}, whereas for the first term we may estimate each summand by using \ref{Q:10} and Lemma \ref{lemma:harrypotter} \ref{hp1}, i.e.
\begin{equation} \label{proof:hr3}
   \bigl \vert \phi_j \nabla \phi_i \otimes (v_i -v_j) \bigr \vert 
   \leq 
   \Vert \nabla \phi_i \Vert_{L_\infty} \vert v_i -v_j \vert \leq C l_x(Q_i)^{-1} CL l_x(Q_i)^2 \leq CL l_x(Q_i) \leq CL.
\end{equation}
Indeed \eqref{proof:hr3} is justified, as $\phi_i$ and $\phi_j$ form a partitition of unity and the sum is locally finite.
Using that only $C(d)^2$ summands of the sum are nonzero yields
    \[
    \vert \nabla v^L(s,y) \vert \leq CL.
    \]
    
\medskip
\noindent \textit{Second space derivative. } The second derivative in space is handled in a similar fashion. First, observe that 
    \begin{equation} \label{proof:hr4}
        \nabla^2 v^L = \sum_{i,j \in \N} \phi_j\nabla^2 \phi_i \otimes (v_i-v_j) + 2\sum_{i,j \in \N} \phi_j \nabla \phi_i \otimes (\nabla v_i - \nabla  v_j).
    \end{equation}
Every summand of the second term may be estimated with Lemma \ref{lemma:harrypotter} \ref{hp2}:
    \begin{equation} \label{proof:hr5}
        \bigl \vert \phi_j \nabla \phi_i \otimes (\nabla v_i - \nabla  v_j) \bigr \vert \leq Cl_x(Q_i)^{-1} CLl_x(Q_i) \leq CL.
    \end{equation}
The first summand can again be estimated by  Lemma \ref{lemma:harrypotter} \ref{hp1} (and also \ref{Q:10}):
    \begin{equation} \label{proof:hr6}
        \bigl \vert \phi_j \nabla^2 \phi_i \otimes (v_i - v_j) \bigr \vert \leq Cl_x(Q_i)^{-2} CLl_x(Q_i)^2 \leq CL.
    \end{equation}
Again, using that the number of summands is uniformly bounded, yields
    \[
    \vert \nabla^2 v^L(s,y)  \vert \leq CL.
    \]
    
\medskip
\noindent \textit{First time derivative. } The first time derivative reads as 
    \begin{equation}\label{proof:hr65}
        \partial_t v^L = \sum_{i,j \in \N} \phi_j \partial_t \phi_i (v_i-v_j) + \sum_{i \in \N} \phi_i \partial_t v_i.
    \end{equation}
The second sum may be estimated through Lemma \ref{lemma:harrypotter} \ref{hp05}, whereas the summmands in the first sum are controlled by \ref{Q:10} and Lemma \ref{lemma:harrypotter} \ref{hp1}:
     \begin{equation} \label{proof:hr7}
        \bigl \vert \phi_j \partial_t \phi_i (v_i-v_j) \bigr \vert \leq \left(C \kappa\eta^{-1/2} l_x(Q_i)^{-1}\right) \left(CL l_x(Q_i)^2\right) \leq CL \kappa \eta^{-1/2} l_x(Q_i) \leq CL^\alpha,
     \end{equation}
as $l_x(Q_i) \leq C d_{\el}((s,y),\good_L) \leq C\eta^{1/2} \kappa$. Now \eqref{proof:hr7} and the estimate of the second sum give  
    \[
        \vert \partial_t v^L(s,y) \vert \leq CL^\alpha.
    \]
    
\medskip
\noindent \textit{Second time derivative. } The second time derivative is given by
    \begin{equation} \label{proof:hr8}
        \partial_t^2 v^L 
        =
        \sum_{i,j \in \N} \phi_j \partial_t^2 \phi_i (v_i-v_j) + 2\sum_{i,j \in \N} \phi_j \partial_t \phi_i (\partial_t v_i - \partial_t v_j).
   \end{equation}
The summands in the second sum may be bounded using Lemma \ref{lemma:harrypotter} \ref{hp3} and \ref{Q:10}:
    \begin{equation} \label{proof:hr85}
        \bigl \vert \phi_j \partial_t \phi_i (\partial_t v_i - \partial_t v_j) \bigr \vert
        \leq \left(C\kappa\eta^{-1/2} l_x(Q_i)^{-1}\right) \left( CL \kappa\eta^{-1/2} l_x(Q_i)^1\right) \leq C \kappa^2 \eta^{-1}L 
        \leq C \eta^{-1} L^\alpha
    \end{equation}
and the summands in the first sum can be handled with Lemma \ref{lemma:harrypotter} \ref{hp1} and \ref{Q:10}:
    \begin{equation} \label{proof:hr9}
        \bigl \vert \phi_j \partial_t^2 \phi_i (v_i - v_j) \bigr \vert
        \leq \left( C \eta^{-1} \kappa^{2} l_x(Q_i)^{-2} \right) \left( CL l_x(Q_i)^2\right) \leq C \eta^{-1}L \kappa^2 \leq C \eta^{-1} L^\alpha.
    \end{equation}
We finally arrive at
    \[
        \vert \partial_t^2 v^L(s,y) \vert \leq C \eta^{-1} L^\alpha.
    \]
    
\medskip

\noindent \textit{Mixed time-space derivative. }
    This derivative is given by 
        \begin{equation}\label{proof:hr10}
            \partial_t \nabla v^L = \sum_{i,j \in \N} \phi_j \nabla \partial_t \phi_i (v_i-v_j) + \sum_{i,j \in \N} \phi_j \partial_t \phi_i (\nabla v_i - \nabla v_j) + \sum_{i,j\in \N} \phi_j \nabla \phi_i \otimes (\partial_t v_i -\partial_t v_j).
        \end{equation}
Combining again Lemma \ref{lemma:harrypotter} and \ref{Q:10} we obtain
    \begin{align*}
        \bigl \vert \phi_j \nabla \partial_t \phi_i (v_i-v_j) \bigr \vert 
        \leq \left(C \kappa \eta^{-1/2}  l_x(Q_i)^{-2}\right) \left(CL l_x(Q_i)^2\right) \leq CL\kappa \eta^{-1/2}, 
        \\
        \bigl \vert \phi_j \partial_t \phi_i (\nabla v_i - \nabla v_j) \bigr \vert \leq \left(C \kappa \eta^{-1/2}  l_x(Q_i)^{-1}\right) \left(CL l_x(Q_i)^1 \right) \leq CL \kappa \eta^{-1/2}, 
        \\
        \bigl \vert \phi_j \nabla \phi_i \otimes (\partial_t v_i -\partial_t v_j)  \bigr \vert  \leq \left(C l_x(Q_i)^{-1} \right) \left(CL \kappa \eta^{-1/2} l_x(Q_i)^1\right) 
        \leq 
        CL\kappa \eta^{-1/2}.
    \end{align*}
Consequently, we obtain the bound 
    \[
    \vert \partial_t \nabla v^L(s,y) \vert \leq C \kappa\eta^{-1/2}L = C \eta^{-1/2} L^{p/2}
    \]
for $(s,y)$ in the elliptic regime.

\medskip
\noindent \textbf{Parabolic regime: $d_{\pa}((s,y),\good_L) \geq \kappa \eta^{1/2}$. } Then any cube containing $(s,y)$ is an element of $\widehat{\Qu^{\pa}}$. The estimates for $v^L$, $\nabla v^L$ and $\nabla^2 v^L$ work with the same bounds displayed in \eqref{proof:hr1}--\eqref{proof:hr5}. Thus, we only estimate time derivatives.

\medskip
\noindent \textit{First time derivative. }
Again, we write the time derivative $\partial_t v^L$ as in \eqref{proof:hr5} and estimate the second summand with Lemma \ref{lemma:harrypotter} \ref{hp0}. The first summand is estimated by Lemma \ref{lemma:harrypotter} and \ref{Q:9}:
 \begin{equation*}
        \bigl \vert \phi_j \partial_t \phi_i (v_i-v_j) \bigr \vert \leq \left( C \kappa^2 l_x(Q_i)^{-2}\right) \bigl(CL l_x(Q_i)^2\bigr) 
        \leq CL\kappa^2,
     \end{equation*}
so that we (again using that only $C(d)^2$ summands are nonzero) obtain
    \[
        \vert \partial_t v^L(s,y) \vert \leq CL^\alpha.
    \]
    
\medskip

\noindent \textit{Second time derivative. }
Write the second time derivative as in \eqref{proof:hr7} and argue parallel to \eqref{proof:hr8} and \eqref{proof:hr9}. First, using Lemma \ref{lemma:harrypotter} \ref{hp4} and \ref{Q:9}, we obtain
    \[
        \bigl \vert \phi_j \partial_t \phi_i (\partial_t v_i - \partial_t v_j) \bigr \vert \leq \left(C\kappa^2 l_x(Q_i)^{-2}\right) \bigl(CL\kappa^2\bigr)\leq CL \kappa^4 l_x(Q_i)^{-2} \leq C \eta^{-1}L^\alpha,
    \]
as $l_x(Q_i) \geq d_{\pa}((y,s),\good_L) \geq \kappa \eta^{1/2}$. Second, using Lemma \ref{lemma:harrypotter} \ref{hp1} and \ref{Q:9}, we obtain
    \[
        \bigl \vert \phi_j \partial_t^2 \phi_i (v_i - v_j) \bigr \vert
        \leq \left(C \kappa^4 l_x(Q_i)^{-4}\right) \left(CL l_x(Q_i)^2\right) \leq CL\kappa^4 l(Q_i)^{-2} \leq C\eta^{-1}L^{\alpha}
    \]
and thus,
    \[
    \vert \partial_t^2v^L(s,y) \vert \leq C\eta^{-1}L^\alpha.
    \]

\medskip
\noindent \textit{Mixed time-space derivative. }
Finally, using \eqref{proof:hr10} and Lemma \ref{lemma:harrypotter}, the bounds on the derivatives, \ref{Q:9}, and that $(s,y)$ is in the parabolic region, we obtain
  \begin{align*}
         \bigl \vert \phi_j \nabla \partial_t \phi_i (v_i-v_j) \bigr \vert 
         \leq \left(C \kappa l_x(Q_i)^{-3}\right) \left( CL l_x(Q_i)^2 \right) \leq CL\kappa l_x(Q_i)^{-1} \leq CL\kappa^2 \eta^{-1/2}, \\
          \bigl \vert \phi_j \partial_t \phi_i (\nabla v_i - \nabla v_j) \bigr \vert 
          \leq \left(C \kappa l_x(Q_i)^{-2}\right) \left(CL l_x(Q_i)^1\right) \leq  CL\kappa  l_x(Q_i)^{-1} \leq CL \kappa^2 \eta^{-1/2}, \\
           \bigl \vert \phi_j \nabla \phi_i \otimes (\partial_t v_i -\partial_t v_j)  \bigr \vert  \leq \left(C l_x(Q_i)^{-1} \right)\left(CL \kappa^2 \right)\leq CL\kappa^2 l_x(Q_i)^{-1} \leq CL\kappa^3 \eta^{-1/2}.
    \end{align*}
Consequently, we arrive at the desired bound 
    \[
    \vert \partial_t \nabla v^L(y,s) \vert \leq CL^\beta. 
    \]
This finishes the proof of \ref{hr:1}--\ref{hr:3}.    
\end{proof}


\subsection{From higher regularity back to $w$} \label{section:sol}
As a consequence of Proposition \ref{prop:hrtruncation}, we may finally prove Lemma \ref{lemma:trunc:v2}.

First, we realise that Proposition \ref{prop:hrtruncation} is also valid if the function $v \colon \R \times \R^d \to \R^{d \times d}_{\skw}$ is replaced by a periodic function $v \colon \R \times \T_d \to \R^{d\times d}_{\skw}$, as everything (in particular the cover of the bad set and the partition of unity) can be defined periodically. 

Second, we note that it suffices to localise in time by multiplying with an in-time cutoff that is equal to $1$ in $[-T,T]$ and vanishes outside of $[-2T,2T]$. Indeed, if we only consider radii smaller than one in the definition of the maximal function, it is clear that this does not affect the construction inside the interval $[-T/2,T/2]$.

The task now is to find such a $v$. For this purpose, consider the differential operator $\curl^{\ast} \colon C^{\infty}(\R^d;\R^{d\times d}_{\skw}) \to C^{\infty}(\R^d;\R^d)$ given by 
\[
    (\curl^{\ast} u)_i = \sum_{j=1}^d \partial_j u_{ij}.
\]
Then $\curl^{\ast}$ is a potential for $\diverg$, that is 
\begin{enumerate} [label=(\roman*)]
    \item $\diverg \circ \curl^{\ast} =0$;
    \item the following Proposition \ref{prop:pot} holds (cf. \cite{SW,GR}).
\end{enumerate}

\begin{proposition} \label{prop:pot}
    Let $\mu \in \R$, $1 <p<\infty$. There is a linear and bounded map $T \colon W^{\mu}_{p}(\T_d;\R^d) \to W^{\mu+1}_{p}(\T_d;\R^{d \times d}_{\skw})$ such that 
    \begin{equation*}
        \diverg w =0 \quad \text{and} \quad \int_{\T_d} w =0 \dd x \quad \Longrightarrow \quad w= \curl^{\ast} (Tw).
    \end{equation*}
    Moreover, the following estimate holds for any $w \in W^{\mu}_{p}(\T_d;\R^d)$:
    \begin{equation*}
        \Vert \curl^{\ast} Tw - w \Vert_{W^{\mu}_{p}} \leq  \Vert \diverg w \Vert_{W^{\mu-1}_{p}}.
    \end{equation*}
\end{proposition}
We shortly remark that writing $\int_{\T_d} w \dd x=0$
is also sensible for negative Sobolev spaces, as it is the dual pairing of $w$ with the (smooth) constant $1$-function. 
The map $T$ may be defined through a Fourier multiplier, i.e. independently of $\mu$ and $p$. Given some $w_\eta$ satisfying the assumptions \ref{P:1}--\ref{P:4}, we define $v_\eta$ as
    \begin{equation} \label{def:pot}
        v_\eta \coloneq T w_\eta.
    \end{equation}
Due to linearity of the operator $T$ and Proposition \ref{prop:pot} $v_\eta$ obeys the properties \ref{Q:1}--\ref{Q:4}.

\begin{lemma}
    Suppose that $w_\eta$ satisfies \ref{P:1}--\ref{P:4}. Then $v_\eta$ satisfies the following:
    \begin{enumerate}[label=(\roman*)]
        \item $\Vert v_\eta \Vert_{L_p((0,\infty);W^2_p)} \leq C(p) \Vert w_\eta \Vert_{L_p((0,\infty);V^1_p)} \leq C$;
        \item $\partial_t v_\eta = T\hat{g}_{\eta} + T\hat{h}_\eta$ with $\hat{g}_\eta \weakto 0$ in $L_q((0,\infty);(V^1_p)')$ and $\hat{h}_\eta \to 0$ in $L_{s'}((0,\infty);(V^1_s)')$;
        \item $\Vert \partial_t v_\eta \Vert_{L_2((0,\infty);W^1_2)} \leq \Vert \partial_t w_\eta \Vert_{L_2((0,\infty);L_2)} \leq C \eta^{-1/2} $;
        \item $\eta \partial_t v_\eta = T\hat{\varg}_\eta + T \hat{\varh}_\eta$ with $\hat{\varg}_\eta \weakto 0$ in
        $L_q((0,\infty);(V^1_p)')$ and $\hat{\varh}_\eta \to 0$ in $L_{s'}((0,\infty);(V^1_s)')$.
    \end{enumerate} 
\end{lemma}

In particular, as $v_\eta$ satisfies \ref{Q:1}--\ref{Q:4} (to be precise: the version on the torus and not on the full space), we can now truncate $v_\eta$ to obtain $v_\eta^L$. We then finish the proof of Lemma \ref{lemma:trunc:v2} by setting 
    \begin{equation} \label{def:wetaL}
        w_\eta^L = \curl^{\ast} v_\eta^L.
    \end{equation}
As a final preparation, we need the following observation that is slightly stronger than in Proposition \ref{prop:hrtruncation}.
\begin{lemma} \label{lemma:better}
   Let $\bad_L^{\eta}$ be the bad set defined in \eqref{def:badset} for the function $v_\eta$. Then we can write $\bad_L^{\eta}= \bad_L^{\eta,1} \cup \bad_L^{\eta,2}$ such that:
   \begin{enumerate} [label=(\roman*)]
       \item $L^{p}\mathcal{L}^{d+1}(\bad_L^{\eta,1}) \to 0$ uniformly in $\eta$, as $L \to \infty$;
       \item $\mathcal{L}^{d+1}(\bad_L^{\eta,2}) \to 0$ for fixed $L>0$, as $\eta \to 0$.
   \end{enumerate}
\end{lemma}
\begin{proof}
    It suffices to show that we can split up the set, where the \emph{parabolic} maximal function is large; the same can be done for the elliptic one.
    We are given the bad set $\bad_L^{\pa,\eta}$ as the union of four different sets. Observe as $\barh_\eta$ and $\barvarh_\eta$ converge to zero in $L_{s'}((0,T))$, hence
    \[
    \mathcal{L}^{d+1}(\bad_L^{\pa,4}) \to 0 \quad \text{as } \eta \to 0.
    \]
    For the remainder we use the following result: if 
    $f_\eta \in L_r(\R \times \R^d)$ we have (for $\M= \M^{\pa}$ or $\M=\M^{\el}$)
    \[
     \{ \M f_\eta \geq \lambda \} = A^{\eta,1} \cup A^{\eta,2},
    \]  
    with $\lambda^{r} \mathcal{L}^{d+1}(A^{\eta,1}) \to 0$ uniformly in $\eta >0$, as $\lambda \to \infty$ and $\mathcal{L}^{d+1}(A^{\eta,1}) \to 0$, as $\eta \to 0$.
    To this end, recall that we can divide $f_\eta$ into an $r$-equi-integrable part $f_\eta^{\equi}$ and a concentrating part $f_\eta^{\conc}$. Then, due to sublinearity,
    \[
        \{ \M f_\eta \geq \lambda \} \subset \{ \M f_\eta^{\equi} \geq \lambda/2 \} \cup  \{ \M f_\eta^{\conc} \geq \lambda/2 \}. 
    \]  
    On the one hand, as established in \eqref{eq:ronweasley}, we have
         \begin{equation*} 
            \lambda^r \mathcal{L}^{d+1}( \{ \M f_\eta^{\equi} \geq \lambda\}) 
            \leq 
            C_d \int_{\{\vert f_\eta^{\equi} \vert \geq \lambda/2\}}\vert f_\eta \vert^r \dd t \dd x.
        \end{equation*}
    The right-hand side of this equation converges uniformly to zero, as $L \to \infty$ since $f_\eta^{\equi}$ is $r$-equi-integrable. On the other hand, $f_\eta^{\conc} \to 0$ in $L_{\tilde{r}}(\R \times \R^d)$ for some $\tilde{r}<r$, and therefore 
        \begin{equation*}
            \mathcal{L}^{d+1}( \{ \M f_\eta^{\conc} \geq \lambda\}) \longrightarrow 0, 
            \quad \text{as} \quad
            \eta \to 0.
        \end{equation*}
    Choosing $A^{\eta,1} = \{ \M f_\eta^{\equi} \geq \lambda/2 \} \cap \{ \M f_\eta >\lambda \}$ and $A^{\eta,2} = \{ \M f_\eta^{\conc} \geq \lambda/2 \} \cap \{ \M f_\eta >\lambda \}$ proves the claim.
    Now applying this abstract result for
    \begin{enumerate} [label=(\roman*)]
        \item $f=v$, $f= \nabla v$, $f=\nabla^2 v$, $r=p$ and $\lambda=2L$;
        \item $f=\partial_t \nabla v$, $r=2$ and $\lambda=2L^{p/2}$;
        \item $f= \barg$, $f=\barvarg$, $r=q$ and $\lambda=2L^\alpha$
    \end{enumerate}
    and splitting the different sublevel sets according to this claim yields the statement of the lemma.   
\end{proof}
We are now ready to give a proof of the truncation lemma.

\begin{proof}[Proof of Lemma \ref{lemma:trunc:v2}]
By construction we directly have $\diverg w_\eta^L=0$, as $\diverg \circ \curl^{\ast}=0$. Moreover, as $v_\eta^L \in L_{\infty}((0,\infty);V^1_\infty) \cap W^1_{\infty}((0,\infty);L_{\infty})$ with both norms bounded by $CL$. This is \ref{hr:1} in Proposition \ref{prop:hrtruncation} and \ref{trunc1:1} and \ref{trunc2:1} directly follow. The same argument works for \ref{trunc5:1}, as $\eta^{1/2} \partial_t v_\eta^L$ is uniformly bounded in $L_2((0,\infty);V^1_2)$. It remains to verify \ref{trunc3:1} and \ref{trunc4:1}.
We only show \ref{trunc3:1} and \ref{trunc4:1}, \ref{trunc45:1} follows in the exact same fashion as \ref{trunc4:1}.
\smallskip

We can subdivide $\nabla^2 v_\eta$ into a concentrating and a $p$-equi-integrable part. Moreover, $\nabla^2 v_\eta -\nabla^2 v_\eta^L$ does not vanish on $\bad_L^\eta$. For $\bad_L^{\eta,1}$ and $\bad_L^{\eta,2}$ as in the previous lemma we define
\[
    \tilde{G}_\eta^L 
    \coloneq 
    1_{\bad_L^{\eta,1}} \cdot \left((\nabla^2 v_\eta)^{\equi} -\nabla^2 v_\eta^L\right)
    \quad \text{and} \quad
    \tilde{H}_\eta^L 
    \coloneq
    1_{\bad_L^{\eta,2}} \cdot \left(\nabla^2 v_\eta -\nabla^2 v_\eta^L\right) + 1_{\bad_L^{\eta,1}} (\nabla^2 v_\eta)^{\conc}.
\]
Then $\nabla^2 v_\eta- \nabla^2 v_\eta^L= \tilde{G}_\eta^L + \tilde{H}_\eta^L$. A straightforward calculation (cf. proof of \ref{coro:finale}) gives 
\begin{align*}
    \lim_{L \to \infty} \sup_{\eta>0} \Vert \tilde{G}^L_\eta \Vert_{L_p((0,\infty);L_p)} =0
    \quad \text{and} \quad
    \lim_{\eta \to 0} \sup_{L>0} \Vert \tilde{H}^L_\eta \Vert_{L_{s'}((0,T);L_{s'})} =0 \quad \text{for some } s'<\infty.
\end{align*}
Now applying \cite[Lemma 2.14]{FM} for any given time $t$ gives sequences $\hat{G}_\eta^L$ and $\hat{H}_\eta^L$ such that 
\begin{align*}
    \nabla^2 \left( \hat{G}_\eta^L + \hat{H}_\eta^L \right) = \tilde{G}_\eta^L + \tilde{H}_\eta^L
\end{align*}
with 
\begin{equation*}
    \lim_{L \to \infty} \sup_{\eta>0} \Vert \hat{G}^L_\eta \Vert_{L_p((0,\infty);W^2_p)} =0
    \quad \text{and} \quad
    \lim_{\eta \to 0} \sup_{L>0} \Vert \hat{H}^L_\eta \Vert_{L_{s'}((0,T);W^2_{s'})} =0 \quad \text{for some } s'<\infty.
\end{equation*}
Finally, defining $G_\eta^L \coloneq \curl^{\ast} \hat{G}^L_\eta$ and $H_\eta^L \coloneq \curl^{\ast} \hat{H}^L_\eta$, yields \eqref{GLeta_v2} and \eqref{HLeta_v2}, i.e. the first part of property \ref{trunc3:1}.

The second part, i.e. the decomposition $w_\eta^L= w_{\eta,1}^L + w_{\eta,2}^L$, follows similarly. First of all, $w_\eta^L = w_\eta + (w_\eta^L-w_\eta)$ and $w_\eta$ is uniformly bounded in $L^p((0,T);V^1_p)$. Observe that by construction one may show
\begin{align*}
    \lim_{L \to \infty} \sup_{\eta>0} \Vert \tilde{G}^L_\eta \Vert_{L_p((0,\infty);L_p)} =0
    \quad \text{and} \quad
    \lim_{\eta \to 0}\Vert \tilde{H}^L_\eta \Vert_{L_{p}((0,T);L_{p})} =0 \quad \text{for fixed } L>0.
\end{align*}
Hence, arguing as before with \cite[Lemma 2.14]{FM} we can decompose $w_\eta^L-w_\eta$ accordingly into $(G_\eta^L)'$ and $(H_\eta^L)'$ such that 
\begin{equation*}
    \lim_{L \to \infty} \sup_{\eta>0} \Vert (G_\eta^L)' \Vert_{L_p((0,\infty);V^1_p)} =0
    \quad \text{and} \quad
    \lim_{\eta \to 0} \sup_{L>0} \Vert (H_\eta^L)' \Vert_{L_{p}((0,T);V^1_{p})} =0 \quad \text{for some } s'<\infty.
\end{equation*}
Now defining $w_{\eta,1}^L \coloneq w_\eta + (G_\eta^L)'$ and $w_{\eta,2}^L \coloneq (H_\eta^L)'$ yields the result.
\smallskip

For the time derivative, we need to additionally use the splitting already inferred for $\partial_t v_\eta$. In particular, write
\[
\partial_t v_\eta - \partial_t v_\eta^L = \bar{g}_\eta +\bar{h}_\eta - \partial_t v_\eta^L,
\]
split $\hat{g}_\eta$ into an equi-integrable and a concentrating part and write:
\begin{align*}
    \tilde{g}_\eta^L &\coloneq 1_{\bad^{\eta,1}_L} \cdot \left(\bar{g}_\eta^{\equi} -\partial_t v_\eta^L \right) \\
     \tilde{h}_\eta^L &\coloneq 1_{\bad^{\eta,2}_L} \cdot (\partial_t v_\eta - \partial_t v_\eta^L ) + 1_{\bad^{\eta,1}_L} \cdot (\barh_\eta + g_\eta^{\conc}).
\end{align*}
Again, a calculation gives
\begin{align*}
    \lim_{L \to \infty} \sup_{\eta>0} \Vert \tilde{g}^L_\eta \Vert_{L_q((0,\infty);L_q)} =0
    \quad \text{and} \quad
    \lim_{\eta \to 0} \sup_{L>0} \Vert \tilde{h}^L_\eta \Vert_{L_{s'}((0,T);L_{s'})} =0.
\end{align*}
Then defining $g_\eta^L \coloneq \curl^{\ast} \tilde{g}_\eta^L$ and $h_\eta^L \coloneq \curl^{\ast} \tilde{h}_\eta^L$ yields \ref{trunc4:1}.
\end{proof}

\bibliographystyle{abbrv}
\bibliography{literature}

\begin{thebibliography}{10}

\bibitem{AF84}
E.~Acerbi and N.~Fusco.
\newblock Semicontinuity problems in the calculus of variations.
\newblock {\em Arch. Rational Mech. Anal.}, 86(2):125--145, 1984.

\bibitem{OldTestament2}
H.~Amann.
\newblock {\em Linear and quasilinear parabolic problems. {V}ol. {I}}, volume~89 of {\em Monographs in Mathematics}.
\newblock Birkh\"{a}user/Springer, Cham, 1995.
\newblock Abstract Linear Theory.

\bibitem{Bibel}
H.~Amann.
\newblock {\em Linear and quasilinear parabolic problems. {V}ol. {II}}, volume 106 of {\em Monographs in Mathematics}.
\newblock Birkh\"{a}user/Springer, Cham, 2019.
\newblock Function spaces.

\bibitem{BS22}
M.~Bathory and U.~Stefanelli.
\newblock Variational resolution of outflow boundary conditions for incompressible {N}avier-{S}tokes.
\newblock {\em Nonlinearity}, 35(11):5553--5592, 2022.

\bibitem{BC20}
L.~C. Berselli and E.~Chiodaroli.
\newblock On the energy equality for the 3{D} {N}avier-{S}tokes equations.
\newblock {\em Nonlinear Anal.}, 192:111704, 24, 2020.

\bibitem{BFS}
L.~C. Berselli, S.~Fagioli, and S.~Spirito.
\newblock Suitable weak solutions of the {N}avier-{S}tokes equations constructed by a space-time numerical discretization.
\newblock {\em J. Math. Pures Appl. (9)}, 125:189--208, 2019.

\bibitem{BDF}
D.~Breit, L.~Diening, and M.~Fuchs.
\newblock Solenoidal {L}ipschitz truncation and applications in fluid mechanics.
\newblock {\em J. Differential Equations}, 253(6):1910--1942, 2012.

\bibitem{BDS}
D.~Breit, L.~Diening, and S.~Schwarzacher.
\newblock Solenoidal {L}ipschitz truncation for parabolic {PDE}s.
\newblock {\em Math. Models Methods Appl. Sci.}, 23(14):2671--2700, 2013.

\bibitem{BV19b}
T.~Buckmaster and V.~Vicol.
\newblock Convex integration and phenomenologies in turbulence.
\newblock {\em EMS Surv. Math. Sci.}, 6(1-2):173--263, 2019.

\bibitem{BV19}
T.~Buckmaster and V.~Vicol.
\newblock Nonuniqueness of weak solutions to the {N}avier-{S}tokes equation.
\newblock {\em Ann. of Math. (2)}, 189(1):101--144, 2019.

\bibitem{Bulicek2}
M.~Bul\'{\i}\v{c}ek, J.~Burczak, and S.~Schwarzacher.
\newblock A unified theory for some non-{N}ewtonian fluids under singular forcing.
\newblock {\em SIAM J. Math. Anal.}, 48(6):4241--4267, 2016.

\bibitem{Bulicek3}
M.~Bul\'{\i}\v{c}ek, L.~Diening, and S.~Schwarzacher.
\newblock Existence, uniqueness and optimal regularity results for very weak solutions to nonlinear elliptic systems.
\newblock {\em Anal. PDE}, 9(5):1115--1151, 2016.

\bibitem{Bulicek}
M.~Bul\'{\i}\v{c}ek, F.~Ettwein, P.~Kaplick\'{y}, and D.~Pra\v{z}\'{a}k.
\newblock On uniqueness and time regularity of flows of power-law like non-{N}ewtonian fluids.
\newblock {\em Math. Methods Appl. Sci.}, 33(16):1995--2010, 2010.

\bibitem{BMS}
J.~Burczak, S.~Modena, and L.~Sz\'{e}kelyhidi.
\newblock Non uniqueness of power-law flows.
\newblock {\em Comm. Math. Phys.}, 388(1):199--243, 2021.

\bibitem{DL09}
C.~De~Lellis and L.~Sz\'{e}kelyhidi, Jr.
\newblock The {E}uler equations as a differential inclusion.
\newblock {\em Ann. of Math. (2)}, 170(3):1417--1436, 2009.

\bibitem{DL10}
C.~De~Lellis and L.~Sz\'{e}kelyhidi, Jr.
\newblock On admissibility criteria for weak solutions of the {E}uler equations.
\newblock {\em Arch. Ration. Mech. Anal.}, 195(1):225--260, 2010.

\bibitem{DL22}
C.~De~Lellis and L.~Sz\'{e}kelyhidi, Jr.
\newblock Weak stability and closure in turbulence.
\newblock {\em Philos. Trans. Roy. Soc. A}, 380(2218):Paper No. 20210091, 16, 2022.

\bibitem{DKS}
L.~Diening, C.~Kreuzer, and E.~S\"{u}li.
\newblock Finite element approximation of steady flows of incompressible fluids with implicit power-law-like rheology.
\newblock {\em SIAM J. Numer. Anal.}, 51(2):984--1015, 2013.

\bibitem{DMS}
L.~Diening, J.~M\'{a}lek, and M.~Steinhauer.
\newblock On {L}ipschitz truncations of {S}obolev functions (with variable exponent) and their selected applications.
\newblock {\em ESAIM Control Optim. Calc. Var.}, 14(2):211--232, 2008.

\bibitem{DRW}
L.~Diening, M.~R{u}\v{z}i\v{c}ka, and J.~Wolf.
\newblock Existence of weak solutions for unsteady motions of generalized {N}ewtonian fluids.
\newblock {\em Ann. Sc. Norm. Super. Pisa Cl. Sci. (5)}, 9(1):1--46, 2010.

\bibitem{DSSV}
L.~Diening, S.~Schwarzacher, B.~Stroffolini, and A.~Verde.
\newblock Parabolic {L}ipschitz truncation and caloric approximation.
\newblock {\em Calc. Var. Partial Differential Equations}, 56(4):Paper No. 120, 27, 2017.

\bibitem{FM}
I.~Fonseca and S.~M\"{u}ller.
\newblock {$\mathscr A$}-quasiconvexity, lower semicontinuity, and {Y}oung measures.
\newblock {\em SIAM J. Math. Anal.}, 30(6):1355--1390, 1999.

\bibitem{FMP}
I.~Fonseca, S.~M\"{u}ller, and P.~Pedregal.
\newblock Analysis of concentration and oscillation effects generated by gradients.
\newblock {\em SIAM J. Math. Anal.}, 29(3):736--756, 1998.

\bibitem{FMS}
J.~Frehse, J.~M\'{a}lek, and M.~Steinhauer.
\newblock On analysis of steady flows of fluids with shear-dependent viscosity based on the {L}ipschitz truncation method.
\newblock {\em SIAM J. Math. Anal.}, 34(5):1064--1083, 2003.

\bibitem{FJM}
G.~Friesecke, R.~D. James, and S.~M\"uller.
\newblock A theorem on geometric rigidity and the derivation of nonlinear plate theory from three-dimensional elasticity.
\newblock {\em Comm. Pure Appl. Math.}, 55(11):1461--1506, 2002.

\bibitem{Gigli}
N.~Gigli and S.~J.~N. Mosconi.
\newblock A variational approach to the {N}avier-{S}tokes equations.
\newblock {\em Bull. Sci. Math.}, 136(3):256--276, 2012.

\bibitem{GR}
A.~Guerra and B.~Rai\c{t}\u{a}.
\newblock On the necessity of the constant rank condition for {$L^p$} estimates.
\newblock {\em C. R. Math. Acad. Sci. Paris}, 358(9-10):1091--1095, 2020.

\bibitem{KL00}
J.~Kinnunen and J.~L. Lewis.
\newblock Higher integrability for parabolic systems of {$p$}-{L}aplacian type.
\newblock {\em Duke Math. J.}, 102(2):253--271, 2000.

\bibitem{Lady1}
O.~A. Lady\v{z}enskaja.
\newblock New equations for the description of the motions of viscous incompressible fluids, and global solvability for their boundary value problems.
\newblock {\em Trudy Mat. Inst. Steklov.}, 102:85--104, 1967.

\bibitem{Lady2}
O.~A. Lady\v{z}enskaja.
\newblock Modifications of the {N}avier-{S}tokes equations for large gradients of the velocities.
\newblock {\em Zap. Nau\v{c}n. Sem. Leningrad. Otdel. Mat. Inst. Steklov. (LOMI)}, 7:126--154, 1968.

\bibitem{Lady4}
O.~A. Ladyzhenskaya.
\newblock On some problems from the theory of continuous media.
\newblock In {\em Proceedings of International Congress of Mathematicians}, 1966.

\bibitem{Lady3}
O.~A. Ladyzhenskaya.
\newblock {\em The mathematical theory of viscous incompressible flow}.
\newblock Mathematics and its Applications, Vol. 2. Gordon and Breach Science Publishers, New York-London-Paris, 1969.
\newblock Second English edition, revised and enlarged, Translated from the Russian by Richard A. Silverman and John Chu.

\bibitem{Leray}
J.~Leray.
\newblock Sur le mouvement d'un liquide visqueux emplissant l'espace.
\newblock {\em Acta Math.}, 63(1):193--248, 1934.

\bibitem{Lionsbook}
J.-L. Lions.
\newblock {\em Quelques m\'{e}thodes de r\'{e}solution des probl\`emes aux limites non lin\'{e}aires}.
\newblock Dunod, Paris; Gauthier-Villars, Paris, 1969.

\bibitem{Liu}
F.~C. Liu.
\newblock A {L}uzin type property of {S}obolev functions.
\newblock {\em Indiana Univ. Math. J.}, 26(4):645--651, 1977.

\bibitem{MNRR}
J.~M\'{a}lek, J.~Ne\v{c}as, M.~Rokyta, and M.~Ru\v{z}i\v{c}ka.
\newblock {\em Weak and measure-valued solutions to evolutionary {PDE}s}, volume~13 of {\em Applied Mathematics and Mathematical Computation}.
\newblock Chapman \& Hall, London, 1996.

\bibitem{MNR93}
J.~M\'{a}lek, J.~Ne\v{c}as, and M.~Ru\v{z}i\v{c}ka.
\newblock On the non-{N}ewtonian incompressible fluids.
\newblock {\em Math. Models Methods Appl. Sci.}, 3(1):35--63, 1993.

\bibitem{MN01}
J.~M\'{a}lek, J.~Ne\v{c}as, and M.~Ru\v{z}i\v{c}ka.
\newblock On weak solutions to a class of non-{N}ewtonian incompressible fluids in bounded three-dimensional domains: the case {$p\geq2$}.
\newblock {\em Adv. Differential Equations}, 6(3):257--302, 2001.

\bibitem{Malek1}
J.~M\'{a}lek, D.~Pra\v{z}\'{a}k, and M.~Steinhauer.
\newblock On the existence and regularity of solutions for degenerate power-law fluids.
\newblock {\em Differential Integral Equations}, 19(4):449--462, 2006.

\bibitem{Malek2}
J.~M\'{a}lek, K.~R. Rajagopal, and M.~Ru\v{z}i\v{c}ka.
\newblock Existence and regularity of solutions and the stability of the rest state for fluids with shear dependent viscosity.
\newblock {\em Math. Models Methods Appl. Sci.}, 5(6):789--812, 1995.

\bibitem{MO08}
A.~Mielke and M.~Ortiz.
\newblock A class of minimum principles for characterizing the trajectories and the relaxation of dissipative systems.
\newblock {\em ESAIM Control Optim. Calc. Var.}, 14(3):494--516, 2008.

\bibitem{OSS}
M.~Ortiz, B.~Schmidt, and U.~Stefanelli.
\newblock A variational approach to {N}avier-{S}tokes.
\newblock {\em Nonlinearity}, 31(12):5664--5682, 2018.

\bibitem{RRS}
J.~C. Robinson, J.~L. Rodrigo, and W.~Sadowski.
\newblock {\em The three-dimensional {N}avier-{S}tokes equations}, volume 157 of {\em Cambridge Studies in Advanced Mathematics}.
\newblock Cambridge University Press, Cambridge, 2016.
\newblock Classical theory.

\bibitem{SW}
J.~R. Schulenberger and C.~H. Wilcox.
\newblock Coerciveness inequalities for nonelliptic systems of partial differential equations.
\newblock {\em Ann. Mat. Pura Appl. (4)}, 88:229--305, 1971.

\bibitem{Stefanelli}
U.~Stefanelli.
\newblock The weighted inertia-energy-dissipation principle.
\newblock {\em arXiv:2407.20933}, 2024.

\bibitem{Stein}
E.~M. Stein.
\newblock {\em Singular integrals and differentiability properties of functions}.
\newblock Princeton Mathematical Series, No. 30. Princeton University Press, Princeton, N.J., 1970.

\bibitem{Temam}
R.~Temam.
\newblock {\em Navier-{S}tokes equations}, volume~2 of {\em Studies in Mathematics and its Applications}.
\newblock North-Holland Publishing Co., Amsterdam-New York, revised edition, 1979.
\newblock Theory and numerical analysis, With an appendix by F. Thomasset.

\bibitem{WS}
D.~E. Weidner and L.~W. Schwartz.
\newblock Contact-line motion of shear-thinning liquids.
\newblock {\em Physics of Fluids}, 6:3535--3538, 1994.

\bibitem{Zhang}
K.~Zhang.
\newblock A construction of quasiconvex functions with linear growth at infinity.
\newblock {\em Ann. Scuola Norm. Sup. Pisa Cl. Sci. (4)}, 19(3):313--326, 1992.

\end{thebibliography}

\end{document}